\documentclass{article}[11pt]
\usepackage{graphicx}
\usepackage{amsfonts}
\usepackage{amsmath}
\usepackage{amssymb}
\usepackage{amsthm}
\usepackage{stmaryrd}
\usepackage{hyperref}
\usepackage{cite}
\usepackage[all]{xy}
\usepackage[utf8]{inputenc}
\usepackage{tikz-cd}
\usepackage{tikz}

\usepackage[letterpaper, margin=1in]{geometry}

\theoremstyle{plain}
\newtheorem{thm}{Theorem}[section]
\newtheorem{prop}[thm]{Proposition}
\newtheorem{corr}[thm]{Corollary}
\newtheorem{lem}[thm]{Lemma}

\theoremstyle{remark}
\newtheorem{rem}[thm]{Remark}
\newtheorem{quest}{Question}

\theoremstyle{definition}
\newtheorem{defn}[thm]{Definition}
\newtheorem{ex}[thm]{Example}

\newcommand{\inv}{^{-1}}
\newcommand{\units}{^{(0)}}
\newcommand{\ra}{\rightarrow}

\newcommand{\calA}{\mathcal A}
\newcommand{\calG}{\mathcal G}
\newcommand{\calH}{\mathcal H}
\newcommand{\calK}{\mathcal K}
\newcommand{\calL}{\mathcal L}

\newcommand{\ca}{\mathcal}

\newcommand{\cst}{$C^*$}

\newcommand{\kk}{\mathrm{KK}}
\newcommand{\ind}{\mathrm{Ind}}
\newcommand{\infl}{\mathrm{Inf}}

\newcommand{\lgl}{\langle}
\newcommand{\rgl}{\rangle}
\newcommand{\Hom}{\mathrm{Hom}}
\newcommand{\End}{\mathrm{End}}
\newcommand{\supp}{\mathrm{supp}}

\newcommand{\inter}{\mathrm{int}}
\newcommand{\cl}{\mathrm{cl}}
\newcommand{\lch}{locally compact Hausdorff }

\title{Roe algebras and coarse index maps for spaces with proper actions of \'etale groupoids. I}
\author{Kai Mao}
\date{}

\begin{document}

\maketitle

\begin{abstract}
    This is the first in a series of papers extending Roe algebras and coarse index theory to the groupoid-equivariant setting. We introduce some techniques to develop a framework of Roe algebras and their K-theory for spaces equipped with proper actions of \'etale groupoids. For a fixed \'etale groupoid $\mathcal G$ and $\mathcal G$-$C^*$-algebra $A$, we construct a functor $KC_*(-;\mathcal G,A)$ from the category of locally compact Hausdorff proper $\mathcal G$-spaces and equivariant proper continuous maps to the category of graded abelian groups, which provides a natural receptacle for an equivariant coarse index map. We study the existence of universal modules, an analog of ample modules in non-equivariant setting, and develop a decomposition technique to establish their existence. As an application, we prove that groupoid simplicial complexes satisfying suitable hypotheses admit universal modules.
\end{abstract}

\tableofcontents

\section{Introduction}

The aim of this series of papers is to generalize the framework of Roe algebras and coarse index maps from the non-equivariant and group action settings (see \cite{Roe96}, \cite{higson2000analytic}, \cite{willett2020higher}, etc.) to the groupoid action setting. This introduction is divided into several parts. In each part, we first recall and deconstruct the classical theory of Roe algebras for proper metric spaces and then present its analog for spaces equipped with proper actions of \'etale groupoids. Let $X$ be a second countable proper metric space, $\calG$ be a locally compact Hausdorff \'etale groupoid, $A$ be a $\calG$-\cst-algebra and $Z, Z_1,Z_2, Z_3$ be \lch proper $\calG$-spaces.

\subsection*{Roe algebras and K-theory}

Roe algebras arose from the study of index theory on noncompact manifolds. Elliptic operators on such manifolds need not be Fredholm. For suitable operators, notably Dirac-type operators, one can nevertheless define coarse index classes in the \(K\)-theory of Roe algebras, thereby retaining information about the manifolds’ large-scale geometry. See \cite{Roe96}.

\subsubsection*{Non-equivariant case}
We begin by reviewing the basic elements of the classical theory of Roe algebras, the details can be found in \cite{higson2000analytic} or \cite{willett2020higher}, etc. The proper metric space $X$ is equipped with the coarse structure defined by its metric: a subset $E$ of $X\times X$ is called controlled if $\sup_{(x,y)\in E}d(x,y)<\infty$. Roughly speaking, $E$ is not too far from the diagonal.

We say that $(H,\pi)$ is an $X$-module if $H$ is a separable Hilbert space and $\pi:C_0(X)\ra B(H)$ is a non-degenerate representation. For two $X$-modules $(H_1,\pi_1)$ and $(H_2,\pi_2)$, we can define the support of an operator $T\in B(H_1,H_2)$ as the closed subset of $X\times X$ consisting of all points $(y,x)$ such that for all open neighborhoods $U$ of $x$ and $V$ of $y$, there exists $f\in C_0(U)$ and $g\in C_0(V)$ such that
\[\pi_2(g)\circ T\circ \pi_1(f)\neq 0.\]
We can think of the support of an operator $T$ as a measure of the obstruction to $T$ intertwining $\pi_1$ and $\pi_2$. For example, if for any $f\in C_0(X)$, $T\pi_1(f)=\pi_2(f)T$, then $\supp(T)$ is contained in the diagonal. We say that $T$ is locally compact if for any $f\in C_0(Z)$, $T\pi_1(f)$ and $\pi_2(f)T$ are compact operators.

For an $X$-module $(H,\pi)$, we define the Roe algebra $C^*(H,\pi)$ as the norm closure of the *-subalgebra of locally compact operators with controlled support. For two $X$-modules $(H_1,\pi_1)$ and $(H_2,\pi_2)$, if there is an isometry $V:H_1\ra H_2$ with controlled support, then it gives a well-defined *-homomorphism
\[Ad_V:C^*(H_1,\pi_1)\ra C^*(H_2,\pi_2), T\mapsto VTV^*.\]
Moreover, the induced map in K-theory $(Ad_V)_*:K_*(C^*(H_1,\pi_1))\ra K_*(C^*(H_2,\pi_2))$ is independent of the choice of $V$.

\subsubsection*{Groupoid-equivariant case}
In this setting, we say that $(E,\pi)$ is a $Z,\calG,A$-module if $E$ is a countably generated $\calG$-Hilbert $A$-module and $\pi:C_0(Z)\ra \calL(E)$ is a non-degenerate $\calG$-equivariant *-homomorphism. We use the term \emph{operator} for a bounded $A$-linear map between two Hilbert $A$-modules.

For $Z_i,\calG,A$-modules $(E_i,\pi_i)$ ($i=1,2$) and an operator $T:E_1\ra E_2$ (which is not necessarily adjointable), we can similarly define the support of $T$ as a closed subset of $Z_2\times_{\calG\units}Z_1$ consisting of points $(z_2,z_1)$ such that for all open neighborhoods $U$ of $z_1$ and $V$ of $z_2$, there exist $f\in C_0(U)$ and $g\in C_0(V)$ such that
\[\pi_2(g)T\pi_1(f)\neq 0.\]
We say that $T$ is locally compact if for any $f\in C_0(Z_1)$ and $g\in C_0(Z_2)$, $T\pi_1(f)$ and $\pi_2(g)T$ are compact operators.

We consider the following extended ``coarse structure'': a subset $F$ of $Z_2\times_{\calG\units} Z_1$ is called controlled if it is contained in some $\calG$-compact subset of $Z_2\times_{\calG\units} Z_1$. A first difference arises here: if $Z$ is not $\calG$-compact, this coarse structure is not unital. To address this, we introduce a weaker coarse structure. We say that a subset $F$ of $Z\times_{\calG}Z$ is weakly controlled if it is contained in some $\calG$-invariant closed subset $F'$ of $Z\times_{\calG\units}Z$, such that for $i=1,2$, the restriction of the projection map $Z_2\times_{\calG\units}Z_1\ra Z_i$ to $F'$ is a proper map. 

For an operator, we say that it is controlled (resp. weakly controlled), if its support is controlled (resp. weakly controlled). Thus, a $\calG$-equivariant properly supported operator is weakly controlled.

A useful fact is that, for $F_2\subseteq Z_3\times_{\calG\units}Z_2$, $F_1\subseteq Z_2\times_{\calG\units}Z_1$, if one is controlled and the other is weakly controlled, then
\[F_2\circ F_1=\{(z_3,z_1)\in Z\times_{\calG\units}Z: \exists z_2\in Z, (z_3,z_2)\in F_2, (z_2,z_1)\in F_1\}\]
is controlled. This implies that a composition of a controlled operator and a weakly controlled operator is controlled.

A bounded operator between Hilbert spaces is always adjointable, but this does not hold for Hilbert modules. We therefore do not assume adjointability in general. Nevertheless, we obtain the following result.

\begin{thm}\label{main result properly supported locally compact implies adjointable}
    Let $(E_i,\pi_i)$ be a $Z_i,\calG,A$-module for $i=1,2$. If $T:E_1\ra E_2$ is a properly supported locally compact operator, then $T$ is adjointable.
\end{thm}

In particular, if we define $C^*_\calG(E,\pi)$ as the norm closure of the subspace consisting of all operators that are $\calG$-equivariant, locally compact and controlled, then $C^*_\calG(E,\pi)$ is a *-subalgebra of $\calL(E)$, the \cst-algebra of adjointable operators.

Some technical difficulties arise when defining the maps $Ad_V$ for non-adjointable isometries $V$. Such isometries need to be considered in the groupoid-equivariant setting, for reasons explained below. As motivation, recall that even if an isometry $V:E\ra F$ between two Hilbert modules is not adjointable, it still induces a *-homomorphism between the \cst-algebras of compact operators
\[Ad_V:\calK(E)\ra \calK(F),\quad \theta_{e,f}\mapsto \theta_{Ve,Vf}.\]
Equivalently, we can write the map as $Ad_V(T)=(V\circ (V\circ T)^*)^*$, since the compactness of $T$ implies that $VT$ is compact and hence adjointable; in turn, $V(VT)^*$ is also compact and adjointable. The same construction applies to Roe algebras and properly supported isometries by theorem \ref{main result properly supported locally compact implies adjointable}: if $V$ is a properly supported isometry and $T$ is a locally compact controlled operator, then $VT$ is also locally compact and controlled, and hence adjointable. We can then apply the same argument once again.

\begin{thm}\label{thm main 1.2}
    Let $\calG$ be a locally compact Hausdorff \'etale groupoid, $A$ be a $\calG$-\cst-algebra and $Z_1,Z_2,Z_3$ be locally compact Hausdorff proper $\calG$-spaces. Let $(E_i,\pi_i)$ be a $Z_i,\calG,A$-module for $i=1,2,3$.
    \begin{enumerate}
        \item Let $V:E_1\ra E_2$ be a $\calG$-equivariant properly supported isometry (not necessarily adjointable). Then the map
        \[Ad_V:C_\calG^*(E_1,\pi_1)\ra C_\calG^*(E_2,\pi_2), \quad T\mapsto (V\circ (V\circ T)^*)^*\]
        is a well-defined *-homomorphism between the two Roe algebras. Moreover, the induced map in K-theory $(Ad_V)_*:K_*(C_\calG^*(E_1,\pi_1))\ra K_*(C_\calG^*(E_2,\pi_2))$ is independent of the choice of $V$.
        \item For two $\calG$-equivariant properly supported isometries $V:E_1\ra E_2$ and $W:E_2\ra E_3$, we have $Ad_W\circ Ad_V=Ad_{W\circ V}$.
    \end{enumerate}
\end{thm}

\subsection*{Analog of the functor \texorpdfstring{$X\mapsto K_*(C^*(X))$}{K(C*(X))}}

\subsubsection*{Non-equivariant case}
The Roe algebra associated with an ample $X$-module is called the Roe algebra of $X$ and is denoted by $C^*(X)$. In the word of \cite{willett2020higher}, the group $K_*(C^*(X))$ is a natural home for higher indices and a coarse geometric invariant.

Recall that an $X$-module $(H,\pi)$ is ample if for any nonzero $f\in C_0(X)$, $\pi(f)$ is not a compact operator. Indeed, the role of ample modules can be explained by the following fact: if $(H_1,\pi_1)$ is an $X$-module and $(H_2,\pi_2)$ is an ample $X$-module, then there exists an isometry $V:H_1\ra H_2$ with controlled support.

We can see it from the following viewpoint: in the category whose objects are $X$-modules and whose morphisms are controlled isometries, ample modules are weakly terminal objects. Consider the category whose objects are the graded abelian groups $K_*(C^*(H,\pi))$ for all the $X$-modules $(H,\pi)$ and whose morphisms are all the maps induced by controlled isometries. Then $K_*(C^*(X))$ is a terminal object in this category, and hence unique up to isomorphism.

\subsubsection*{Groupoid-equivariant case}
Motivated by the non-equivariant case, we now define the relation $\preceq$ between two $Z,\calG,A$-modules as the existence of $\calG$-equivariant properly supported isometries between them. We say that $(E,\pi)$ is a universal $Z,\calG,A$-module if for any $Z,\calG,A$-module $(E',\pi')$, we have $(E',\pi')\preceq (E,\pi)$. That is, a universal $Z,\calG,A$-module is a weakly terminal object in this directed set of equivariant unitary equivalence classes of $Z,\calG,A$-modules. Thus, universal modules play the same role here as ample modules do in the non-equivariant case. For technical reasons, we require universal modules to be stable in the sense of definition \ref{defn stable modules}.

Even without knowing whether universal $Z,\calG,A$-modules exist, theorem \ref{thm main 1.2} implies that the graded abelian groups $K_*(C_\calG^*(E,\pi))$ for all the (unitary equivalence classes of) $Z,\calG,A$-modules $(E,\pi)$ and the homomorphisms induced by adjoint maps given by $\calG$-equivariant properly supported isometries form a directed system, so we obtain a well-defined directed limit
\[KC_*(Z;\calG,A)=\varinjlim_{(E,\pi)}K_*(C_\calG^*(E,\pi)).\]
The functor $Z\mapsto KC_*(Z;\calG,A)$ is then an analog of the functor $X\mapsto K_*(C^*(X))$.

\begin{thm}
    Let $\calG$ be a locally compact Hausdorff \'etale groupoid and $A$ be a $\calG$-\cst-algebra.
    \begin{enumerate}
        \item Let $Z$ be a locally compact Hausdorff proper $\calG$-space. If $(E,\pi)$ is a universal $Z,\calG,A$-module, then
        \[KC_*(Z;\calG,A)\cong K_*(C_\calG^*(E,\pi)).\]
        \item For every $\calG$-equivariant proper continuous map $f:Z_1\ra Z_2$ between two locally compact Hausdorff proper $\calG$-spaces, let the map
        \[f_*:KC_*(Z_1;\calG,A)\ra KC_*(Z_2;\calG,A)\]
        be as described in proposition \ref{functoriality of KC}. Then the assignments
        \[Z\mapsto KC_*(Z;\calG,A),\quad f\mapsto f_*\]
        form a well-defined functor from the category of locally compact Hausdorff proper $\calG$-spaces and $\calG$-equivariant proper continuous maps to the category of graded abelian groups.
        \item For any two $\calG$-equivariant proper continuous maps $f,g:Z_1\ra Z_2$ between two locally compact Hausdorff proper $\calG$-spaces, we say that $f$ is close to $g$ if the set
        \[\{(f(z_1),g(z_1))\in Z_2\times_{\calG\units}Z_2:z_1\in Z_1\}\]
        is a weakly controlled subset of $Z_2\times_{\calG\units}Z_2$. In this case,
        \[f_*=g_*: KC_*(Z_1;\calG,A)\ra KC_*(Z_2;\calG,A).\]
    \end{enumerate}
\end{thm}

Similarly, $KC_*(Z;\calG,A)$ is a natural home for a $\calG$-equivariant coarse index map from $\kk_*^\calG(C_0(Z),A)$. The coarse index map will be investigated in the next paper.

\subsection*{Construction of universal modules}

\subsubsection*{Non-equivariant case}
So far, the only property of ample modules we have needed is their universality: if $(H,\pi)$ is an ample $X$-module, then for any $X$-module $(H',\pi')$, there exists a controlled isometry $H'\ra H$. The proof uses a Borel decomposition: let $X=\sqcup_{i\in I}B_i$ be a partition of $X$ into a countable family of Borel subsets with non-empty interiors, such that $\cup_{i\in I}\overline{B_i}\times \overline{B_i}$ is a controlled subset of $X\times X$. The ampleness of $\pi$ implies that $\bar\pi(\chi_{B_i})H$ is an infinite-dimensional Hilbert space for each $i\in I$, where $\bar\pi_i$ is the canonical extension of $\pi$ to a unital representation of $\ca B_b(X)$, the \cst-algebra of bounded Borel functions on $X$. Hence, there exists an isometry $V_i:\bar\pi'(\chi_{B_i})H'\ra \bar\pi(\chi_{B_i})H$. So
\[V:H'\cong \oplus_{i\in I}\bar\pi'(\chi_{B_i})H'\ra \oplus_{i\in I}\bar\pi(\chi_{B_i})H\cong H\]
is the desired controlled isometry. See \cite[Construction 4.2.5]{willett2020higher}.

\subsubsection*{Groupoid-equivariant case}

Establishing a universal module in general is expected to be difficult in this case. For example, we cannot simply copy the method above, since there is no reasonable extension of a representation on a Hilbert modules of $C_0(Z)$ to $\ca B_b(Z)$. The main result in this part is to prove their existence for groupoid simplicial complexes with suitable hypotheses.

\begin{thm}\label{main thm existence of universal module}
    Let $\calG$ be a second countable locally compact Hausdorff \'etale groupoid, $A$ be a separable $\calG$-\cst-algebra and $(Y,\Delta)$ be a proper cocompact $\calG$-simplicial complex satisfying hypotheses $(H_1)$ and $(H_2)$ (see definition \ref{defn groupoid simplicial complex} and definition \ref{H_1 H_2}) such that $Y$ is second countable. Then there exists a universal $|\Delta|,\calG,A$-module.
\end{thm}

Note that Rips complexes satisfy the hypotheses of theorem \ref{main thm existence of universal module}, which suffices for the formulation of the Baum--Connes conjecture for \'etale groupoids.

The construction is achieved by a similar decomposition technique and a detailed study of the local structure of groupoid simplicial complexes. If $(E,\pi)$ is a $Z,\calG,A$-module, $\pi$ does not, in general, admit a natural extension to $\ca B_b(Z)$. So instead of a Borel decomposition, we need an ``equivariant'' open cover. Thanks to the local structure of proper actions (see proposition \ref{local structure of proper cocompact action of étale groupoid}), we can choose a countable family of open sets $(V_i)_{i\in I}$ and a countable family of proper open subgroupoids $(\calH_i)_{i\in I}$ such that
\begin{enumerate}
    \item for each $i\in I$, $\rho(V_i)\subseteq \calH_i\units$, $V_i$ is an $\calH_i$-invariant relatively compact open subset of $Z$, and for any $\gamma\in \calG\setminus\calH_i$, $\gamma\overline{V_i}\cap \overline{V_i}=\emptyset$ (which implies that $\calG V_i$ is $\calG$-equivariantly homeomorphic to a balanced product $\calG_{\calH\units}\times_\calH V_i$);
    \item $(\calG V_i)_{i\in I}$ is a countable locally finite cover of $Z$.
\end{enumerate}
In a suitable sense, this cover can be chosen to be arbitrarily fine. By analogy with the non-equivariant case, we develop the following decomposition technique.
\begin{thm}\label{thm main 1.4}
    Let $\calG$ be a second countable locally compact Hausdorff \'etale groupoid, $A$ be a separable $\calG$-\cst-algebra, $Z$ be a locally compact Hausdorff proper $\calG$-space. Let $(V_i)_{i\in I}$ and $(\calH_i)_{i\in I}$ be defined as above. Assume that $(E',\pi')$ is a $Z,\calG,A$-module and $(E,\pi)$ is a stable $Z,\calG,A$-module. For each $i\in I$, we write
    \[E_{[V_i]}=\overline{\pi(C_0(V_i))E}\cong C_0(V_i)\otimes_{\pi}E,\quad E'_{[V_i]}=\overline{\pi'(C_0(V_i))E'}\cong C_0(V_i)\otimes_{\pi'}E',\]
    which are canonically $\calH_i$-Hilbert $A|_{\calH_i}$-modules.
    
    If, for every $i\in I$, there exists an $\calH_i$-equivariant isometry $T_i\in \Hom_A(E'_{[V_i]}, E_{[V'_i]})$, then there exists a $\calG$-equivariant properly supported isometry $V\in \Hom_A(E',E)$. Moreover, $\supp_{\pi,\pi'}(V)\subseteq \cup_{i\in I}\calG(\overline{V_i}\times_{\calG\units}\overline{V_i})$.
\end{thm}

Note that, even if every $T_i$ is adjointable, the proof of theorem \ref{thm main 1.4} involves isometries that are not necessarily adjointable. This explains why we need to consider non-adjointable isometries.

An observation concerning theorem \ref{thm main 1.4} is that, if we can construct a stable $Z,\calG,A$-module such that each $E_{[V_i]}$ contains $L^2(\calH_i, A|_{\calH_i})^\infty$ as an orthogonal direct summand, then by an equivariant analog of Kasparov's stabilization theorem, $(E,\pi)$ will be a universal $Z,\calG,A$-module.

To generalize the coarse decomposition method for the Baum--Connes conjecture for groupoids (Guentner--Willett--Yu \cite{Guentner-Willet-Yu17} \cite{Guentner-Willet-Yu24}, Oyono-Oyono \cite{oyono2023groupoids}, Oyono-Oyono--Yu \cite{Oyono-Yu}, etc.), we need an analog of the localization algebra in this setting. In \cite{oyono2023groupoids}, Oyono-Oyono proved that for second countable locally compact Hausdorff groupoids with Haar systems that admit $\gamma$-elements, the Baum--Connes conjecture is stable under coarse decomposition of groupoids. We are interested in proving this permanence property without assuming the existence of a $\gamma$-element, and our approach can be seen as a first step in this direction.

The paper is organized as follows. In section \ref{sec prelim}, we recall the basic notions of $C_0(X)$-algebras, upper semicontinuous \cst-bundles and groupoid \cst-algebras. In section \ref{sec geo mod}, we introduce some operations on geometric modules and study the supports of operators on geometric modules. In section \ref{sec roe alg}, we use our construction involving properly supported isometries to develop a framework of Roe algebras and their K-theory for spaces equipped with proper actions of \'etale groupoids. We define a functor $KC_*(-;\calG,A)$ to replace the functor $K_*(C^*(X))$ in the non-equivariant case. In section \ref{sec construction univ mod}, we use the local structure of groupoid simplicial complexes to prove that proper cocompact groupoid simplicial complexes admit universal geometric modules.

\section*{Acknowledgements}

This paper covers part of the author's doctoral thesis. The author would like to thank his advisor, Professor Herv\'e Oyono-Oyono for his guidance and encouragement. The author was partially supported by the project OpART (ANR-23-CE40-0016) of the Agence Nationale de la Recherche.

\section{Preliminaries}\label{sec prelim}

\subsection{\texorpdfstring{$C_0(X)$}{C0(X)}-algebras and upper semicontinuous bundles}
We start by recalling the notion of (upper semicontinuous) Banach bundles. The references are \cite[Appendix A]{Muhly-Williams} and \cite[Appendix C]{williams2007crossed}. What they called upper semicontinuous Banach bundles and upper semicontinuous \cst-bundles will be simply called Banach bundles and \cst-bundles here.

In this section we fix $X$ as a locally compact Hausdorff space.

\begin{defn}[Banach bundles]
    A Banach bundle over $X$ is a topological space $\calA$ equipped with a continuous open surjection $p: \calA\ra X$, and every fiber $\calA_x:=p\inv(x)$ is equipped with complex Banach space structure, such that
\begin{enumerate}
    \item the map $a\mapsto \|a\|:=\|a\|_{\calA_{p(a)}}$ is upper semicontinuous from $\calA$ to $\mathbb R$;
    \item the addition $\calA\times_X \calA\ra \calA, (a,b)\mapsto a+b$ is continuous;
    \item for every $\lambda\in \mathbb C$, $\calA\ra \calA, a\mapsto \lambda a$ is continuous;
    \item if $(a_\lambda)_\lambda$ is a net in $\calA$ such that $p(a_\lambda)\ra x$ and $\|a_\lambda\|\ra 0$, then $a_\lambda\ra 0_x$.
\end{enumerate}

A morphism between two Banach bundles $\ca A, \ca B$ over $X$ is a continuous map $\varphi:\ca A\ra \ca B$, such that for every $x\in X$, $\varphi(\ca A_x)\subseteq \ca B_x$, and every fiber $\varphi_x:=\varphi|_{\ca A_x}:\ca A_x\ra \ca B_x$ is a bounded linear map.
\end{defn}

The following results are derived from the axioms in the definition.

\begin{prop}\label{first properties of Banach bundles}
    \textnormal{(C.f. \cite[Proposition C.17, Lemma C.18]{williams2007crossed})} Let $\ca A\ra X$ be a Banach bundle.
    \begin{enumerate}
        \item If $(a_\lambda)_\lambda$ is a net in $\ca A$ that converges to $0_x\in \ca A_x$ for some $x\in X$, then $\|a_\lambda\|\ra 0$.
        \item The scalar multiplication $\mathbb C\times \ca A\ra \ca A$ is continuous.
    \end{enumerate}
\end{prop}

\begin{defn}[\cst-bundles]
    A \cst-bundle over $X$ is a Banach bundle $p:\ca A\ra X$, such that every fiber is a \cst-algebra, and
    \begin{enumerate}
        \item The map $(a,b)\mapsto ab$ is continuous from $\ca A\times_{p,X,p}\ca A$ to $\ca A$.
        \item The map $a\mapsto a^*$ is continuous from $\ca A$ to $\ca A$.
    \end{enumerate}

    A morphism between two \cst-bundles $\ca A, \ca B$ over $X$ is a morphism of Banach bundles such that every fiber is a *-homomorphism of \cst-algebras.
\end{defn}

For a Banach bundle $p:\ca A\ra X$, we denote the space of bounded continuous sections by $\Gamma_b(X,\ca A)$, the space of continuous sections that vanish at infinity by $\Gamma_0(X,\ca A)$ and the spaces of compactly supported continuous sections by $\Gamma_c(X,\ca A)$. The three space are all equipped with a norm $\|f\|:=\sup_{x\in X}\|f(x)\|_{\ca A_x}$, such that $\Gamma_b(X,\ca A)$ and $\Gamma_0(X,\ca A)$ are Banach spaces, and $\Gamma_c(X,\ca A)$ is dense in $\Gamma_0(X,\ca A)$. If $\ca A\ra X$ is a \cst-bundle, then $\Gamma_b(X,\ca A)$ and $\Gamma_0(X,\ca A)$ are canonically \cst-algebras.

The following result asserts that a Banach bundle over a locally compact Hausdorff spaces have enough sections.

\begin{prop}\label{banach bundle has enough sections}
\textnormal{\cite[Corollary 2.10]{Lazar18}}
Let $\ca A\ra X$ be a Banach bundle. Then for any $x\in X$ and $a_x\in \ca A_x$, there exists $\xi\in \Gamma_0(X,\ca A)$, such that $\xi(x)=a_x$.
\end{prop}

We recall now the notion of $C_0(X)$-algebras.

\begin{defn}[$C_0(X)$-algebras]
    A $C_0(X)$-algebra is a pair $(A,\Phi_A)$, where $A$ is a \cst-algebra, $\Phi_A$ is a non-degenerate *-homomorphism from $C_0(X)$ to $ZM(A)$. The non-degeneracy means that the linear span of $\Phi_A(C_0(X))A$ is dense in $A$. We call $\Phi_A$ the structure map of the $C_0(X)$-algebra $A$, and when there is no ambiguity about the choice of the structure map, we can just say that $A$ is a $C_0(X)$-algebra and write $\Phi_A(f)a$ as $f\cdot a$ for any $f\in C_0(X)$, $a\in A$.

    Let $A,B$ be two $C_0(X)$-algebras. A $C_0(X)$-linear *-homomorphism $\varphi:A\ra B$ is a *-homomorphism $\varphi$ such that for any $f\in C_0(X)$,
    \[\varphi(f\cdot a)=f\cdot \varphi(a).\]
\end{defn}

\begin{ex}
Let $p:\ca A\ra X$ be a \cst-bundle. We define the map
\[\Phi:C_0(X)\ra ZM(\Gamma_0(X,\ca A)),\]
\[[\Phi(f)\xi](x)=f(x)\xi(x),\quad \forall f\in C_0(X), \xi\in \Gamma_0(X,\ca A).\]
Then $(\Gamma_0(X,\ca A),\Phi)$ is a $C_0(X)$-algebra (see \cite[Proposition C.23]{williams2007crossed}).
\end{ex}

\begin{defn}
    Let $(A,\Phi_A)$ be a $C_0(X)$-algebra. For any $x\in X$, we define the fiber of $A$ at $x$ as the quotient
    \[A_x=A/\overline{\Phi_A(C_0(X\setminus\{x\}))A}.\]
    For $a\in A$, we denote its image under the quotient map as $a(x)$.

    We define the \cst-bundle associated to $A$ as the canonical surjection
    \[\ca A:=\sqcup_{x\in X}A_x\ra X,\]
    where $\ca A$ is equipped with the weakest topology such that for any $a\in A$, the map $x\mapsto a(x)\in A_x$ is a continuous map from $X$ to $\ca A$.
\end{defn}

Indeed, \cite[Proposition C.10, Proposition C.25]{williams2007crossed} ensure that $\ca A\ra X$ is a well-defined \cst-bundle, and there is a canonical $C_0(X)$-linear isomorphism $A\cong \Gamma_0(X,\ca A)$ by \cite[Theorem C.26]{williams2007crossed}. From now on, without specific mentioning, when we denote a $C_0(X)$-algebra by a Roman capital letter, we will implicitly denote its associated \cst-bundle over $X$ by the corresponding calligraphic letter.

\begin{prop}
    \textnormal{\cite[Proposition 3.9, Proposition 3.10]{Lalonde15}}
    Let $(A,\Phi_A)$ be a $C_0(X)$-algebra and $I$ be a closed two-sided ideal of $A$. Then
    \[\Phi_I: C_0(X)\ra ZM(I),\quad \Phi_I(f)(a)=\Phi_A(f)a\]
    gives a $C_0(X)$-algebra structure of $I$. For any $x\in X$, let $q_x:A\ra A_x$ be the quotient map.
    \begin{enumerate}
        \item For any $x\in X$, $I_x:=q_x(I)\cong I/\overline{\Phi_I(C_0(X\setminus\{x\}))I}$;
        \item Let $\ca I:=\sqcup_{x\in X}q_x(I)$ be equipped with subspace topology of $\ca A=\sqcup_{x\in X}A_x$, then $\ca I$ is a \cst-bundle over $X$ with $\Gamma_0(X,\ca I)\cong I$.
    \end{enumerate}  
\end{prop}

\begin{defn}
    Let $(A,\Phi_A)$ be a $C_0(X)$-algebra, $E$ be a Hilbert $A$-module. We define an action of $C_0(X)$ on $E$ as
    \[\Phi_E(f)(ea)=e(\Phi_A(f)a)\]
    for any $f\in C_0(X)$, $e\in E$ and $a\in A$. (Since $EA=E$ by \cite[Lemma 1.3]{blanchard1996deformations}, we can check easily that this action is well-defined.) This gives a non-degenerate *-homomorphism
    \[\Phi_E:C_0(X)\ra Z\calL(E).\]
    The non-degeneracy means that the linear span of $\Phi_E(C_0(X))E$ is dense in $E$. We call $\Phi_E$ the structure map of $E$. We write $f\cdot e=\Phi_E(f)e$ when there is no ambiguity about the choice of the structure map.

    For any $x\in X$, we define the fiber of $E$ at $x$ as the quotient
    \[E_x:=E/\overline{\Phi_E(C_0(X\setminus\{x\})E},\]
    which is isomorphic to the inner tensor product $E\otimes_A A_x$, and hence can be seen as a Hilbert $A_x$-module. For any $e\in E$, we denote its image under the quotient map as $e(x)\in E_x$.

    We define the Banach bundle associated to $E$ as the canonical surjection
    \[\ca E:=\sqcup_{x\in X}E_x\ra X,\]
    where $\ca E$ is equipped with the weakest topology such that for any $e\in E$, the map $x\mapsto e(x)\in E_x$ is a continuous map from $X$ to $\ca E$.
\end{defn}

Whenever we denote a Hilbert module over some $C_0(X)$-algebra by a Roman capital letter, the same letter in calligraphic font will be used to denote the associated Banach bundle.

The Banach bundle $\ca E\ra X$ is in fact a Hilbert $\ca A$-bundle in notation of \cite{miller2022k}. The space $\Gamma_0(X,\ca E)$ is equipped with a Hilbert $\Gamma_0(X,\ca A)$-module structure as
\[[\eta\cdot \xi](x)=\eta(x)\xi(x), \quad \forall \eta\in \Gamma_0(\ca E,X), \xi\in \Gamma_0(X,\ca A), x\in X,\]
\[\langle \eta_1,\eta_2\rangle(x)=\langle \eta_1(x),\eta_2(x)\rangle_{E_x},\quad \forall \eta_1,\eta_2\in \Gamma_0(X,\ca E), \xi\in \Gamma_0(X,\ca A).\]
After identifying $A$ with $\Gamma_0(X,\ca A)$ through the canonical isomorphism, there is a canonical unitary isomorphism $E\cong \Gamma_0(X,\ca E)$ of Hilbert $A$-modules.

The following proposition is a consequence of \cite[Proposition 2.13]{Kaliszewski-Quigg-Williams}.
\begin{prop}\label{Hilbert bundle of submodule}
    Let $A$ be a $C_0(X)$-algebra, $E$ be a Hilbert $A$-module and $F\subseteq E$ be a closed $A$-submodule. For any $x\in X$, let $q_x:E\ra E_x$ be the quotient map.
    \begin{enumerate}
        \item For any $x\in X$, $F_x:=q_x(F)\cong F/\overline{C_0(X\setminus\{x\})\cdot F}$;
        \item Let $\ca F:=\sqcup_{x\in X}F_x$ be equipped with subspace topology of $\ca E=\sqcup_{x\in X}E_x$, then $\ca F$ is a Banach bundle over $X$ with $\Gamma_0(X,\ca F)\cong F$.
    \end{enumerate}
\end{prop}

Now if $A$ is a $C_0(X)$-algebra and $E,F$ are Hilbert $C_0(X)$-modules, $x\in X$, for a bounded $A$-linear map (which will be called an operator in rest of the paper) $T\in \Hom_A(E,F)$, the $A$-linearity implies that $T$ is $C_0(X)$-linear, and hence maps $\overline{C_0(X\setminus\{x\})\cdot E}$ into $\overline{C_0(X\setminus\{x\})\cdot F}$. Therefore, there is uniquely an operator $T_x\in \Hom_{A_x}(E_x,F_x)$, called the fiber of $T$ at $x$, such that
\[(Te)(x)=T_x e(x), \quad \forall e\in E.\]
If $T$ is adjointable, then so is $T_x$. And if $T\in \calK(E,F)$ is a compact operator, then so is $T_x$. We define $\calK(\ca E,\ca F):=\sqcup_{x\in X}\calK(E_x,F_x)$, equipped with the weakest topology such that for any $T\in \calK(E,F)$, the map $x\mapsto T_x\in \calK(E_x,F_x)$ is a continuous map from $X$ to $\calK(\ca E,\ca F)$, then the canonical surjection $\calK(\ca E,\ca F)\ra X$ is a Banach bundle, and there is a canonical isomorphism $\calK(E,F)\ra \Gamma_0(X,\calK(\ca E,\ca F))$. See \cite[Proposition 1.51]{miller2022k}. Especially, when $E=F$, $\calK(\ca E)=\sqcup_{x\in X}\calK(E_x)$ is a \cst-bundle, which coincides with the fact that $\Phi_E:C_0(X)\ra Z\calL(E)\cong ZM(\calK(E))$ gives a $C_0(X)$-algebra structure on $\calK(E)$.

\begin{defn}
    \textnormal{\cite[Definition 1.52]{miller2022k}} Let $A$ be a $C_0(X)$-algebra, $E,F$ be Hilbert $A$-modules. We define the set $\calL(\ca E,\ca F)$ as the disjoint union $\sqcup_{x\in X}\calL(E_x,F_x)$, equipped with the weakest topology such that for any $e\in E$, $f\in F$, the maps
    \[\calL(\ca E,\ca F)\ra \ca F,\quad T\mapsto T(e(p(T))),\]
    \[\calL(\ca E,\ca F)\ra \ca E,\quad T\mapsto T^*(f(p(T)))\]
    are continuous, where $p:\calL(\ca E,\ca F)\ra X$ is the canonical surjection. We call $\calL(\ca E,\ca F)\ra X$ the bundle of adjointable operators over $X$.
\end{defn}

\begin{prop}\label{contiuous action of adjointable operator bundle}
    \textnormal{\cite[Proposition 1.54, Proposition 1.55]{miller2022k}}
    Let $A$ be a $C_0(X)$-algebra, $E,F$ be Hilbert $A$-modules.
    \begin{enumerate}
        \item The map $T\mapsto [x\mapsto T_x]$ is an isomorphism of Banach spaces from $\calL(E,F)$ to $\Gamma_b(\ca E,\ca F)$.
        \item The application map
    \[\calL(\ca E,\ca F)\times_X \ca E\ra \ca F, (T,e)\mapsto Te\]
    and the adjoint application map
    \[\calL(\ca E,\ca F)\times_X \ca F\ra \ca E, (T,f)\mapsto T^*f\]
    are continuous.
    \end{enumerate}
\end{prop}

\begin{defn}
    Let $A, B$ be a $C_0(X)$-algebra, $E$ be a Hilbert $B$-module, $\pi:A\ra \calL_B(E)$ be a *-homomorphism (i.e. a representation). We say that $\pi$ is $C_0(X)$-linear or a $C_0(X)$-representation, if
    \[\pi(a)(f\cdot e)=\pi(f\cdot a)e,\quad \forall a\in A, e\in E, f\in C_0(X).\]
\end{defn}

The $C_0(X)$-linearity implies that, for any $x\in X$, there exists uniquely a representation $\pi_x:A_x\ra \calL(E_x)$, called the fiber of $\pi$ at $x$, such that
\[[\pi(a)e](x)=\pi_x(a(x))e(x),\quad \forall a\in A,e\in E.\]

A $C_0(X)$-representation $\pi$ is non-degenerate if and only if for any $x\in X$, $\pi_x$ is non-degenerate, see \cite[Lemma 2.14]{Mao26}.

\begin{defn}[Pullback]
    Let $f:Y\ra X$ be a continuous map between locally compact Hausdorff spaces.
    \begin{enumerate}
        \item Let $p:\ca A\ra X$ be a Banach bundle. We define the pullback of $\ca A$ by $f$ as the fiber product $f^*\ca A:=\ca A\times_{p,X,f}Y$, equipped with the continuous open surjection onto $Y$ as the projection onto the second factor. The bundle $f^*\ca A\ra Y$ is a well-defined Banach bundle over $Y$.
        \item Let $A$ be a $C_0(X)$-algebra. We define the pullback of $A$ by $f$ as the $C_0(Y)$-algebra $f^*A:=\Gamma_0(Y,f^*\ca A)$.
        \item Let $E$ be Hilbert $A$-module. We define the pullback of $E$ by $f$ as the Hilbert $f^*A$-module $f^*E:=\Gamma_0(Y,f^*\ca E)$.
        \item Let $F$ be another Hilbert $A$-module. We define the pullback of an operator $T\in \Hom_A(E,F)$ by $f$ as the operator $f^*T\in \Hom_{f^*A}(f^*E,f^*F)$ such that
        \[[(f^*T)\eta](y)=T_{f(y)}\eta(y),\quad \forall \eta\in f^*E, y\in Y.\]
        Clearly if $T$ is compact, then so is $f^*T$; if $T$ is adjointable, then so is $f^*T$, with adjoint $f^*(T^*)$.
        \item Let $D$ be another $C_0(X)$-algebra, $\pi:D\ra \calL(E)$ be a $C_0(X)$-representation. We define the pullback of $\pi$ by $f$ as the $C_0(Y)$-representation $f^*\pi:f^*D\ra \calL(f^*E)$ defined as
        \[[(f^*\pi)(\xi)\eta](y)=\pi_{f(y)}(\xi(y))\eta(y),\quad \forall \xi\in f^*D, \eta\in f^*E, y\in Y.\]
    \end{enumerate}
\end{defn}

Use the same notation as above, we have identification of fibers for any $y\in Y$: $(f^*A)_y\cong A_{f(y)}$, $(f^*E)_y\cong E_{f(y)}$ and $(f^*\pi)_y=\pi_{f(y)}$. Moreover, for the bundles of operators, we have homeomorphisms
\[f^*\calK(\ca E,\ca F)\cong \calK(f^*E,f^*F),\quad f^*\calL(\ca E,\ca F):=\calL(\ca E,\ca F)\times_X Y\cong \calL(f^*\ca E, f^*\ca F).\]
See \cite[Proposition 1.62]{miller2022k} for details.

The following lemma about pullback by local homeomorphisms will be important.

\begin{lem}\label{sum of fiber}
    \textnormal{\cite[Lemma 3.12]{bonicke2020going}, \cite[Lemma 2.19]{miller2024functors}}
    Let $\rho:Y\ra X$ be a local homeomorphism between locally compact Hausdorff spaces and $\ca A\ra X$ be a Banach bundle over $X$. Then for any $\xi\in \Gamma_b(Y,\rho^*\calA)$ such that $\rho|_{\supp(\xi)}:\supp(\xi)\ra X$ is proper,
    \[\rho_*\xi:x\mapsto \sum_{y\in Y_x}\xi(y)\]
    is well-defined and continuous. If $\xi$ is compactly supported then so is $\rho_*\xi$.
\end{lem}

\begin{defn}[Pushforward]
    Let $g:X\ra Z$ be a continuous map between locally compact Hausdorff spaces.
    
    \begin{enumerate}
        \item  Let $(A,\Phi_A)$ be a $C_0(X)$-algebra. Then the composite
        \[g_*\Phi_A:C_0(Z)\xrightarrow{g^*}C_b(X)\xrightarrow{\tilde \Phi_A} ZM(A)\]
        gives a $C_0(Z)$-algebra structure of $A$ (see \cite[Proposition 3.5]{bonicke2020going}), where $\tilde\Phi_A$ is the unique strictly continuous extension $C_b(X)\ra ZM(A)$ of $\Phi_A$ thanks to its non-degeneracy. We write $A$ as $g_*A$ to emphasize its $C_0(Z)$-algebra structure.
        \item Let $E$ be a Hilbert $A$-module. We write $E$ as $g_*E$ to emphasize the action of $C_0(Z)$ on it by the composite
        \[g_*\Phi_E:C_0(Z)\xrightarrow{g^*}C_b(X)\xrightarrow{\tilde \Phi_E}Z\calL(E).\]
        \item Let $D$ be another $C_0(X)$-algebra and $\pi:D\ra \calL(E)$ be a $C_0(X)$-representation. We write $\pi$ as $g_*\pi:g_*D\ra \calL(g_*E)$ to emphasize its $C_0(Z)$-linearity.
    \end{enumerate}
\end{defn}

Use the same notation as above, we have identifications of fibers for any $z\in Z$: $(g_*A)_z\cong \Gamma_0(X_z,\ca A|_{X_z})$, $(g_*E)_z\cong \Gamma_0(X_z,\ca E|_{X_z})$. See \cite[Proposition 3.6]{bonicke2020going}.

\subsection{Groupoid \texorpdfstring{\cst}{C*}-algebras}

We assume that the reader is familiar with basic definition of topological groupoids, especially locally compact groupoids and \'etale groupoids. More details can be found in \cite{paterson2012groupoids}, \cite{williams2019tool} etc. For convenience of writing, in this paper all \'etale groupoids are assumed to be locally compact Hausdorff.

Let $\calG$ be a locally compact Hausdorff groupoid.

\begin{defn}
    Let $X$ be a topological space and $\rho:X\ra \calG\units$ be a continuous map (called anchor map). We say that $(X,\rho)$ is a (left) $\calG$-space, or $X$ is a (left) $\calG$-space, if there is a continuous map (called a left action of $\calG$ on $X$)
    \[\calG\times_{s,\calG\units,\rho}X\ra X, (\gamma, x)\mapsto \gamma x,\]
    such that
    \begin{enumerate}
        \item for any $x\in X$, $\rho(x)x=x$;
        \item for any $(\gamma_1,\gamma_2)\in \calG^{(2)}$ and $x\in X$ such that $s(\gamma_2)=\rho(x)$, we have $\gamma_1(\gamma_2 x)=(\gamma_1\gamma_2)x$.
    \end{enumerate}
    The right actions of $\calG$ on $X$ is defined similarly.
\end{defn}

Now if $X$ is a locally compact Hausdorff $\calG$-space, there is an associated locally compact Hausdorff groupoid $\calG\ltimes X$ called action groupoid (see \cite[Definition 2.5]{williams2019tool}). If $\calG$ is an \'etale groupoid, then so is $\calG\ltimes X$. We say that the action of $\calG$ on $X$ is proper (free, resp.), if $\calG\ltimes X$ is a proper (principal, resp.) groupoid. When the anchor map is a local homeomorphism, we say that $X$ is an \'etale $\calG$-space.

\begin{defn}
    Let $p:\ca A\ra \calG\units$ be a Banach bundle. We say that $\ca A$ is a Banach $\calG$-bundle, if $(\ca A,p)$ is a $\calG$-space, and for any $\gamma \in \calG$, the map
    \[\ca A_{s(\gamma)}\ra \ca A_{r(\gamma)},\quad a\mapsto \gamma.a\]
    is an isomorphism of Banach spaces.
\end{defn}

We recall now the definition of groupoid actions on \cst-algebras or Hilbert modules.

\begin{defn}
    A $\calG$-\cst-algebra is a $C_0(\calG\units)$-algebra $A$ together with a $C_0(\calG)$-linear *-isomorphism $\alpha: s^*A\ra r^*A$ such that for any composable pair $(\gamma,h)\in \calG^{(2)}$, $\alpha_{\gamma h}=\alpha_\gamma\circ \alpha_h:A_{s(h)}\ra A_{r(\gamma)}$.
     
    A $\calG$-equivariant *-homomorphism between two $\calG$-\cst-algebras $(A,\alpha)$ and $(B,\beta)$ is a $C_0(\calG\units)$-linear *-homomorphism $\varphi:A\ra B$ such that $r^*\varphi\circ \alpha=s^*\varphi \circ \beta$.
\end{defn}

From the definition, we see that the associated \cst-bundle $\ca A$ is a Banach $\calG$-bundle.

\begin{ex}
Let $X$ be a locally compact Hausdorff $\calG$-space, then there is a homeomorphism
\[\calG\times_{s,\calG\units,\rho}X\ra \calG\times_{r,\calG\units,\rho}X,\quad (\gamma, x)\mapsto (\gamma,\gamma x).\]
It induces a $C_0(\calG)$-linear *-homomorphism $lt:s^*C_0(X)\cong C_0(\calG\times_{s,\calG\units,\rho}X)\ra C_0(\calG\times_{r,\calG\units,\rho}X)\cong r^*C_0(X)$.
The fiber of $lt$ at $\gamma\in \calG$ is the *-homomorphism
\[lt_\gamma: C_0(X_{s(\gamma)})\ra C_0(X_{r(\gamma)}),\quad f\mapsto f(\gamma\inv-).\]
It is easy to see that $(C_0(X),lt)$ is a well-defined $\calG$-\cst-algebra.
\end{ex}

\begin{defn}
    Let $(A,\alpha)$ be a $\calG$-\cst-algebra. A $\calG$-Hilbert $A$-module is a Hilbert $A$-module $E$ together with a unitary $V\in \calL_{s^*A}(s^*E,r^*E)$ (here $r^*E$ is seen as a Hilbert $s^*A$-module after being identified with $r^*E\otimes_{\alpha\inv} s^*A$) such that for any composable pair $(\gamma,h)\in \calG^{(2)}$, $V_{\gamma h}=V_\gamma \circ V_h$.
\end{defn}

From the definition, we see that the associated Banach bundle $\ca E$ is a Banach $\calG$-bundle. For any $\gamma\in \calG$ and $e_1\in E_{s(\gamma)}, e_2\in E_{r(\gamma)}$, we have
\[\langle V_\gamma e_1, e_2\rangle_{E_{r(\gamma)}}=\alpha_\gamma(\langle e_1,V_\gamma^* e_2\rangle_{E_{s(\gamma)}}).\]

Naturally, for a $\calG$-\cst-algebra $(A,\alpha)$ and a $\calG$-Hilbert $A$-module $(E,V)$, their associated Banach bundles can be seen as equipped with canonical action of $\calG$. The continuity of the map $\calG\times_{s,\calG\units}\ca A\ra \calG\times_{r,\calG\units}\ca A$ corresponds to the $C_0(\calG)$-linearity of $\alpha$, and the continuity of the map $\calG\times_{s,\calG\units}\ca E\ra \calG\times_{r,\calG\units}\ca E$ corresponds to the fact that $V$ gives a well-defined continuous section in $\Gamma_b(\calG, \calL(s^*\ca E, r^*\ca E))\cong \calL(s^*E, r^*E)$.

\begin{prop}\label{determine continuity of groupoid action}
    \begin{enumerate}
        \item \textnormal{\cite[Lemma 3.9]{bonicke2020going}} Let $A$ be a $C_0(\calG\units)$-algebra and $(\alpha_\gamma)_{\gamma\in \calG}$ be a family of *-isomorphisms $\alpha_\gamma: A_{s(\gamma)}\ra A_{r(\gamma)}$ such that for any composable pair $(\gamma,h)\in \calG^{(2)}$, $\alpha_{\gamma h}=\alpha_\gamma\circ \alpha_h$. Then $(\alpha_\gamma)_{\gamma\in \calG}$ gives a $\calG$-\cst-algebra structure on $A$ if and only if for any $a\in A$, the map $\gamma\mapsto \alpha_\gamma(a(s(\gamma)))$ is continuous section $\calG\ra r^*\ca A$.
        \item Let $A$ be a $\calG$-\cst-algebra and $E$ be a Hilbert $A$-module, $(V_\gamma)_{\gamma\in \calG}$ be a family of unitary isomorphisms $V_\gamma: E_{s(\gamma)}\ra E_{r(\gamma)}$ such that for any composable pair $(\gamma,h)\in \calG^{(2)}$, $V_{\gamma h}=V_\gamma \circ V_h$. Then $(V_\gamma)_{\gamma\in \calG}$ gives a $\calG$-Hilbert $A$-module structure on $E$ if and only if for any $e\in E$, the map $\gamma\mapsto V_\gamma(e(s(\gamma)))$ is continuous section $\calG\ra r^*\ca E$.
    \end{enumerate}
\end{prop}

\begin{proof}
We prove only (2). If there exists a unitary $V\in \calL_{s^*A}(s^*E,r^*E)$ such that $V_\gamma$ is the fiber of $V$ at $\gamma$ for any $\gamma$ in $\calG$, then for any $e\in E$, the map $\gamma\mapsto V_\gamma(e(s(\gamma)))$ is clearly a continuous section. Conversely, it suffices to prove that the map $\calG\ra \calL(s^*\ca E, r^*\ca E), \gamma\mapsto V_\gamma$ is strictly continuous. Use the same method as above, the map
\[\beta: \calG\times_{s,\calG\units}\ca E\ra \ca E, (\gamma,e)\mapsto V_\gamma(e)\]
is continuous. For any $\xi\in s^*E$, $\eta\in r^*E$ and net $(\gamma_\lambda)_\lambda$ in $\calG$ converging to $\gamma$, using the continuity of $\beta,\xi$ and $\eta$, $V_{\gamma_\lambda}(\xi(\gamma_\lambda))=\beta(\gamma_\lambda,\xi(\gamma_\lambda))\ra \beta(\gamma,\xi(\gamma))=V_\gamma(\xi(\gamma))$ and $V^*_{\gamma_\lambda}(\eta(\gamma_\lambda))=\beta(\gamma_\lambda\inv,\eta(\gamma_\lambda))\ra \beta(\gamma\inv,\eta(\gamma))=V^*_\gamma(\eta(\gamma))$. Hence, the map $\calG\ra \calL(s^*\ca E, r^*\ca E)$ is strictly continuous.
\end{proof}

\begin{defn}
    Let $A$ be a $\calG$-\cst-algebra, $(E,V)$ and $(F,W)$ be two $\calG$-Hilbert $A$-modules. We say that a bounded $A$-linear map $T:E\ra F$ is $\calG$-equivariant, if for any $\gamma\in \calG$, $W_\gamma\circ T_{s(\gamma)}=T_{r(\gamma)}\circ V_\gamma$.
\end{defn}

\begin{defn}
    Let $(A,\alpha)$ and $(B,\beta)$ be two $\calG$-\cst-algebras, $(E,V)$ be a $\calG$-Hilbert $B$-module, $\pi:A\ra \calL_B(E)$ be a $C_0(\calG\units)$-representation. We say that $\pi$ is $\calG$-equivariant, if for any $\gamma\in \calG$, $a\in A_{s(\gamma)}$, we have
    \[\pi_{r(\gamma)}(\alpha_\gamma(a))=V_\gamma\circ\pi_{s(\gamma)}(a)\circ V_\gamma^*.\]
\end{defn}

The proofs of the following results are left for the reader.

\begin{prop}\label{continuity of action on pullback}
    Let $(X,\rho)$ be a locally compact Hausdorff $\calG$-space, $(A,\alpha)$ be a $\calG$-\cst-algebra, and $(E,V)$ be a $\calG$-Hilbert $A$-module. For every $(\gamma,x)\in \calG\ltimes X$, we define the *-isomorphism
    \[\beta_{(\gamma,x)}:(\rho^*A)_x\cong A_{s(\gamma)}\xrightarrow{\alpha_\gamma}A_{r(\gamma)}\cong (\rho^*A)_{\gamma x}\]
    and the unitary isomorphism
    \[W_{(\gamma,x)}:(\rho^*E)_x\cong E_{s(\gamma)}\xrightarrow{V_\gamma} E_{r(\gamma)}\cong (\rho^*E)_{\gamma x},\]
    then $(\beta_{(\gamma,x)})_{(\gamma,x)\in \calG\ltimes X}$ gives a $\calG\ltimes X$-\cst-algebra structure of $\rho^*A$, $(W_{(\gamma,x)})_{(\gamma,x)\in \calG\ltimes X}$ gives a $\calG\ltimes X$-Hilbert $\rho^*A$-module structure of $\rho^*E$.
\end{prop}

\begin{prop}\label{G-equivariance of pullback}
    Let $(X,\rho)$ be a locally compact Hausdorff $\calG$-space, $A,B$ be $\calG$-\cst-algebras and $E$ be a $\calG$-Hilbert $B$-module. If $\pi:A\ra \calL(E)$ is a $\calG$-equivariant representation, then $\rho^*\pi:\rho^*A\ra \calL(\rho^*E)$ is a $\calG\ltimes X$-equivariant representation.
\end{prop}

\begin{prop}\label{continuity of action on pushout}
    \textnormal{\cite[Proposition 3.10]{bonicke2020going}}
    Let $(X,\rho)$ be a locally compact Hausdorff $\calG$-space, $(A,\alpha)$ be a $\calG\ltimes X$-\cst-algebra, and $(E,V)$ be a $\calG\ltimes X$-Hilbert $A$-module. For any $\gamma\in \calG$, we define the *-isomorphism
    \[\beta_\gamma:(\rho_*A)_{s(\gamma)}\cong \Gamma_0(X_{s(\gamma)},\ca A)\ra \Gamma_0(X_{r(\gamma)},\ca A)\cong (\rho_*A)_{r(\gamma)},f \mapsto \alpha_{(\gamma,-)}(f(\gamma\inv-))\] 
    and the unitary isomorphism
    \[W_{\gamma}:(\rho_*E)_{s(\gamma)}\cong \Gamma_0(X_{s(\gamma)},\ca E)\ra \Gamma_0(X_{r(\gamma)},\ca E)\cong (\rho_*E)_{r(\gamma)}, g \mapsto V_{(\gamma,-)}(g(\gamma\inv-)),\] 
    then $(\beta_\gamma)_\gamma$ gives a $\calG$-\cst-algebra structure of $\rho_*A$, $(W_\gamma)_\gamma$ gives a $\calG$-Hilbert $\rho_*A$-module structure of $\rho_*E$.
\end{prop}

\begin{prop}\label{G-equivariance of pushout}
    Let $(X,\rho)$ be a locally compact Hausdorff $\calG$-space, $A,B$ be $\calG\ltimes X$-\cst-algebras and $E$ be a $\calG\ltimes X$-Hilbert $B$-module. If $\pi:A\ra \calL(E)$ is a $\calG\ltimes X$-equivariant representation, then $\rho_*\pi:\rho_*A\ra \calL(\rho_*E)$ is a $\calG$-equivariant representation.
\end{prop}

\begin{prop}\label{continuity of action on inner tensor product}
    Let $(A,\alpha)$ and $(B,\beta)$ be two $\calG$-\cst-algebras, $(E,V)$ be a $\calG$-Hilbert $A$-module, $(F,W)$ be a $\calG$-Hilbert $B$-module and $\pi:A\ra \calL(F)$ be a $\calG$-equivariant representation. After the identifications $s^*(E\otimes_\pi F)\cong s^*E\otimes_{s^*\pi}s^*F$ and $r^*(E\otimes_\pi F)\cong r^*E\otimes_{r^*\pi}r^*F$ by \cite[Proposition 4.2]{le1999theorie}, consider the unitary isomorphisms
    \[V_\gamma\otimes W_\gamma:E_{s(\gamma)}\otimes_{\pi_{s(\gamma)}}F_{s(\gamma)}\ra E_{r(\gamma)}\otimes_{\pi_{r(\gamma)}}F_{r(\gamma)}\] 
    for all $\gamma\in \calG$, then $(V_\gamma\otimes W_\gamma)_{\gamma\in \calG}$ gives a $\calG$-Hilbert $B$-module structure of $E\otimes_\pi F$.
\end{prop}

\begin{prop}\label{G-equivariance of inner tensor product}
    Let $A,B,D$ be $\calG$-\cst-algebras, $E$ be a $\calG$-Hilbert $B$-module, $F$ be a $\calG$-Hilbert $D$-module, $\pi:A\ra \calL_B(E)$ and $\pi':B\ra \calL_D(F)$ be $\calG$-equivariant representations. Then $\pi\otimes id: A\ra \calL_D(E\otimes_{\pi'}F)$ is also a $\calG$-equivariant representation.
\end{prop}

We recall now the notion of reduced crossed product by groupoids.
Let $\calG$ be an étale groupoid and $(A,\alpha)$ be a $\calG$-\cst-algebra.

\begin{defn}\label{defn L^2(Z,rho,E)}
    Let $\rho:Z\ra Y$ be a local homeomorphism between two locally compact Hausdorff spaces, $A$ be a $C_0(Y)$-algebra and $E$ be a Hilbert $A$-module. We define a pre-Hilbert $A$-module structure on $\Gamma_c(Z,\rho^*\ca E)$ as 
    \[(\xi\cdot a)(z)=\xi(z)a(\rho(z)),\quad \forall \xi\in \Gamma_c(Z,\rho^*\ca E),a\in A,z\in Z,\]
    \[\langle \xi_1,\xi_2\rangle(y)=\sum_{z\in \rho\inv(y)}\langle\xi_1(z),\xi_2(z)\rangle_{E_y},\quad \forall \xi_1,\xi_2\in \Gamma_c(Z,\rho^*\ca E).\]
    We denote its completion by $L^2(Z,\rho,E)$, which is a Hilbert $A$-module. 
\end{defn}

Clearly, the fiber of $L^2(Z,\rho,E)$ at $y\in Y$ is the Hilbert $A_y$-module $\ell^2(\rho\inv(y),E_y)$. The quotient map $L^2(Z,\rho,E)\ra \ell^2(\rho\inv(y),E_y)$ is the restriction map $\xi\mapsto \xi|_{Z_y}$.

\begin{lem}\label{uniformly finite fiber}
    If $\rho:Y\ra X$ is a local homeomorphism between two locally compact Hausdorff spaces, $K\subseteq Y$ a compact subset of $Y$, then $\sup_{x\in X}\#(K\cap \rho\inv(x))<\infty$.
\end{lem}
\begin{proof}
    For every $y\in K$, there exists an open neighborhood $V_y$ of $y$ such that $\rho|_{V_y}$ is a homeomorphism onto an open of $X$. Since $K$ is compact, there exists finitely many point $y_1,\cdots, y_n$, such that $K\subseteq \cup_{i=1}^nV_{y_i}$. Since every $\rho|_{V_{y_i}}$ is injective, we have $\#(K\cap \rho\inv(x))\leqslant n$ for every $x\in X$.
\end{proof}

\begin{prop}\label{L^2 has continuous G-action}
    Let $(Z,\rho)$ be a locally compact Hausdorff \'etale right $\calG$-space, $(A,\alpha)$ be a $\calG$-\cst-algebra and $(E,V)$ be a $\calG$-Hilbert $A$-module. For any $\gamma\in \calG$, we define $\tilde V_\gamma$ as the unitary isomorphism
    \[\ell^2(Z_{s(\gamma)},E_{s(\gamma)})\ra \ell^2(Z_{r(\gamma)}, E_{r(\gamma)}), f\mapsto V_\gamma(f(-\gamma)), \]
    then $(\tilde V_\gamma)_{\gamma\in \calG}$ gives a $\calG$-Hilbert $A$-module structure on $L^2(Z,\rho, E)$.
\end{prop}
\begin{proof}
    Let $\ca E'=\sqcup_{x\in \calG\units} \ell^2(Z_x,E_x)$ be the associated Hilbert bundle of $L^2(Z,\rho,E)$. By proposition \ref{determine continuity of groupoid action}, it suffices to prove that for any $\xi\in \Gamma_c(Z,\rho^*\ca E)\subseteq L^2(Z,\rho,E)$ and a net $(\gamma_\lambda)_\lambda$ in $\calG$ converging to $\gamma$, we have $\tilde V_{\gamma_\lambda}(\xi|_{Z_{s(\gamma_\lambda)}})\ra \tilde V_\gamma(\xi|_{Z_{s(\gamma)}})$ in $\ca E'$. By proposition \ref{banach bundle has enough sections}, let $\eta\in L^2(\calG,E)$ such that its fiber at $r(\gamma)$ is $\eta|_{Z_{r(\gamma)}}=\tilde V_\gamma(\xi|_{Z_{s(\gamma)}})=V_\gamma(\xi|_{Z_{s(\gamma)}}(-\gamma))$. Since $\supp(\xi)$ is a compact subset of $Z$, by Urysohn's lemma, there exists $f\in C_c(Z)$ such that $f|_{\supp(\xi)}=1$. Then $(f\cdot\xi)|_{Z_{s(\gamma)}}=\xi|_{Z_{s(\gamma)}}$. Hence, after replacing $\eta$ by $f\cdot \eta$, we can assume that $\eta\in \Gamma_c(Z,\rho^*\ca E)$.

    Claim: $\|\tilde V_{\gamma_\lambda}(\xi|_{Z_{s(\gamma_\lambda)}})-\eta|_{Z_{r(\gamma_\lambda)}}\|_{\ca E'_{r(\gamma_\lambda)}}\ra 0$.

    Let $K_1$ be a compact neighborhood of $\gamma$. Since $\gamma_\lambda$ is eventually in $K_1$, we can assume that $(\gamma_\lambda)_\lambda$ is contained in $K_1$ after replacing by a subnet. Let $K=K_1 \supp(\xi)\cup \supp(\eta)$, which is a compact subset of $\calG$. By lemma \ref{uniformly finite fiber}, $M=\sup_{x\in \calG\units}\#(\rho\inv(x)\cap K)$ is finite. We have
    \begin{align*}
        \|\tilde V_{\gamma_\lambda}(\xi|_{Z_{s(\gamma_\lambda)}})-\eta|_{Z_{r(\gamma_\lambda)}}\|^2_{\ca E'_{r(\gamma_\lambda)}} & = \|\sum_{z\in Z_{r(\gamma_\lambda)}}\langle V_{\gamma_\lambda}\xi(z\gamma_\lambda)-\eta(z),V_{\gamma_\lambda}\xi(z\gamma_\lambda)-\eta(z)\rangle_{E_{r(\gamma_\lambda)}}\|_{A_{r(\gamma_\lambda)}}\\
        &\leqslant M\cdot \sup_{z\in Z_{r(\gamma_\lambda)}}\|V_{\gamma_\lambda}\xi(z\gamma_\lambda)-\eta(z)\|^2_{E_{r(\gamma_\lambda)}}\\
        & = M\cdot \|\tilde V_{\gamma_\lambda}(\xi|_{Z_{s(\gamma_\lambda)}})-\eta|_{Z_{r(\gamma_\lambda)}}\|^2_{C_0(Z_{r(\gamma_\lambda)}, E_{r(\gamma_\lambda)})}.
    \end{align*}
    By the continuity of the action of $\calG$ on $\rho_*\rho^*\ca E=\sqcup_{x\in \calG\units}C_0(Z_x, E_x)$ (see proposition \ref{continuity of action on pullback} and proposition \ref{continuity of action on pushout}), we have $\eta|_{Z_{r(\gamma_\lambda)}}\ra \eta|_{Z_{r(\gamma)}}$ and $\tilde V_{\gamma_\lambda}(\xi|_{Z_{s(\gamma_\lambda)}})\ra \tilde V_\gamma(\xi|_{Z_{s(\gamma)}})$ in the Hilbert bundle $\rho_*\rho^*\ca E$. Then $\tilde V_{\gamma_\lambda}(\xi|_{Z_{s(\gamma_\lambda)}})-\eta|_{Z_{r(\gamma_\lambda)}}\ra 0_{r(\gamma)}$ in $\rho_*\rho^*\ca E$. Then by proposition \ref{first properties of Banach bundles}, we have
    \[\|\tilde V_{\gamma_\lambda}(\xi|_{Z_{s(\gamma_\lambda)}})-\eta|_{Z_{r(\gamma_\lambda)}}\|_{C_0(Z_{r(\gamma_\lambda)}, E_{r(\gamma_\lambda)})}\ra 0,\]
    then use the inequality above, we proved our claim. Our claim yields that $\tilde V_{\gamma_\lambda}(\xi|_{Z_{s(\gamma_\lambda)}})\ra \tilde V_\gamma(\xi|_{Z_{s(\gamma)}})$ in $\ca E'$.
\end{proof}

\begin{defn}
    Let $(E,V)$ be a $\calG$-Hilbert $A$-module. We define $L^2(\calG,E)$ as the $\calG$-Hilbert $A$-module $L^2(\calG,s,E)$.
\end{defn}

\begin{rem}\label{L^2 commute with tensor product}
    We can prove that $L^2(Z,\rho,E)$ is $\calG$-equivariantly unitarily isomorphic to the inner tensor product $L^2(Z,\rho,C_0(Y))\otimes_{\Phi_E}E$.
\end{rem}

\section{Geometric modules}\label{sec geo mod}

From now on, we will use the following convention for notations about Hilbert modules: given a \cst-algebra $A$ and two Hilbert $A$-modules $E$ and $F$, a map $T:E\ra F$ is called an operator if it is a bounded $A$-linear map. We denote by $\Hom_A(E,F)$ the set of all operators from $E$ to $F$, which is a Banach space equipped with operator norm. The set of all adjointable operators $\calL(E,F)$ is a closed subspace of $\Hom_A(E,F)$. If $E=F$, $\End_A(E)=\Hom_A(E,F)$ is a Banach algebra. The subalgebra of adjointable operators $\calL(E)$ is a \cst-algebra. An operator $V\in \Hom_A(E,F)$ is called an isometry, if for any $e_1,e_2\in E$, we have $\langle V e_1, Ve_2\rangle=\langle e_1,e_2\rangle$. For a locally compact Hausdorff space $X$, an $X$-space is a pair $(Z,\rho)$, where $Z$ is a locally compact Hausdorff space and $\rho:Z\ra X$ is a continuous map.

Fixing a locally compact Hausdorff space $X$ as a base space, we introduce the following definition of geometric modules.

\begin{defn}
    Let $Z$ be an $X$-space, $A$ be a $C_0(X)$-algebra. A (geometric) $Z,X,A$-module is a pair $(E,\pi)$, where $E$ is a \textbf{countably generated} Hilbert $A$-module, $\pi:C_0(Z)\ra \calL_{A}(E)$ is a non-degenerate $C_0(X)$-representation. We denote the unique strictly continuous extension of $\pi$ by $\tilde\pi:C_b(Z)\ra \calL_A(E)$.

    For two $Z,X,A$-modules $(E_1,\pi_1)$ and $(E_2,\pi_2)$, we say that an operator $T\in \Hom_A(E,F)$ (not necessarily adjointable) intertwines $\pi_1$ and $\pi_2$, if for any $f\in C_0(Z)$, $T\circ \pi_1(f)=\pi_2(f)\circ T$. We say that $(E_1,\pi_1)$ and $(E_2,\pi_2)$ are unitarily equivalent or $(E_1,\pi_1)\cong (E_2,\pi_2)$ or $\pi_1\cong \pi_2$, if there exists a unitary element in $\calL(E_1,E_2)$ that intertwines $\pi_1$ and $\pi_2$. 
\end{defn}

\begin{defn}
    Let $\calG$ be an étale groupoid, $Z$ be a locally compact Hausdorff left $\calG$-space with anchor map $\rho:Z\ra \calG\units$, $A$ be a $\calG$-\cst-algebra. A (geometric) $Z,\calG,A$-module is a $Z,\calG\units,A$-module $(E,\pi)$ such that $E$ is a $\calG$-Hilbert $A$-module and $\pi$ is a $\calG$-equivariant representation.

    For two $Z,\calG,A$-modules $(E_1,\pi_1)$ and $(E_2,\pi_2)$, We say that $(E_1,\pi_1)$ and $(E_2,\pi_2)$ are unitarily equivalent or $(E_1,\pi_1)\cong (E_2,\pi_2)$ or $\pi_1\cong \pi_2$, if there exists a unitary element $T$ in $\calL(E_1,E_2)$ that intertwines $\pi_1$ and $\pi_2$, and $T$ is $\calG$-equivariant.
\end{defn}

Clearly, if we see $X$ as a trivial groupoid $X\rightrightarrows X$, then the two definitions above coincide.

\subsection{Operations on geometric modules}\label{section operations on geometric modules}
We fix an étale groupoid $\calG$ and a $\calG$-\cst-algebra $(A,\alpha)$. In this section, we introduce some basic operations on geometric modules.

\begin{defn}\label{operation pushout}
    Let $Y,Z$ be locally compact Hausdorff $\calG$-spaces and $f:Y\ra Z$ be a $\calG$-equivariant continuous map, $A$ be a $\calG$-\cst-algebra. For a $Y,\calG,A$-module $(E,\pi)$, we define $f_*\pi:C_0(Z)\ra \calL_A(E)$ as the composite
    \[C_0(Z)\xrightarrow{f^*} C_b(Y)\xrightarrow{\tilde\pi} \calL_A(E),\]
    where $\tilde\pi$ is the strictly continuous extension of $\pi$.
\end{defn}

The map $f_*\pi$ is also a non-degenerate $\calG$-equivariant representation, hence $(E,f_*\pi)$ is a $Z,\calG,A$-module.

\begin{defn}\label{operation braquet}
    Let $Z$ be a locally compact Hausdorff $\calG$-space, $U\subseteq Z$ be a $\calG$-invariant open subset, $i:U\hookrightarrow Z$ be the inclusion map, $A$ be a $\calG$-\cst-algebra. Let $(E,\pi)$ be a $Z,\calG,A$-module. We define $E_{[U]}=\overline{\pi(C_0(U))E}$ and
    \[\pi_{[U]}: C_0(U)\ra \calL_{A}(E_{[U]}),\]
    \[\pi_{[U]}(f)e=\pi(f)e, \quad \forall f\in C_0(U),e\in E_{[U]}.\]
\end{defn}

\begin{lem}
    Use the same notation as above, $\pi_{[U]}$ is a non-degenerate $\calG$-equivariant representation, hence $(E_{[U]},\pi_{[U]})$ is a $U,\calG,A$-module.
\end{lem}
\begin{proof}
    It is clear to identify $(E_{[U]},\pi_{[U]})$ with $(C_0(U)\otimes_{\pi}E, id\otimes_{\pi}id_E)$, which is a $U,\calG,A$-module.
\end{proof}

\begin{defn}\label{defn extraordinary pullback}
    Let $Z$ be a locally compact Hausdorff $\calG$-space, $C\subseteq Z$ be a $\calG$-invariant closed subset, $j:C\hookrightarrow Z$ be the closed inclusion map, $A$ be a $\calG$-\cst-algebra. Let $(E,\pi)$ be a $Z,\calG,A$-module. We define $j^!E=\{e\in E:\pi(C_0(Z\setminus C))e=0\}$ and
    \[j^!\pi:C_0(C)\ra \calL_A(j^!E),\]
    \[j^!\pi(f)e=\pi(\tilde f)e, \quad \forall f\in C_0(C),e\in j^!E,\]
    where $\tilde f$ is any lift of $f$ in $C_0(Z)$.
\end{defn}
\begin{lem}\label{extraordinary pullback}
    Use the same notation as above, $j^!\pi$ is a non-degenerate $\calG$-equivariant representation, hence $(j^!E,j^!\pi)$ is a $C,\calG,A$-module.
\end{lem}
\begin{proof}
    Obviously for any $f\in C_0(C)$, $j^!\pi(f)$ has adjoint $j^!\pi(\bar f)$. Fix a lift $\tilde f\in C_0(Z)$ of $f\in C_0(C)$, for $\phi\in C_0(\calG\units)$,
    \[j^!\pi(f)(\phi\cdot e)=\pi(\tilde f)(\phi\cdot e)=\pi(\tilde f\cdot \rho^*\phi)e=j^!\pi(f\cdot (\rho^*\phi)|_C)e=j^!\pi(f\cdot \rho|_C^*\phi)e.\]
    Hence, $j^!\pi$ is a well-defined $C_0(\calG\units)$-representation.

    Let $(\tilde f_\lambda)_\lambda$ be an approximate unit of $C_0(X)$, $f_\lambda=\tilde f_\lambda|_C$ is an approximate unit of $C_0(C)$. For any $e\in j^!E$, $j^!\pi(f_\lambda)e=\pi(\tilde f_\lambda)e$ converges to $e$ since $\pi$ is non-degenerate, hence $j^!\pi$ is also non-degenerate.

    Let $V\in \calL(s^*E,r^*E)$ be the action of $\calG$ on $E$. Since $j^!E$ is a closed submodule of $E$, for any $x\in\calG\units$, its fiber $(j^!E)_x$ can be identified with $\{e(x):e\in j^!E\}\subseteq E_x$ by proposition \ref{Hilbert bundle of submodule}. It is easy to see that, for $e\in E$, $e\in j^!E$ if and only if for any $x$, $\pi_x(C_0((Z\setminus C)_x))e(x)=0$, so $(j^!E)_x$ is exactly $\{e\in E_x:\pi_x(C_0((Z\setminus C)_x))e=0\}$. For any $\gamma\in \calG$, $V_\gamma$ is a unitary from $E_{s(\gamma)}$ to $E_{r(\gamma)}$, and for any $f\in C_0(Z)$,
    \[V_\gamma\circ \pi_{s(\gamma)}(f)=\pi_{r(\gamma)}(f(\gamma\inv-))\circ V_\gamma.\] 
    So for any $e\in (j^!E)_{s(\gamma)}$, 
    \[\pi_{r(\gamma)}(C_0((Z\setminus C)_{r(\gamma)}))V_\gamma e=V_\gamma \pi_{s(\gamma)}(C_0((Z\setminus C)_{s(\gamma)}))e=0,\] 
    hence $V_\gamma e\in (j^!E)_{r(\gamma)}$. This shows that the action of $\calG$ on $E$ restricts to an action of $\calG$ on $j^!E$, so $j^!E$ is a well-defined $\calG$-Hilbert $A$-module.

    For any $\gamma\in \calG$, $f\in C_0(C_{s(\gamma)})$, $e\in (j^!E)_{s(\gamma)}\subseteq E_{s(\gamma)}$, $\tilde f\in C_0(Z_{s(\gamma)})$ be a lift of $f$, and hence $\tilde f(\gamma\inv-)\in C_0(Z_{r(\gamma)})$ is a lift of $f(\gamma\inv-)\in C_0(C_{r(\gamma)})$, so we have
    \begin{align*}
        (j^!\pi)_{r(\gamma)}(f(\gamma\inv-))V_\gamma e & = \pi_{r(\gamma)}(\tilde f(\gamma\inv-))V_\gamma e\\
        & = V_\gamma \pi_{s(\gamma)}(\tilde f)e\\
        & = V_\gamma (j^!\pi)_{s(\gamma)}(f)e.
    \end{align*}
    The representation $j^!\pi$ is $\calG$-equivariant.
\end{proof}

\begin{defn}\label{defn A_C}
    Let $X$ be a locally compact Hausdorff space, $A$ be a $C_0(X)$-algebra. For a closed subset $C$ of $X$, we define
    \[A_{\lgl C\rgl}=\{a\in A:C_0(X\setminus C)\cdot a=0\},\]
    which is a closed two-sided ideal of $A$.
\end{defn}

\begin{ex}\label{canonical module for closed inclusion}
    Let $\Gamma$ be a discrete group, $X$ be a locally compact Hausdorff $\Gamma$-space, $C\subseteq X$ be a $\Gamma$-invariant closed subset, $A$ be a $\Gamma\ltimes X$-algebra, $\phi:C_0(X)\ra ZM(A)$ be the $C_0(X)$-algebra structure of $A$. Let $j:C\hookrightarrow X$ be the closed inclusion map, then $(A_{\lgl C\rgl}, j^!\phi)$ is a $C,\Gamma\ltimes X,A$-module, where $j^!\phi:C_0(C)\ra ZM(A_{\lgl C\rgl})$ is the $C_0(C)$-algebra structure of $A_{\lgl C\rgl}$.
\end{ex}

The following result seems trivial but shows that, for a $Z,X,A$-module $(E,\pi)$, the $C_0(X)$-linearity of $\pi$ will force $E$ to be supported in the closure of $\rho(Z)$.

\begin{prop}\label{C_0(X)-linearity force base to shrink}
    Let $(E,\pi)$ be a $Z,X,A$-module. Let $F$ be a closed subset of $X$ containing $\rho(Z)$. Then $\lgl E,E\rgl\subseteq A_{\lgl F\rgl}$.
\end{prop}
\begin{proof}
    For any $f\in C_0(X\setminus F)$, we have $\rho^*f=0$. For any $e_1,e_2\in E$, by \cite[lemma 1.3]{blanchard1996deformations}, there exists $e'\in E$ and $a\in A$ such that $e_2=e'a$ and therefore $\lgl e_1,e_2\rgl=\lgl e_1,e'\rgl a$. So for any $f\in C_0(X\setminus F)$,
    \begin{align*}
        f\cdot(\lgl e_1,e_2\rgl) & =f\cdot(\lgl e_1,e'\rgl a)\\
        & = \lgl e_1,e'\rgl f\cdot a\\
        & = \lgl e_1,e'(f\cdot a)\rgl\\
        & = \lgl e_1, \pi(\rho^*f)(e'a)\rgl=0.
    \end{align*}
    $\lgl E,E\rgl$ is the closed span of elements like $\lgl e_1,e_2\rgl$. Therefore, $\lgl E,E\rgl\subseteq A_{\lgl F\rgl}$.
\end{proof}

Remark that, $C_0(X)_{\lgl F\rgl}=C_0(\inter(F))$. So if $(E,\pi)$ is an $F,X$-module (where the anchor map of $F$ is the closed inclusion), $E$ must be a Hilbert $C_0(\inter(F))$-module.

\begin{defn}
    Let $Z$ be a locally compact Hausdorff $\calG$-space with anchor map $\rho:Z\ra \calG\units$, $Y$ be a locally compact Hausdorff $\calG$-space with anchor map $\phi:Y\ra \calG\units$. Let $A$ be a $\calG$-\cst-algebra. Let $(E,\pi)$ be a $Z,\calG,A$-module. We define the pullback of $(E,\pi)$ by $\phi$ as the $Z\times_{\calG\units}Y,\calG\ltimes Y,\phi^*A$-module $(\phi^*E,\phi^*\pi)$, where $\phi^*E$ is the pullback of $E$ by $\phi$ and $\phi^*\pi$ is the pullback of $\pi$ by $\phi$. (See proposition \ref{G-equivariance of pullback}.)
\end{defn}

\begin{defn}\label{defn stable modules}
    Let $Z$ be a locally compact Hausdorff $\calG$-space, $(E,\pi)$ be a $Z,\calG,A$-module. We denote the $Z,\calG,A$-module $(E\otimes\ell^2(\mathbb N), \pi\otimes id)$ by $(E^\infty,\pi^\infty)$. We say that $(E,\pi)$ is stable, if $(E,\pi)$ is $\calG$-equivariantly unitarily equivalent to $(E^\infty, \pi^\infty)$.
\end{defn}

\subsection{Inductions and inflations of geometric modules}\label{section Inflations of geometric modules}

We recall the notion of induction functors from \'etale groupoid correspondences, introduced by Miller in \cite{miller2024functors}.

\begin{defn}
    Let $\calG,\calH$ be two étale groupoids. A $\calG,\calH$-bibundle is a locally compact Hausdorff space $\Omega$ equipped with a left $\calG$-action and a right $\calH$-action, such that the two actions commute. If the right action of $\calH$ on $\Omega$ is free, proper and étale, then we say that $\Omega:\calG\leftarrow \calH$ is an étale groupoid correspondence or a correspondence.

    We denote the anchor map of the left $\calG$-action by $\rho_\Omega:\Omega\ra \calG\units$, and the anchor map of the right $\calH$-action by $\sigma_\Omega:\Omega\ra \calH\units$. For any $x\in \calG\units$, $y\in \calH\units$, $A\subseteq \calG\units$ and $B\subseteq\calH\units$, we write
\[\Omega^x=\rho_\Omega\inv(x),\quad \Omega_y=\sigma_\Omega\inv(y), \quad \Omega^A=\rho_\Omega\inv(A),\quad \Omega_B=\sigma_\Omega\inv(B).\]
\end{defn}

\begin{defn}[Balanced products]
    Let $\calH$ be an étale groupoid, $(Z,\rho_Z)$ be a left $\calH$-space and $(Y,\rho_Y)$ be a right $\calH$-space. We define the balanced product $Y\times_\calH Z$ as the quotient space of $Y\times_{\rho_Y,\calH\units,\rho_Z}Z$ by the equivalence relation generated by $(yh,z)\sim (y,hz)$ for all $z\in Z$, $y\in Y$ and $h\in \calH$ with $\rho_Z(z)=s(h)$ and $\rho_Y(y)=r(h)$.
\end{defn}

When $Y$ is locally compact Hausdorff and $Z$ is a locally compact Hausdorff proper $\calH$-space, the balanced product $Y\times_\calH Z$ is also a locally compact Hausdorff space.

\begin{defn}
    Let $\Omega:\calG\leftarrow \calH$ be a correspondence, $\ca A\ra \calH\units$ be a Banach $\calH$-bundle. We define the induction of $\ca A$ by $\Omega$ as the Banach bundle $\ind_\Omega \ca A:=\Omega\times_{\calH}\ca A\ra \Omega/\calH,[\omega,a]\mapsto \omega\calH$.
\end{defn}

\begin{rem}
    To see that it is a well-defined Banach bundle, we need to see $\Omega\times_\calH \ca A$ as the quotient of the pullback bundle $\Omega\times_{\calH\units}\ca A\ra \Omega$ by the action of the principal proper \'etale groupoid $\Omega\rtimes\calH$, and apply \cite[Proposition 1.86]{miller2022k}.
\end{rem}

\begin{defn}
    Let $(A,\alpha)$ be an $\calH$-\cst-algebra. We define $\ind_\Omega \ca A$ as the closed *-subalgebra of $\Gamma_b(\Omega,\sigma^* \ca A)$ consisting of continuous sections $\xi:\Omega\ra \sigma^*\calA$ such that
\begin{enumerate}
    \item for any $(\omega,h)\in \Omega\times_{\sigma,\calH\units, r_\calH}\calH$, $\xi(\omega h)=\alpha_{h\inv}(\xi(\omega))$,
    \item the map $\Omega/\calH\ra \mathbb R: \omega\calH\mapsto \|\xi(\omega)\|$ vanishes at infinity.
\end{enumerate}
For $\xi\in \ind_\Omega A$, we define $\supp(\xi)=\overline{\{\omega\calH\in \Omega/\calH:\xi(\omega)\neq 0\}}$. We denote the subspace of compactly supported elements of $\ind_\Omega A$ by $\ind_{\Omega,c}A$.
\end{defn}

Clearly, $\ind_\Omega A\cong \Gamma_0(\Omega/\calH,\ind_\Omega \ca A)$, and under this isomorphism, $\ind_{\Omega,c}A$ is in bijection with $\Gamma_c(\Omega/\calH,\ind_\Omega \ca A)$. There is a $C_0(\Omega/\calH)$-algebra structure on $\ind_\Omega A$, given by
\[C_0(\Omega/\calH)\ra ZM(\ind_\Omega A), \quad (f\cdot \xi)(\omega)=f(\omega\calH)\xi(\omega),\quad \forall f\in C_0(\Omega/\calH), \xi\in \ind_\Omega A, \omega\in \Omega.\]
And there is a $C_0(\calG\units)$-algebra structure on $\ind_\Omega A$, given by
\[C_0(\calG\units)\ra ZM(\ind_\Omega A),\quad (g\cdot \xi)(\omega)=g(\rho(\omega))\xi(\omega),\quad \forall g\in C_0(\calG\units), \xi\in \ind_\Omega A, \omega\in \Omega.\]

The fiber of $\ind_\Omega A$ at $x\in \calG\units$ can be identified with the *-subalgebra $\ind_{\Omega^x} A$ of $\Gamma_b(\Omega^x,\sigma^* \ca E)$ consisting of continuous sections $\xi:\Omega^x\ra \sigma^*\ca A$ such that
\begin{enumerate}
    \item for any $(\omega,h)\in \Omega^x\times_{\sigma,\calH\units, r_\calH}\calH$, $\xi(\omega h)=\alpha_{h\inv}(\xi(\omega))$,
    \item $\Omega^x/\calH\ra \mathbb R:\omega\calH\mapsto \|\xi(\omega)\|$ vanishes at infinity.
\end{enumerate}
For every $\gamma\in\calG$, let $\beta_\gamma$ be the isomorphism
\[\beta_\gamma:\ind_{\Omega^{s(\gamma)}}A\ra \ind_{\Omega^{r(\gamma)}}A, \xi\mapsto \xi(\gamma\inv-).\]
$\beta=(\beta_\gamma)_{\gamma\in \calG}$ makes $\ind_\Omega A$ a $\calG$-\cst-algebra. We will write $\gamma.\xi=\xi(\gamma\inv-)\in \ind_{\Omega^{r(\gamma)}}A$.

\begin{defn}
    Let $(E,V)$ be an $\calH$-Hilbert $A$-module. We define $\ind_\Omega E$  as the closed subspace of $\Gamma_b(\Omega,\sigma^*\ca E)$ consisting of continuous sections $\eta:\Omega\ra \sigma^* \ca E$ such that
    \begin{enumerate}
    \item for any $(\omega,h)\in \Omega\times_{\sigma,\calH\units, r_\calH}\calH$, $\eta(\omega h)=V_{h\inv}(\eta(\omega))$,
    \item $\Omega/\calH\ra \mathbb R: \omega\calH\mapsto \|\eta(\omega)\|$ vanishes at infinity.
\end{enumerate}

We define the Hilbert $\ind_\Omega A$-module structure on $\ind_\Omega E$ as,
\[(\eta\xi)(\omega)=\eta(\omega)\xi(\omega),\quad \forall \eta\in \ind_\Omega E, \xi\in \ind_\Omega A, \omega\in \Omega,\]
\[\langle \eta_1,\eta_2\rangle(\omega)=\langle \eta_1(\omega),\eta_2(\omega)\rangle_{E_{\sigma(\omega)}}, \quad \forall \eta_1,\eta_2\in \ind_\Omega E, \omega\in \Omega.\]
For $\eta\in \ind_\Omega E$, we define $\supp(\eta)=\overline{\{\omega\calH\in \Omega/\calH:\eta(\omega)\neq 0\}}$. We denote the subspace of compactly supported elements of $\ind_\Omega E$ by $\ind_{\Omega,c}E$.
\end{defn}

It is also easy to see that $\ind_\Omega E\cong \Gamma_0(\Omega/\calH,\ind_\Omega\ca E)$, and under this isomorphism, $\ind_{\Omega,c}E$ is in bijection with $\Gamma_c(\Omega/\calH,\ind_\Omega \ca E)$.

For $x\in \calG\units$, we can similarly define $\ind_{\Omega^x}E$. The fiber of $\ind_\Omega E$ at $x$ can be identified with $\ind_{\Omega^x}E$ and the $\calG$-action on $\ind_\Omega E$ is given by the unitary isomorphisms
\[W_\gamma:\ind_{\Omega^{s(\gamma)}}E\ra \ind_{\Omega^{r(\gamma)}}E, \eta\mapsto \eta(\gamma\inv-).\]
$W=(W_\gamma)_{\gamma\in \calG}$ makes $\ind_\Omega E$ a $\calG$-Hilbert $\ind_\Omega A$-module. All details for the above content can be found in \cite[Section 3]{bonicke2020going} and \cite[Section 2]{miller2022k}.

\begin{prop}\label{operators on induced modules}
    \textnormal{\cite[Proposition 4.11, Proposition 4.12]{miller2024functors}}
    Let $\Omega:\calG\leftarrow \calH$ be a correspondence, $A$ be an $\calH$-\cst-algebra and $E$ be an $\calH$-Hilbert $A$-module. Then
    \begin{enumerate}
        \item The map
        \[\calK(\ind_\Omega E)\ra \ind_\Omega \calK(E),\]
        \[T\mapsto [\omega\mapsto T_\omega]\]
        is an isomorphism.
        \item The map
        \[\calL(\ind_\Omega E)\ra \Gamma_b(\Omega, \calL(\sigma_\Omega^*\ca E)),\]
        \[T\mapsto [\omega\mapsto T_\omega]\]
        is an isometry, whose image is the $\calH$-equivariant sections.
    \end{enumerate}
\end{prop}

\begin{defn}\label{induction of representation}
    Let $A,B$ be $\calH$-\cst-algebras, $E$ be an $\calH$-Hilbert $B$-module, $\pi:A\ra \calL(E)$ be an $\calH$-equivariant representation, we define the representation $\ind_\Omega \pi$ as the $\calG$-equivariant representation $\ind_\Omega\pi:\ind_\Omega A\ra \calL(\ind_\Omega E)$ such that, for every $\xi\in \ind_\Omega A$, $\eta\in \ind_\Omega E$, $\omega\in \Omega$,
\[((\ind_\Omega\pi)(\xi)\eta)(\omega)=\pi_{\sigma(\omega)}(\xi(\omega))\eta(\omega).\]
\end{defn}

For the correspondences associated to relatively clopen subgroupoids, we recovered B\"onickes' subgroupoid induction functor.

\begin{defn}
    \textnormal{\cite[Definition 2.2]{oyono2023groupoids}}
    For a locally compact Hausdorff groupoid $\calG$ and an open subgroupoid $\calH\subseteq \calG$, we say that $\calH$ is relatively clopen in $\calG$, if $\calH$ is clopen in $\calG_{\calH\units}$.
\end{defn}

\begin{prop}
    \textnormal{\cite[Lemma 1.21, Corollary 1.22]{Mao26}}
    For an open subgroupoid $\calH\subseteq\calG$, the right action of $\calH$ on $\calG_{\calH\units}$ by multiplication is proper if and only if $\calH$ is relatively clopen in $\calG$. Especially, all proper open subgroupoids of $\calG$ is relatively clopen in $\calG$.
\end{prop}

In other words, $\calH$ is relatively clopen in $\calG$ is a necessary and sufficient condition for $\calG_{\calH\units}:\calG\leftarrow \calH$ to become a well-defined correspondence. 

Now we will introduce an operation called inflation, motivated by the inflation map in \cite[Section 6]{bonicke2020going}, which is constructed as the inverse of the compression isomorphism.

\begin{defn}\label{defn of inflation}
    Let $\calG$ be an étale groupoid, $\calH$ be a relatively clopen subgroupoid of $\calG$, $\Omega=\calG_{\calH\units}$, $(B,\beta)$ be a $\calG$-\cst-algebra. For an $\calH$-Hilbert $B|_\calH$-module $(E,V)$, we can define a pre-Hilbert $B$-module structure on $\ind_{\Omega,c}E$ as
    \[(\eta\cdot b)(\gamma)=\eta(\gamma)\beta_{\gamma\inv}(b(r(\gamma))),\quad \forall \eta\in \ind_{\Omega,c}E, b\in B,\gamma\in \Omega,\]
    \[\langle \eta_1,\eta_2\rangle(x)=\sum_{\omega\calH\in \Omega^{x}/\calH}\beta_\omega(\langle \eta_1(\omega),\eta_2(\omega)\rangle),\quad \forall \eta_1,\eta_2\in \ind_{\Omega,c}E, x\in \calG\units. \]
    We will denote the completion as $\infl^\calG_\calH E$, which is a Hilbert $B$-module.
\end{defn}

Notice that, let $\tilde r:\Omega/\calH\ra \calG\units$ be the map $\omega\calH\mapsto r(\omega)$. Since the map $r|_\Omega$ and the quotient map $\Omega\ra \Omega/\calH$ are local homeomorphisms (see \cite[Lemma 2.12, Proposition 2.19]{antunes2021bicategory}), so $\tilde r$ is also a local homeomorphism. Then by lemma \ref{sum of fiber}, this $B$-valued inner product is well-defined. We can check that $\langle \eta_1,\eta_2\cdot b\rangle=\langle \eta_1,\eta_2\rangle b$ for any $\eta_1,\eta_2\in \ind_{\Omega,c}E$ and $b\in B$.

\begin{prop}\label{inflation is induction inner tensor product a module}
    Use the same notation as above, let $\psi:\ind_\Omega(B|_\calH)\ra \calL(L^2(\Omega/\calH,\tilde r, B))$ be the map defined by
    \[[\psi(\xi)\eta](\omega\calH)=\beta_\omega(\xi(\omega))\eta(\omega\calH),\quad \forall \xi\in \ind_\Omega(B|_\calH), \eta\in \Gamma_c(\Omega/\calH,\tilde r^*\ca B), \omega\calH\in \Omega/\calH.\]
    \begin{enumerate}
        \item The map $\psi$ is a well-defined representation. Moreover, $\psi(\ind_\Omega(B|_\calH))\subseteq \calK(L^2(\Omega/\calH,\tilde r, B))$.
        \item There is a unitary isomorphism of Hilbert $B$-modules $\infl^\calG_\calH E \cong \ind_\Omega E\otimes_\psi L^2(\Omega/\calH,\tilde r,B)$.
    \end{enumerate}
\end{prop}
\begin{proof}
    (1) It is easy to check that $\beta_\omega(\xi(\omega))\eta(\omega\calH)$ does not depend on the choice of $\omega$. The map $\Omega/\calH\ra \tilde r^*\ca B, \omega\calH\mapsto \beta_\omega(\xi(\omega))\eta(\omega\calH)$ is also an element of $\Gamma_c(\Omega/\calH,\tilde r^*\ca B)$. So $\psi$ is a well-defined representation.

    If $\xi\in \ind_{\Omega,c}(B|_\calH)$, then $\supp(\xi)$ is a compact subset of $\Omega/\calH$. We will say that $\xi$ has small enough support, if there exists an open neighborhood $V$ of $\supp(\xi)$ such that $\tilde r|_V$ is a homeomorphism onto an open of $\calG\units$. If $\xi$ has small enough support, $\omega\calH\mapsto \beta_\omega(\xi(\omega))$ is an element of $\Gamma_c(\Omega/\calH,\tilde r^*B)$, and hence $\psi(\xi)$ is a rank one operator. By a partition of unity, all elements of $\ind_{\Omega,c}(B|_\calH)$ can be written as a finite sum of elements that have small enough supports. So we proved that $\psi(\ind_\Omega(B|_\calH))\subseteq \calK(L^2(\Omega/\calH,\tilde r, B))$.

    (2) For any $\eta\in \ind_{\Omega,c}E\subseteq \ind_\Omega E$ and $\xi\in \Gamma_c(\Omega/\calH,\tilde r^*\ca B)$, for any $\omega\in \Omega$, we define
    \[T(\eta\otimes \xi)(\omega)=\eta(\omega)\beta_{\omega\inv}(\xi(\omega\calH))\in E_{s(\omega)}.\]
    Then $T(\eta\otimes \xi)$ is an element of $\Gamma_b(\Omega,s^*\ca E)$. Now if $(\omega, h)\in \Omega\times_{s,\calH\units,r}\calH$, we have
    \begin{align*}
        T(\eta\otimes \xi)(\omega h) & = \eta(\omega h)\beta_{h\inv \omega\inv}(\xi(\omega\calH))\\
        & = V_{h\inv}(\eta(\omega))\beta_{h\inv}(\beta_{\omega\inv}(\xi(\omega\calH)))\\
        & = V_{h\inv}(T(\eta\otimes \xi)(\omega)),
    \end{align*}
    and clearly the map $\omega\calH\mapsto \|T(\eta\otimes\xi)(\omega)\|$ is also compactly supported. So $T(\eta\otimes \xi)$ is a well-defined element of $\ind_{\Omega,c}E$. We will prove that $T$ extends to a unitary isomorphism.

    For any $\eta_1,\eta_2\in \ind_{\Omega,c}E$ and $\xi_1,\xi_2\in \Gamma_c(\Omega/\calH, \tilde r^*\ca B)$, $x\in \calG\units$,
    \begin{align*}
        \langle T(\eta_1\otimes \xi_1),T(\eta_2\otimes\xi_2)\rangle_{\infl^\calG_\calH E}(x) & = \sum_{\omega\calH\in \Omega^x/\calH}\beta_\omega(\langle T(\eta_1\otimes \xi_1)(\omega),T(\eta_2\otimes\xi_2)(\omega)\rangle)\\
        & = \sum_{\omega\calH\in \Omega^x/\calH}\beta_\omega(\langle \eta_1(\omega)\beta_{\omega\inv}(\xi_1(\omega\calH)), \eta_2(\omega)\beta_{\omega\inv}(\xi_2(\omega\calH))\rangle)\\
        & = \sum_{\omega\calH\in \Omega^x/\calH}\xi_1(\omega\calH)^* \beta_\omega(\langle \eta_1(\omega), \eta_2(\omega))\rangle) \xi_2(\omega\calH)\\
        & = \sum_{\omega\calH\in \Omega^x/\calH} \xi_1(\omega\calH)^*[\psi(\langle \eta_1,\eta_2\rangle_{\ind_\Omega E})\xi_2](\omega\calH)\\
        & = \langle \xi_1,\psi(\langle \eta_1,\eta_2\rangle_{\ind_\Omega E})\xi_2\rangle_{L^2(\Omega/\calH,\tilde r, B)}(x)\\
        & = \langle \eta_1\otimes \xi_1,\eta_2\otimes \xi_2\rangle_{\ind_\Omega E\otimes_\psi L^2(\Omega/\calH,\tilde r,B)}(x).
    \end{align*}

    Claim: $\ind_{\Omega,c}E$ is in range of $T$. If $\eta\in \ind_{\Omega,c}E$, by \cite[Lemma 1.3]{blanchard1996deformations}, there exists $\eta'\in \ind_\Omega E$ such that $\eta=\eta'\langle \eta',\eta'\rangle_{\ind_\Omega E}$ and clearly $\eta'$ should be in $\ind_{\Omega,c}E$. We define $\xi:\Omega/\calH\ra \tilde r^*\ca B$ as $\xi(\omega\calH)= \beta_\omega(\langle \eta',\eta'\rangle_{\ind_\Omega E}(\omega))$. We can check easily that $\xi\in \Gamma_c(\Omega/\calH,\tilde r^*\ca B)$, and $\eta=T(\eta'\otimes \xi)$. And this implies that $T$ has dense range in $\infl^\calG_\calH E$.

    So in conclusion, there is a unitary isomorphism $\ind_\Omega E\otimes_\psi L^2(\Omega/\calH,\tilde r, B)\cong \infl^\calG_\calH E$.
\end{proof}

Use the same notation as above, for any $x\in \calG\units$, the fiber of $\infl^\calG_\calH E$ at $x$ can be described as the completion of the pre-Hilbert $B_x$-module $\ind_{\Omega^x,c}E$, where its pre-Hilbert module structure is characterized by
\[(\eta\cdot b)(\gamma)=\eta(\gamma)\beta_{\gamma\inv}(b),\quad \forall \eta\in \ind_{\Omega^x,c}E, b\in B_x,\gamma\in \Omega^x,\]
\[\langle \eta_1,\eta_2\rangle=\sum_{\omega\calH\in \Omega^{x}/\calH}\beta_\omega(\langle \eta_1(\omega),\eta_2(\omega)\rangle),\quad \forall \eta_1,\eta_2\in \ind_{\Omega^x,c}E. \]
And we have a unitary isomorphism $(\infl^\calG_\calH E)_x\cong \ind_{\Omega^x}E\otimes_{\psi_x} \ell^2(\Omega^x/\calH,B_x)$.

There is a $\calG$-Hilbert module structure of $\ind_\Omega E$ is given by the unitary isomorphisms
\[W_\gamma: \ind_{\Omega^{s(\gamma)}}E\ra \ind_{\Omega^{r(\gamma)}}E, \quad \xi\mapsto \xi(\gamma\inv-),\]
and using the proposition \ref{L^2 has continuous G-action}, there is a $\calG$-Hilbert module structure of $L^2(\Omega/\calH,\tilde r,B)$ given by the unitary isomorphisms
\[\tilde V_\gamma: \ell^2(\Omega^{s(\gamma)}/\calH, B_{s(\gamma)})\mapsto \ell^2(\Omega^{r(\gamma)}/\calH,B_{r(\gamma)}),\xi\mapsto \beta_\gamma(\xi(\gamma\inv-)).\]
By proposition \ref{continuity of action on inner tensor product}, $W\otimes_{s^*\psi} \tilde V$ gives a $\calG$-Hilbert module structure on $\ind_\Omega E\otimes_\psi L^2(\Omega/\calH,\tilde r, B)$. After identifying this inner tensor product with $\infl^\calG_\calH E$ by proposition \ref{inflation is induction inner tensor product a module}, the $\calG$-Hilbert module structure on $\infl^\calG_\calH E$ is given by unitary isomorphisms $V^{\infl}_\gamma: (\infl^\calG_\calH E)_{s(\gamma)}\ra (\infl^\calG_\calH E)_{r(\gamma)}$, which is the extension of the isometry
\[\ind_{\Omega^{s(\gamma)},c}E\ra \ind_{\Omega^{r(\gamma)},c}E, \quad \xi\mapsto \xi(\gamma\inv-),\]
since for any $\gamma\in \calG$, $\eta\in \ind_{\Omega^{s(\gamma)},c}E\subseteq \ind_{\Omega^{s(\gamma)}}E$, $\xi\in C_c(\Omega^{s(\gamma)}/\calH,B_{s(\gamma)})\subseteq \ell^2(\Omega^{s(\gamma)}/\calH,B_{s(\gamma)})$, $\omega\in \Omega^{r(\gamma)}$, we have
    \begin{align*}
        [T_{r(\gamma)}(W_\gamma(\eta)\otimes \tilde V_\gamma(\xi))](\omega) & = W_\gamma(\eta)(\omega)\beta_{\omega\inv}(\tilde V_\gamma(\xi)(\omega))\\
        & = \eta(\gamma\inv \omega)\beta_{\omega\inv}(\beta_\gamma(\xi(\gamma\inv\omega)))\\
        & = [T_{s(\gamma)}(\eta\otimes \xi)](\gamma\inv \omega).
    \end{align*}

\begin{defn}\label{defn inflation of representation}
    Let $\calG$ be an étale groupoid, $\calH$ be a relatively clopen subgroupoid of $\calG$, $\Omega=\calG_{\calH\units}$, $(A,\alpha)$ be an $\calH$-\cst-algebra, $(B,\beta)$ be $\calG$-\cst-algebras, $(E,V)$ be an $\calH$-Hilbert $B|_\calH$-module, $\pi:A\ra \calL(E)$ be an $\calH$-equivariant representation. Then we define $\infl^\calG_\calH\pi: \ind_\Omega A\ra \calL_B(\infl^\calG_\calH E)$ as, for any $\xi\in \ind_\Omega A$, $\eta\in \ind_{\Omega,c}E\subseteq \infl^\calG_\calH E$, $\omega\in \Omega$,
    \[[(\infl^\calG_\calH\pi)(\xi)\eta](\omega)=\pi_{s(\omega)}(\xi(\omega))\eta(\omega).\]

    Let $(E',V')$ be also an $\calH$-Hilbert $B|_\calH$-module, for any $T\in \Hom_A(E,E')$ which is $\calH$-equivariant, we define $\infl^\calG_\calH(T)\in \Hom_A(\infl^\calG_\calH E, \infl^\calG_\calH E')$ as
    \[(\infl^\calG_\calH(T)\xi)(\gamma)=T_{s(\gamma)}\xi(\gamma),\]
    for any $\gamma\in \Omega$, $\xi\in \ind_{\Omega,c}E\subseteq \infl^\calG_\calH E$.
\end{defn}

\begin{prop}\label{induction and inflation}
    Use the same notation as above.
    \begin{enumerate}
        \item The representation $\infl^\calG_\calH\pi$ is $\calG$-equivariant.
        \item Let $\psi:\ind_{\Omega}(B|_\calH)\ra \calL(L^2(\Omega/\calH,\tilde r, B))$ be the representation defined in proposition \ref{inflation is induction inner tensor product a module}, then the $\ind_\Omega A,B$-bimodule $(\infl^\calG_\calH E,\infl^\calG_\calH\pi)$ is equivariantly unitarily equivalent to the bimodule $(\ind_\Omega E\otimes_\psi L^2(\Omega/\calH,\tilde r,B),\ind_\Omega \pi\otimes_\psi id)$.
    \end{enumerate}
\end{prop}
\begin{proof}
    (1) For any $\gamma\in \calG$, $\xi\in \ind_{\Omega^{s(\gamma)}}A$ and $\eta\in \ind_{\Omega^{r(\gamma)},c}E\subseteq (\infl^\calG_\calH E)_{r(\gamma)}$, $\omega\in \Omega^{r(\gamma)}$,
    \begin{align*}
        [(\infl^\calG_\calH\pi)_{r(\gamma)}(\gamma.\xi) \eta](\omega) & = \pi_{s(\omega)}((\gamma.\xi)(\omega))\eta(\omega)\\
        & = \pi_{s(\omega)}(\xi(\gamma\inv\omega))\eta(\omega)\\
        & = \pi_{s(\gamma\inv\omega)}(\xi(\gamma\inv\omega))(V^{\infl}_{\gamma\inv}\eta)(\gamma\inv\omega)\\
        & = [V^{\infl}_\gamma\circ (\infl^\calG_\calH\pi)_{s(\gamma)}(\xi)\circ V^{\infl}_{\gamma\inv}(\eta)](\omega).
    \end{align*}
    That is $(\infl^\calG_\calH\pi)_{r(\gamma)}(\gamma.-)=V^{\infl}_\gamma\circ (\infl^\calG_\calH\pi)_{s(\gamma)}(-)\circ V^{\infl}_{\gamma\inv}$. So $\infl^\calG_\calH\pi$ is $\calG$-equivariant.

    (2) Let $T$ be the unitary isomorphism defined as in the proof of proposition \ref{inflation is induction inner tensor product a module}. For any $\eta\in \ind_{\Omega,c}E\subseteq \ind_\Omega E$, $\xi\in \Gamma_c(\Omega/\calH,\tilde r^*\ca B)\subseteq L^2(\Omega/\calH,\tilde r, B)$, $f\in \ind_\Omega A$, $\omega\in \Omega$, we have
    \begin{align*}
        [T\circ (\ind_\Omega\pi\otimes_\psi id)(f)(\eta\otimes \xi)](\omega) & = T(\ind_\Omega\pi(f)\eta \otimes \xi)(\omega)\\
        & = \pi_{s(\omega)}(f(\omega))\eta(\omega) \beta_{\omega\inv}(\xi(\omega\calH))\\
        & = \pi_{s(\omega)}(f(\omega))([T(\eta\otimes\xi)](\omega))\\
        & = [(\infl^\calG_\calH\pi)(f)\circ T (\eta\otimes\xi)](\omega).
    \end{align*}
    That is, $T\circ (\ind_\Omega\pi\otimes_\psi id)(-)= (\infl^\calG_\calH\pi)(-)\circ T$. So $T$ intertwines $\infl^\calG_\calH\pi$ and $\ind_\Omega\pi\otimes id$. By the definition of $V^\infl$, $T$ is $\calG$-equivariant.
\end{proof}

\begin{defn}\label{defn operation inflation}
    Let $\calH$ be a relatively clopen subgroupoid of a second countable étale groupoid $\calG$, $Z$ be a locally compact Hausdorff $\calH$-space, $B$ be a separable $\calG$-\cst-algebra, $(E,\pi)$ be a $Z,\calH,B|_\calH$-module. We define the inflation of $(E,\pi)$ as the $\calG_{\calH\units}\times_{\calH}Z,\calG,B$-module $(\infl^\calG_\calH E, \infl^\calG_\calH \pi)$.
\end{defn}

\begin{rem}\label{details on inflation map}
    Use the same notation as above, let $\Omega=\calG_{\calH\units}$, $\ind_\Omega E$ is countably generated by \cite[Proposition 2.36]{Mao26}, hence $\infl^\calG_\calH E\cong \ind_\Omega E\otimes_{\psi}L^2(\Omega/\calH,\tilde r, B)$ is also countably generated. One can prove that, if $\pi$ is non-degenerate, then $\ind_\Omega \pi$ is also non-degenerate, and hence so is $\infl^\calG_\calH\pi$.
\end{rem}

\subsection{Coarse structures}\label{section coarse structures}

Fix a locally compact Hausdorff space $X$. In our setting, we need a relative coarse structure over $X$.

\begin{defn}
    Let $Z_1,Z_2,Z_3$ be three $X$-spaces, for any $E\subseteq Z_3\times_X Z_2$, $F\subseteq Z_2\times_X Z_1$, $A\subseteq Z_1$, $B\subseteq Z_2$, we define
    \[F\inv=\{(z_1,z_2)\in Z_1\times_X Z_2:(z_2,z_1)\in F\},\]
    \[E\circ F:=\{(z_3,z_1)\in Z_3\times_X Z_1:\exists z_2\in Z_2, (z_3,z_2)\in E \text{ and } (z_2,z_1)\in F\},\]
    \[F\circ A=\{z_2\in Z_2:\exists z_1\in A, (z_2,z_1)\in F\}\]
    \[B\circ F=\{z_1\in Z_1:\exists z_2\in B, (z_2,z_1)\in F\}.\]

    When there is no ambiguity, we denote by $r$ the projection $pr_{Z_2}:Z_2\times_X Z_1\ra Z_2$ and by $s$ the projection $pr_{Z_1}:Z_2\times_X Z_1\ra Z_1$. So $F\circ A=r(F\cap s\inv(A))$ and $B\circ F=s(F\cap r\inv(B))$.
\end{defn}

\begin{rem}
    It is easy to see that the operation $\circ$ is associative, and
    \begin{enumerate}
        \item $(F\inv)\inv=F$,
        \item $A\circ F\inv= F\circ A, F\inv\circ B=B\circ F$,
        \item $s(F)=r(F\inv), r(F)=s(F\inv)$.
    \end{enumerate}
    
    And in case that $Z_1=Z_2$, the notations above are compatible with the pair groupoid $Z_1\times_X Z_1\rightrightarrows Z_1$.
\end{rem}

\begin{defn}
    Let $Z$ be an $X$-space. An $X$-coarse structure on $Z$ is a collection $\ca E$ of subsets of $Z\times_{X} Z$, satisfying the following axioms.
    \begin{enumerate}
        \item If $E\in \ca E$, $F\subseteq E$, then $F\in \ca E$;
        \item If $E,F\in \ca E$, then $E\inv\in \ca E$, $E\circ F\in \ca E$ and $E\cup F\in \ca E$.
    \end{enumerate}
    We say that the $X$-coarse structure $\ca E$ is unital, if the diagonal $\Delta_Z\in \ca E$.
\end{defn}

\begin{defn}\label{defn controlled and weakly controlled set}
    Let $\calG$ be an étale groupoid, $Z_1,Z_2$ be locally compact Hausdorff proper $\calG$-spaces. For a subset $E$ of $Z_2\times_{\calG}Z_1$, if it is contained in some $\calG$-compact subset of $Z_2\times_{\calG\units}Z_1$, we say that $E$ is controlled; if it is contained in some $\calG$-invariant closed subset $F$ of $Z_2\times_{\calG\units}Z_1$ such that $r|_F:F\ra Z_2$ and $s|_F:F\ra Z_1$ are proper maps, we say that $E$ is weakly controlled.
    
    We denote the set of all controlled subsets of $Z_2\times_{\calG\units}Z_1$ as $\ca E(Z_2,Z_1)$, and the set of all weakly controlled subsets of $Z_2\times_{\calG\units}Z_1$ as $\ca E'(Z_2,Z_1)$.

    Especially, when $Z_1=Z_2=Z$, we write $\ca E_Z=\ca E(Z,Z)$ and $\ca E_Z'=\ca E'(Z,Z)$.
\end{defn}

\begin{ex}
    Let $\calG$ be an \'etale groupoid, $X,Y$ be two locally compact Hausdorff proper $\calG$-spaces, $f:X\ra Y$ be a $\calG$-equivariant proper continuous map. Then
    \[Gr(f):=\{(y,x)\in Y\times_{\calG\units}X:y=f(x)\}\]
    is a weakly controlled subset of $Y\times_{\calG\units}X$.
\end{ex}

\begin{lem}\label{maps between G-compact sets are proper}
    \textnormal{\cite[Lemma 4.2.1]{bonicke2018going}}
    Let $\calG$ be a locally compact Hausdorff groupoid, $X$ be a locally compact Hausdorff $\calG$-compact $\calG$-space and $Y$ be a locally compact Hausdorff proper $\calG$-space. Then any $\calG$-equivariant continuous map $X\ra Y$ is proper.
\end{lem}

\begin{prop}\label{summary of controlled and weakly controlled sets}
    Let $\calG$ be an étale groupoid, $(Z_1,\rho_1)$,$(Z_2,\rho_2)$ and $(Z_3,\rho_3)$ be locally compact Hausdorff proper $\calG$-spaces.
    \begin{enumerate}
        \item We have $\ca E(Z_2,Z_1)\subseteq \ca E'(Z_2,Z_1)$. If $Z_1$ or $Z_2$ is $\calG$-compact, then $\ca E(Z_2,Z_1)=\ca E'(Z_2,Z_1)$.
        \item If $E_1,E_2\in \ca E(Z_2,Z_1)$, then $E_1\inv\in \ca E(Z_1,Z_2)$, $E_1\cup E_2\in \ca E(Z_2,Z_1)$.
        \item If $E_1',E_2'\in \ca E'(Z_2,Z_1)$, then ${E_1'}\inv\in \ca E'(Z_1,Z_2)$, $E_1'\cup E_2'\in \ca E'(Z_2,Z_1)$.
        \item If $E_{21}'\in \ca E'(Z_2,Z_1)$, $E_{32}'\in \ca E'(Z_3,Z_2)$, then $E_{32}'\circ E'_{21}\in \ca E'(Z_3,Z_1)$.
        \item If $E\in \ca E(Z_2,Z_1)$ and $E'\in \ca E'(Z_3,Z_2)$, then $E'\circ E\in \ca E(Z_3,Z_1)$; if $E'\in \ca E'(Z_2,Z_1)$ and $E\in \ca E(Z_3,Z_2)$, then $E\circ E'\in \ca E(Z_3,Z_1)$.
    \end{enumerate}
\end{prop}

\begin{proof}
    (1) For $E\in \ca E(Z_2,Z_1)$, suppose that $Z'$ is a $\calG$-compact subset of $Z_2\times_{\calG\units} Z_1$ such that $E\subseteq Z'$. Since $Z_2$ is a proper $\calG$-space, by lemma \ref{maps between G-compact sets are proper}, $r|_{Z'}:Z'\ra Z_2$ is a proper map. Similar for the second coordinate projection. So we have $E\in \ca E'(Z_2,Z_1)$.

    Now assume that $Z_1$ is $\calG$-compact and $E'\in \ca E'(Z_2,Z_1)$. We assume that $K_1$ is a compact subset of $Z_1$ such that $Z_1=\calG K_1$, $F$ is a $\calG$-invariant closed subset of $Z_2\times_{\calG\units}Z_1$ containing $E'$ and such that $r|_F,s|_F$ are proper. For any $(z_2,z_1)\in F$, there exists $(\gamma,k_1)\in \calG\times_{s,\rho_1}K_1$, such that $z_1=\gamma k_1$. Since $F$ is $\calG$-invariant, we have $(\gamma\inv z_2,k_1)\in F$. This implies that $\gamma\inv z_2\in F\circ K_1$ and hence $(z_2,z_1)=\gamma(\gamma\inv z_2,k_1)\in \calG((F\circ K_1)\times_{\calG\units} K_1)$. We have that $F\circ K_1$ is compact, and therefore $E'$ is controlled. In conclusion, $\ca E(Z_2,Z_1)=\ca E'(Z_2,Z_1)$. The case that $Z_2$ is $\calG$-compact is similar.

    (2), (3) and (4) are trivial.

    (5) Assume that $E\in \ca E(Z_2,Z_1)$ and $E'\in \ca E'(Z_3,Z_2)$, we can assume that $K_1$ is a compact subset of $Z_1$, $K_2$ is a compact subset of $Z_2$, such that $E\subseteq \calG(K_2\times_{\calG\units}K_1)$ and $F$ is a $\calG$-invariant closed subset of $Z_3\times_{\calG\units}Z_2$ containing $E'$ such that $r|_F,s|_F$ are proper. Now assume that $(z,z'')\in F\circ (\calG(K_2\times_{\calG\units}K_1))$, so there exists $z'\in \calG K_2$, such that $(z,z')\in F$ and $(z',z'')\in \calG(K_2\times_{\calG\units}K_1)$. Then there exists $\gamma\in \calG$ and $(k',k'')\in K_{2,s(\gamma)}\times K_{1,s(\gamma)}$ such that $z'=\gamma k'$ and $z''=\gamma k''$. Since $F$ is $\calG$-invariant, $(\gamma\inv z,k')\in F$, which implies that $\gamma\inv z\in F\circ K_2$, and hence $(z,z'')=\gamma.(\gamma\inv z,k'')\in \calG((F\circ K_2)\times_{\calG\units}K_1)$. So we proved that $E'\circ E\subseteq F\circ \calG(K_2\times_{\calG\units}K_1)\subseteq \calG((F\circ K_2)\times_{\calG\units}K_1)$ is controlled. The proof for $E\circ E'$ is the same.
\end{proof}

\begin{corr}\label{coarse structure of G-space}
Assume that $\calG$ is an étale groupoid, $Z$ is a locally compact Hausdorff proper $\calG$-space. Then $\ca E_Z$ is a well-defined $\calG\units$-coarse structure on $Z$. And $\ca E_Z'$ is a well-defined $\calG\units$-coarse structure on $Z$, which is always unital. If $Z$ is $\calG$-compact, then $\ca E_Z=\ca E_Z'$ is unital. 
\end{corr}

\subsection{Support of operators}\label{section support}
Let $X$ be a locally compact Hausdorff space.

\begin{defn}
Let $(Z,\rho)$ be an $X$-space, $A$ be a $C_0(X)$-algebra, $(E,\pi)$ be a $Z,X,A$-module. For $e\in E$, we define $\supp_\pi (e)$ as the complement of
\[\{z\in Z:\exists f\in C_0(Z), \text{ s.t. }f(z)\neq 0, \pi(f)e=0\}.\]
Let $Ann_{\pi}(e)=\{f\in C_0(Z): \pi(f)e=0\}$, it is a closed ideal of $C_0(Z)$. Then let $U=\cup_{f\in Ann_\pi(e)}\{z\in Z:f(z)\neq 0\}$. The complement of $U$ is exactly $\supp_{\pi}(e)$. When there is no ambiguity, we simply write $\supp(e)$ for $\supp_\pi(e)$.

We say that $e$ is compactly supported if $\supp_\pi(e)$ is compact.
\end{defn}
It is easy to see that $\pi(C_c(Z))E$ is exactly the set of all compactly supported elements of $E$.

\begin{lem}
    Use the same notation as above, $(\supp_{\pi}(e))_x\supseteq \supp_{\pi_x}(e(x))$.
\end{lem}
\begin{proof}
Let $U(e)=\{z\in Z: \exists f\in C_0(Z), f(z)\neq 0\wedge \pi(f)e=0\}$. Then $\supp_{\pi}(e)=\overline{U(e)^c}$.

If $z\in U(e)\cap Z_x$, there exists $f\in C_0(Z)$ such that $f(z)\neq 0$ and $\pi(f)e=0$, then $f|_{Z_x}\in C_0(Z_x)$, $\pi_x(f|_{Z_x})e(x)=0$. This implies that $z\in U(e(x))$. Hence, $U(e)\cap Z_x\subseteq U(e(x)).$
So we have
$\supp_{\pi}(e)\cap Z_x\supseteq \supp_{\pi_x}(e(x))$.
\end{proof}

\begin{lem}\label{empty support must be 0}
    If $\supp_{\pi}(e)=\emptyset$, then $e=0$.
\end{lem}
\begin{proof}
Since $\supp_{\pi}(e)=\emptyset$, we have $Ann_{\pi}(e)=\{f\in C_0(Z):\pi(f)e=0\}=C_0(Z)$, therefore $\pi(C_0(Z))e=0$. Then for any $e'\in E$, $f\in C_0(Z)$,
\[\langle e,\pi(f)e'\rangle_E=\langle \pi(\bar f)e,e'\rangle_E=0.\]
Since $\pi$ is non-degenerate, $\pi(C_0(Z))E$ has dense span in $E$, hence $e=0$.
\end{proof}

\begin{corr}\label{no intersection between supports of vector and operator}
    \begin{enumerate}
        \item If $f\in C_b(Z)$ such that $f|_{\supp_{\pi}(e)}=0$, then $\tilde\pi(f)e=0$.
        \item If $f\in C_b(Z)$ such that $f|_{\supp_{\pi}(e)}=1$, then $\tilde\pi(f)e=e$.
    \end{enumerate}
\end{corr}
\begin{proof}
    (1) For any $g\in C_0(Z)$, $gf\in C_0(Z)$ vanishes on $\supp_\pi(e)$, hence $\pi(g)\tilde\pi(f)e=\pi(gf)e=0$. So we have $\supp_\pi(\tilde\pi(f)e)=\emptyset$. By lemma \ref{empty support must be 0}, we have $\tilde\pi(f)e=0$. (2) is implied by (1).
\end{proof}

\begin{corr}\label{elements with empty support intersection are orthogono}
If $e_1, e_2\in E$ has $\supp_{\pi}(e_1)\cap \supp_{\pi}(e_2)=\emptyset$, then $\langle e_1,e_2\rangle_{E}=0$.
\end{corr}
\begin{proof}
    If for all $f\in C_c(Z)$, we have $\langle \pi(f)e_1,\pi(f)e_2\rangle=0$, then we have $\langle e_1,e_2\rangle=0$. So, it suffices to prove the case that $e_1,e_2$ are compactly supported. Since $\supp_\pi(e_1)$ and $\supp_\pi(e_2)$ are disjoint compact subsets of the locally compact Hausdorff space $Z$, there exists a relatively compact open neighborhood $U$ of $\supp_\pi(e_1)$ such that $U\cap \supp_\pi(e_2)=\emptyset$. By Urysohn's lemma, we can choose $f\in C_c(Z)_+$ such that $\supp(f)\subseteq V$ and $f|_{\supp_\pi(e_1)}=1$. By corollary \ref{no intersection between supports of vector and operator},
    \[\langle e_1,e_2\rangle=\langle \pi(f)e_1,e_2\rangle = \langle e_1,\pi(f)e_2\rangle=0.\]
\end{proof}

\begin{defn}
Let $(Z_1,\rho_1)$ and $(Z_2,\rho_2)$ be two $X$-spaces, $A$ be a $C_0(X)$-algebra, $(E,\pi)$ be a $Z_1,X,A$-module, $(F,\pi')$ be a $Z_2,X,A$-module, $T\in \Hom_A(E,F)$ (not necessarily adjointable). Then we define the support $\supp_{\pi',\pi}(T)$ of $T$ as the complement of the open
\[\{(z_2,z_1)\in Z_2\times_{X}Z_1:\exists f\in C_0(Z_2), g\in C_0(Z_1), \text{ s.t. }f(z_2)\neq 0, g(z_1)\neq 0, \pi'(f)T\pi(g)=0\}.\]
That is, the complement of $\supp_{\pi',\pi}(T)$ is the union of all opens of the form $U_2\times_X U_1$, where $U_2$ is an open of $Z_2$ and $U_1$ is an open of $Z_1$ such that $\pi'(C_0(U_2))T\pi(C_0(U_1))=0$. When there is no ambiguity, we simply write $\supp(T)$ for $\supp_{\pi',\pi}(T)$.

Equivalently, $\supp_{\pi',\pi}(T)$ consists of point $(z_2,z_1)\in Z_2\times_X Z_1$ such that, for any $f\in C_0(Z_2)$ and $g\in C_0(Z_1)$ such that $f(z_2)\neq 0$ and $g(z_1)\neq 0$, we have $\pi'(f)T\pi(g)\neq 0$.
\end{defn}

\begin{lem}\label{equiv cond of support}
    For $(z_2,z_1)\in Z_2\times_X Z_1$, the following are equivalent,
    \begin{enumerate}
        \item $(z_2,z_1)\in \supp_{\pi',\pi}(T)$;
        \item for any open neighborhood $U$ of $z_2$ and open neighborhood $V$ of $z_1$, there exists $f\in C_c(U), g\in C_c(V)$ such that $\pi'(f)T\pi(g)\neq 0$;
        \item for any open neighborhood $U$ of $z_2$ and open neighborhood $V$ of $z_1$, there exists $f\in C_0(U), g\in C_0(V)$ such that $\pi'(f)T\pi(g)\neq 0$.
    \end{enumerate}
\end{lem}
\begin{proof}
    (1) implies (3): If $(z_2,z_1)\in \supp_{\pi',\pi}(T)$, let $U$ be any open neighborhood of $z_2$ and $V$ be any open neighborhood of $z_1$, then there exists $f\in C_0(U)$ and $g\in C_0(V)$ such that $f(z_2)\neq 0$, $g(z_1)\neq 0$, then $(z_2,z_1)\in \supp_{\pi',\pi}(T)$ implies that $\pi'(f)T\pi(g)\neq 0$.

    (3) implies (1): If $(z_2,z_1)\in \supp_{\pi',\pi}(T)$, then there exists open neighborhood $U$ of $z_2$ and open neighborhood $V$ of $z_1$, such that $\pi'(C_0(U))T\pi(C_0(V))=0$.

    (2) implies (3) is trivial.

    (3) implies (2): After replacing $f$ by some $f_1\in C_c(U)$ such that $\|f-f_1\|_\infty$ is small enough and replacing $g$ by some $g_1\in C_c(V)$ such that $\|g-g_1\|_\infty$ is small enough, we can assume that $f\in C_c(U)$ and $g\in C_c(V)$.
\end{proof}

\begin{ex}
    An operator has empty support if and only if it is a zero map.
\end{ex}

\begin{defn}
    Use the same notation as above. We say that $T$ is properly supported, if the restrictions of the coordinate projections
    \[Z_2\times_X Z_1\ra Z_1, Z_2\times_X Z_1\ra Z_2\]
    to $\supp_{\pi',\pi}(T)$ are proper maps.
\end{defn}

Clearly, $T$ is properly supported if and only if for any compact subsets $K_1$ of $Z_1$ and $K_2$ of $Z_2$, $\supp_{\pi',\pi}(T)\circ K_1$ and $K_2\circ \supp_{\pi',\pi}(T)$ are compact.

\begin{defn}
Use the same notation as above. Let $\phi:Z_1\ra Z_2$ be an $X$-continuous map, we say that $T$ intertwines $\pi$ and $\pi'$ through $\phi$, if for any $f\in C_0(Z_2)$,
\[\pi'(f)\circ T=T\circ (\phi_*\pi)(f).\]
\end{defn}

\begin{prop}\label{support of intertwiner}
    Use the same notation as above. If $T$ intertwines $\pi$ and $\pi'$ through $\phi$, then $\supp_{\pi',\pi}(T)\subseteq Gr(\phi)$, where $Gr(\phi)$ is defined as
    \[Gr(\phi):=\{(\phi(z),z)\in Z_2\times_X Z_1:z\in Z_1\}.\]
\end{prop}
\begin{proof}
    It is easy to see that $Gr(\phi)$ is a closed subset of $Z_2\times_X Z_1$. If $(z_2,z_1)\notin Gr(\phi)$, then there exists an open neighborhood $U_2$ of $z_2$ and $U_1$ of $z_1$ such that $U_2\times_X U_1\cap Gr(\phi)=\emptyset$. This implies that $\phi(U_1)\cap U_2=\emptyset$. For any $f\in C_0(U_2)$  and $g\in C_0(U_1)$, we have therefore $g\cdot \phi^*f=0$ in $C_0(Z_1)$. Then
    \[\pi'(f)T\pi(g)=T\tilde\pi(\phi^*f) \pi(g)=0.\]
    That is $(z_2,z_1)\notin \supp_{\pi',\pi}(T)$.
\end{proof}

\begin{ex}
    Let $(E,\pi)$ be a $Z,X,A$-module, then for any $f\in C_b(Z)$, 
    \[\supp_{\pi,\pi}(\tilde \pi(f))\subseteq \overline{\{(z,z)\in Z\times_X Z:f(z)\neq 0\}}.\]
\end{ex}

\begin{prop}\label{property of support of operator}
Let $Z_1,Z_2,Z_3$ be $X$-spaces, $A$ be a $C_0(X)$-algebra, $(E_i,\pi_i)$ be a $Z_i,X,A$-module for $i=1,2,3$ and $T\in \Hom_A(E_1,E_2)$ be an operator.
\begin{enumerate}
    \item For any $x\in X$, $(\supp_{\pi_2,\pi_1}(T))_x\supseteq  \supp_{\pi_{2,x},\pi_{1,x}}(T_x)$.
    \item If $e\in \pi_1(C_c(Z))E_1$, then
    \[\supp_{\pi_2}(Te)\subseteq \supp_{\pi_2,\pi_1}(T)\circ \supp_{\pi_1}(e).\]
    Moreover, $\supp_{\pi_2,\pi_1}(T)$ is the smallest closed subset of $Z_2\times_X Z_1$ satisfying this property.
    \item If $T$ is adjointable, then \[\supp_{\pi_1,\pi_2}(T^*)=\supp_{\pi_2,\pi_1}(T)\inv.\]
    \item If $T_1,T_2\in \Hom_A(E_1,E_2)$ are operators, then \[\supp_{\pi_2,\pi_1}(T_1+T_2)\subseteq \supp_{\pi_2,\pi_1}(T_1)\cup \supp_{\pi_2,\pi_1}(T_2).\]
    \item If $T'\in \Hom_A(E_2,E_3)$ is an operator, $T$ is properly supported, then \[\supp_{\pi_3,\pi_1}(T'\circ T)\subseteq \supp_{\pi_3,\pi_2}(T')\circ \supp_{\pi_2,\pi_1}(T).\]
    \item If $T'\in \Hom_A(E_2,E_3)$ is an operator, then \[\supp_{\pi_3,\pi_1}(T'\circ T)\subseteq \overline{\supp_{\pi_3,\pi_2}(T')\circ \supp_{\pi_2,\pi_1}(T)}.\]
\end{enumerate}
\end{prop}
\begin{proof}

(1) Let $(z_2,z_1)\in Z_{2,x}\times Z_{1,x}$ be not in $\supp_{\pi_2,\pi_1}(T)$. Then there exists an open neighborhood $U_2$ of $z_2$ in $Z_2$ and an open neighborhood $U_1$ of $z_1$ in $Z_1$, such that $\pi_2(C_0(U_2))T\pi_1(C_0(U_1))=0$. This implies that $\pi_{2,x}(C_0(U_{2,x}))T_x\pi_{1,x}(C_0(U_{1,x}))=0$. Therefore, $(z_2,z_1)\notin \supp_{\pi_{2,x},\pi_{1,x}}(T_x)$.

(2) If $z_2\notin \supp_{\pi_2,\pi_1}(T)\circ \supp_{\pi_1}(e)$, then for any $z_1\in \supp_{\pi_1}(e)$, we should have $(z_2,z_1)\notin \supp_{\pi_2,\pi_1}(T)$. Hence, for every $z_1\in \supp_{\pi_1}(e)$, we can choose an open neighborhood $U_{z_1}$ of $z_2$ in $Z_2$ and an open neighborhood $V_{z_1}$ of $z_1$ in $Z_1$, such that $\pi_2(C_0(U_{z_1}))T\pi_1(C_0(V_{z_1}))=0$. Since $(V_{z_1})_{z_1\in \supp_{\pi_1}(e)}$ is an open cover of the compact set $\supp_{\pi_1}(e)$, after selecting a finite subcover, there exists opens $V_1,\cdots, V_n$ of $Z_1$ and $U_1,\cdots, U_n$ of $Z_2$, such that
\begin{enumerate}
    \item for every $1\leqslant i\leqslant n$, $\pi_2(C_0(U_i))T\pi_1(C_0(V_i))=0$;
    \item $\supp_{\pi_1}(e)\subseteq \cup_{i=1}^n V_i$.
\end{enumerate}

By Urysohn's lemma, we can choose $g\in C_c(Z_1)$ such that $\supp(g)\subseteq \cup_{i=1}^n V_i$ and $g|_{\supp_{\pi_1}(e)}=1$. Using a partition of unity, we have $\pi_2(C_0(U))T\pi_1(g)=0$, where $U=\cap_{i=1}^n U_i$ is an open neighborhood of $z_2$. Then for any $f\in C_0(U)$,
\[\pi_2(f)Te=\pi_2(f)T\pi_1(g)e=0.\]
That implies that $z_2\notin \supp_{\pi_2}(Te)$. So we proved that $\supp_{\pi_2}(Te)\subseteq \supp_{\pi_2,\pi_1}(T)\circ \supp_{\pi_1}(e)$.

Now for the minimality, assume that $F\subseteq Z_2\times_X Z_1$ is a closed subset such that for any $e\in \pi_1(C_c(Z))E_1$, $\supp_{\pi_2}(Te)\subseteq F\circ \supp_{\pi_1}(e)$. Assume that $(z_2,z_1)\in \supp_{\pi_2,\pi_1}(T)$. Claim: for any relatively compact open neighborhoods $U_2$ of $z_2$ and $U_1$ of $z_1$, there exists $u\in U_2$ and $v\in U_1$, such that $(u,v)\in F$. Since $(z_2,z_1)\in \supp_{\pi_2,\pi_1}(T)$, there exists $f\in C_0(U_2)$ and $g\in C_0(U_1)$ such that $\pi_2(f)T\pi_1(g)\neq 0$. Then there exists $e\in E_1$ such that $\pi_2(f)T\pi_1(g)e\neq 0$. Let $e_1=\pi_1(g)e$, which is in $\pi_1(C_c(Z))E$ since $U_1$ is relatively compact, we see that $\supp_{\pi_1}(e_1)\subseteq U_1$ and $\supp_{\pi_2}(Te_1)\cap U_2\neq \emptyset$. And $\supp_{\pi_2}(Te_1)\subseteq F\circ \supp_{\pi_1}(e_1)$ implies that there exists $u\in U_2$ and $v\in U_1$ such that $(u,v)\in F$. We proved our claim. And the claim implies that $\supp_{\pi_2,\pi_1}(T)\subseteq F$. We proved the minimality.

The proof of (3) and (4) are trivial.

(5) Since $T$ is properly supported, by (2), $T$ maps $\pi_1(C_c(Z_1))E_1$ into $\pi_2(C_c(Z_2))E_2$. Then again by (2), for any $e\in \pi_1(C_c(Z_1))E_1$,
\[\supp_{\pi_3,\pi_1}(T'(Te))\subseteq \supp_{\pi_3,\pi_2}(T')\circ \supp_{\pi_2,\pi_1}(T)\circ \supp_{\pi_1}(e).\]
By the minimality of $\supp_{\pi_3,\pi_1}(T'\circ T)$, we have
\[\supp_{\pi_3,\pi_1}(T'\circ T)\subseteq \supp_{\pi_3,\pi_2}(T')\circ \supp_{\pi_2,\pi_1}(T).\]

(6) Suppose that $(z_3,z_1)\in \supp_{\pi_3,\pi_1}(T)$. Let $\ca O(z_3)$ be the upward-filtered ordered set of open neighborhoods of $z_3$ in $Z_3$ and $\ca O(z_1)$ be the upward-filtered ordered set of open neighborhoods of $z_1$ in $Z_1$. (If $U,U'\in \ca O(z_3)$, then $U\leqslant U'$ if and only if $U\supseteq U'$.) For any $U\in \ca O(z_3)$ and $V\in \ca O(z_1)$, there exists $f_{3,U,V}\in C_c(U)$ and $f_{2,U,V}\in C_c(V)$ such that $\pi_3(f_{3,U,V})T'T\pi_1(f_{1,U,V})\neq 0$. 

Claim: for any $U\in \ca O(z_3)$, $V\in \ca O(z_1)$, there exists $y_{U,V}\in Z_2$, such that for any open neighborhood $W$ of $y_{U,V}$, there exists $h\in C_0(W)$ such that
\[\pi_3(f_{3,U,V})T'\pi_2(h)T\pi_1(f_{1,U,V})\neq 0.\]
If not, for any $y\in Z_2$, there exists an open neighborhood $W_y$ of $y$, such that 
\[\pi_3(f_{3,U,V})T'\pi_2(C_0(W_y))T\pi_1(f_{1,U,V})=0.\]
Then after using a partition of unity, we have
\[\pi_3(f_{3,U,V})T'\pi_2(C_c(Z_2))T\pi_1(f_{1,U,V})=0,\]
then let $(h_\lambda)_\lambda$ be an approximate unit of $C_0(Z_2)$ contained in $C_c(Z_2)$, then 
\[0=\pi_3(f_{3,U,V})T'\pi_2(h_\lambda)T\pi_1(f_{1,U,V})\ra \pi_3(f_{U,V})T'T\pi_1(g_{U,V})\neq 0\]
in strong module topology, which is a contradiction. We proved our claim.

Claim: for any $U\in \ca O(z_3)$, $V\in \ca O(z_1)$, there exists $z_{3,U,V}\in U$, such that $(z_{3,U,V},y_{U,V})\in \supp_{\pi_3,\pi_2}(T')$. Otherwise, for any $z\in U$, there exists an open neighborhood $U'_z$ of $z$ and an open neighborhood $W_z$ of $y_{U,V}$, such that
\[\pi_3(C_0(U'_z))T'\pi_2(C_0(W_z))=0.\]

Since $f_{3,U,V}$ is compactly supported, there exists finitely many points $p_1, \cdots, p_m\in U$ such that $\supp(f_{3,U,V})\subseteq \cup_{i=1}^m U'_{p_i}$. Now let $W=\cap_{i=1}^m W_{p_i}$, after using a partition of unity, we have
\begin{equation}\label{eqn 2}
    \pi_3(f_{3,U,V})T'\pi_2(C_0(W))=0.
\end{equation}

However, by our choice of $y_{U,V}$, there exists some $h\in C_0(W)$ such that
\[\pi_3(f_{3,U,V})T'\pi_2(h)T\pi_1(f_{1,U,V})\neq 0,\]
which contradicts to (\ref{eqn 2}). Hence, we proved our claim. Similarly, there exists $z_{1,U,V}\in V$ such that $(y_{U,V}, z_{1,U,V})\in \supp_{\pi_2,\pi_1}(T)$.

Now we have: for each $U\in \ca O(z_3), V\in \ca O(z_1)$, $(z_{3,U,V}, z_{1,U,V})\in \supp_{\pi_3,\pi_2}(T')\circ \supp_{\pi_2,\pi_1}(T)$, and
\[\lim_{(U,V)\in \ca O(z_3)\times \ca O(z_1)}(z_{3,U,V}, z_{1,U,V})=(z_3,z_1).\]
We have therefore $(z_3,z_1)\in \overline{\supp_{\pi_3,\pi_2}(T')\circ \supp_{\pi_2,\pi_1}(T)}$.
\end{proof}

The following lemma is an analog of \cite[Lemma 4.1.15]{willett2020higher}.
\begin{lem}\label{properly supported operator transit compact support}
    Let $Z_1, Z_2$ be $X$-spaces, $A$ be a $C_0(X)$-algebra, $(E_1,\pi_1)$ be a $Z_1,X,A$-module, $(E_2,\pi_2)$ be a $(Z_2,X,A)$-module. Let $T:E_1\ra E_2$ be a properly supported bounded $A$-linear map and let $F=\supp_{\pi_2,\pi_1}(T)$.
    \begin{enumerate}
        \item If $K_1$ is a compact subset of $Z_1$, $f\in C_c(Z_1)$ such that $\supp(f)\subseteq K_1$, then for any $g\in C_c(Z_2)$ such that $g|_{F\circ K_1}=1$, we have $\pi_2(g)T\pi_1(f)=T\pi_1(f)$.
        \item If $K_2$ is a compact subset of $Z_2$, $g\in C_c(Z_2)$ such that $\supp(g)\subseteq K_2$, then for any $f\in C_c(Z_1)$ such that $f|_{K_2\circ F}=1$, we have $\pi_2(g)T\pi_1(f)=\pi_2(g)T$.
    \end{enumerate}
\end{lem}
\begin{proof}
    (1) By proposition \ref{property of support of operator},
    \[\supp_{\pi_2,\pi_1}(T\pi_1(f))\subseteq F\circ \supp_{\pi_1,\pi_1}(\pi_1(f))\subseteq (F\circ K_1)\times_X K_1.\]
    Fix $y\in Z_2\setminus (F\circ K_1)$. Then for any $x\in K_1$, $(y,x)\notin \supp_{\pi_2,\pi_1}(T\pi_1(f))$, so there exists an open neighborhood $V_{y,x}$ of $x$ in $Z_1$ and an open neighborhood $U_{y,x}$ of $y$ in $Z_2\setminus (F\circ K_1)$, such that
    \[\pi_2(C_0(U_{y,x}))T\pi_1(f)\pi_1(C_0(V_{y,x}))=0.\]

    There are finitely many points $x_1,\cdots, x_n$ such that $K_1\subseteq \cup_{i=1}^n V_{y,x_i}$. Then there exists a partition of unity: $\phi_1\in C_c(V_{y,x_1},[0,1]),\cdots, \phi_n\in C_c(V_{y,x_n},[0,1])$, such that for any $z\in K_1$, $\sum_{i=1}^n \phi_i(z)=1$.  Let $U_y=\cap_{i=1}^n U_{y,x_i}$. So for any $g_y\in C_0(U_y)$, we have
    \[\pi_2(g_y)T\pi_1(f)=\sum_{i=1}^n \pi_2(g_y)T \pi_1(f)\pi_1(\phi_i)=0.\]
    In conclusion, for any $e\in E_1$, we proved that for any $y\notin F\circ K_1$, $y$ is not in $\supp_{\pi_2}(T\pi_1(f)e)$, so $\supp_{\pi_2}(T\pi_1(f)e)\subseteq F\circ K$. Now $\chi_{Z_2}-g$ is an element of $C_b(Z_2)$ such that $(\chi_{Z_2}-g)|_{F\circ K_1}=0$. By corollary \ref{no intersection between supports of vector and operator}, for any $e\in E_1$,
    \[(T\pi_1(f)-\pi_2(g)T\pi_1(f))e=\tilde \pi_2(\chi_{Z_2}-g)T\pi_1(f)e=0,\]
    that is $\pi_2(g)T\pi_1(f)=T\pi_1(f)$.

    (2) Use similar argument, we can prove that for any $x\in Z_1\setminus (K_2\circ F)$, there exists an open neighborhood $V_x$ of $x$ in $Z_1\setminus (K_2\circ F)$, such that
    \[\pi_2(g)T\pi_1(C_0(V_x))=0.\]
    Let $I=\{\psi\in C_0(Z_1): \pi_2(g)T\pi_1(\psi)=0\}$. So $I$ is a closed ideal of $C_0(Z_1)$. Let $U=\cup_{\psi\in I}\{z\in Z_1: \psi(z)\neq 0\}$. So we proved that $Z_1\setminus (K_2\circ F)\subseteq U$ and therefore for any $f\in C_0(Z_1)$ such that $f|_{K_2\circ F}=0$, we have $\pi_2(g)T\pi_1(f)=0$.

    Let $\Lambda$ be the set of all compact subsets of $Z_1$, equipped with upward-filtering order defined by inclusions. For any $K\in \Lambda$, choose some $h\in C_c(Z_1)$ such that $h|_K=1$. Then $(h_K)_{K\in \Lambda}$ is an approximate unit of $C_0(Z_1)$. Since $\pi_1$ is non-degenerate, $\pi_1(h_K)$ converge strictly to $id_{E_1}$. For any $K\in \Lambda$ such that $\supp(f)\subseteq K$, $h_K-f$ is an element of $C_c(Z_1)$ such that $(h_K-f)|_{K_2\circ F}=0$. So for $K$ large enough, $\pi_2(g)T\pi_1(h_K-f)=0$. After taking the strict limit, we have $\pi_2(g)T=\pi_2(g)T\pi_1(f)$.
\end{proof}

\begin{lem}\label{support of equivariant operator is equivariant}
    Let $\calG$ be an étale groupoid, $A$ be a $\calG$-\cst-algebra, $(Z_1,\rho_1),(Z_2,\rho_2)$ be locally compact Hausdorff proper $\calG$-spaces, $(E_1,\pi_1), (E_2,\pi_2)$ be $Z,\calG,A$-modules. Assume that $T\in \Hom_A(E_1,E_2)$ is $\calG$-equivariant, then $\supp_{\pi_2,\pi_1}(T)$ is a $\calG$-invariant closed subset of $Z_2\times_{\calG\units}Z_1$.
\end{lem}
\begin{proof}
    It suffices to show that the complement of $\supp_{\pi_2,\pi_1}(T)$ is $\calG$-invariant. We will prove that, for any $(z_2,z_1)\not\in \supp_{\pi_2,\pi_1}(T)$ and $\gamma\in \calG$ such that $s(\gamma)=\rho_1(z_1)=\rho_2(z_2)$, we have $(\gamma z_1,\gamma z_2)\not\in \supp_{\pi_2,\pi_1}(T)$.

    Without loss of generality, we can select an open bisection $B$ containing $\gamma$, an open neighborhood $V_1$ of $Z_1$ and an open neighborhood $V_2$ of $Z_2$, such that $s(B)=\rho_1(V_1)=\rho_2(V_2)$, and $(V_2\times_{\calG\units}V_1)\cap \supp_{\pi_2,\pi_1}(T)=\emptyset$. By definition, 
    \[\pi_2|_{s(B)}(C_0(V_2))T|_{s(B)}\pi_1|_{s(B)}(C_0(V_1))=0.\]

    Claim: $\pi_2(C_0(BV_2))T\pi_1(C_0(BV_1))=0$. Assume that $W_i\in \calL(s^*E_i,r^*E_i)$ is the action of $\calG$ on $E_i$ for $i=1,2$. For any $f\in C_0(BV_2)$ and $g\in C_0(BV_1)$, clearly for any $x\in \calG\units\setminus r(B)$,
    \[(\pi_2(C_0(BV_2))T\pi_1(C_0(BV_1)))_x=0.\]
    Otherwise, for any $x\in r(B)$, assume that $\gamma'=r|_B\inv(x)$, then we have
    \begin{align*}
        & (\pi_2(C_0(BV_2))T\pi_1(C_0(BV_1)))_x \\
        & = \pi_{2,r(\gamma')}(C_0(\gamma' V_{2,s(\gamma')}))T_{r(\gamma')}\pi_{1,r(\gamma')}(C_0(\gamma'V_{1,s(\gamma')})) \\
        & =W_{2,\gamma'}\pi_{2,s(\gamma')}(C_0(V_{2,s(\gamma')})) W_{2,\gamma'}\inv T_{r(\gamma')} W_{1,\gamma'} \pi_{1,s(\gamma')}(C_0(V_{1,s(\gamma')})) W_{1,\gamma'}\inv\\
        & =W_{2,\gamma'} \pi_{2,s(\gamma')}(C_0(V_{2,s(\gamma')})) T_{s(\gamma')} \pi_{1,s(\gamma')}(C_0(V_{1,s(\gamma')})) W_{1,\gamma}\inv=0.
    \end{align*}
    We proved our claim. And therefore, $(\gamma z_2,\gamma z_1)\not\in \supp_{\pi_2,\pi_1}(T)$.
\end{proof}

\begin{prop}\label{support of inflation}
    Let $\calG$ be an étale groupoid, $\calH$ be a relatively clopen subgroupoid, $\Omega=\calG_{\calH\units}$, $A$ be a $\calG$-\cst-algebra, $E_1, E_2$ be a $\calG$-Hilbert $A$-module, $(Z,\rho)$ be a locally compact Hausdorff $\calH$-space, and $\pi_i:C_0(Z)\ra \calL(E_i)$ be a non-degenerate $\calH$-equivariant representation for $i=1,2$. If $T\in \Hom_A(E_1,E_2)$ is $\calH$-equivariant, then $\supp_{\infl^\calG_\calH\pi_2, \infl^\calG_\calH\pi_1}(\infl^\calG_\calH(T))$ is contained in
    \[ \{([\gamma, z_2], [\gamma, z_1])\in (\Omega\times_\calH Z)\times_{\calG\units}(\Omega\times_\calH Z):\gamma\in \Omega, (z_2,z_1)\in \supp_{\pi_2,\pi_1}(T), s(\gamma)=\rho(z_1)=\rho(z_2)\},\]
    which can be identified with $\Omega\times_\calH \supp(T)$.
\end{prop}
\begin{proof}
    If $([\gamma_1,z_1],[\gamma_2,z_2])\in (\Omega\times_\calH Z)\times_{\calG\units}(\Omega\times_\calH Z)$ is not in the latter set, then either $\gamma_1\calH\neq \gamma_2\calH$, or there exists $\gamma\in \Omega$, $(z_2',z_1')\in \supp_{\pi_2,\pi_1}(T)^c$ such that $([\gamma_1,z_1],[\gamma_2,z_2])=([\gamma, z_2'],[\gamma,z_1']))$.

    Let $q:\Omega\ra \Omega/\calH$ be the quotient map. In this first case, if $\gamma_1\calH\neq \gamma_2\calH$, then there exists a bisection $B_1\subseteq \Omega$ containing $\gamma_1$, and a bisection $B_2\subseteq \Omega$ containing $\gamma_2$, such that $q(B_1)\cap q(B_2)=\emptyset$. Therefore, $B_1\calH\cap B_2\calH=\emptyset$. We define
    \[[B_i,Z]:=\{[b,u]\in \Omega\times_\calH Z:(b,u)\in B_i\times_{\calH\units}Z\}\]
    for $i=1,2$, and $[B_i,Z]$ is an open neighborhood of $[\gamma_i,z_i]$ in $\Omega\times_\calH Z$. And for any $f\in C_0([B_2,Z])$, $g\in C_0([B_1,Z])$, for any $\gamma\in \Omega$, $\xi\in \ind_{\Omega,c}E\subseteq \infl^\calG_\calH(E)$,
    \begin{equation}\label{formula_new}
        [(\infl^\calG_\calH\pi_2)(f) \infl^\calG_{\calH}(T) (\infl^\calG_\calH\pi_1)(g)\xi](\gamma)=\pi_{2,s(\gamma)}(f([\gamma,-]))T_{s(\gamma)}\pi_1(g([\gamma,-]))\xi(\gamma).
    \end{equation}
    If $\gamma\not\in B_1\calH$, then $g([\gamma,-])=0$; if $\gamma\not\in B_2\calH$, then $f([\gamma,-])=0$. So we have
    \[(\infl^\calG_\calH\pi_2)(C_0([B_2,Z])) \infl^\calG_{\calH}(T) (\infl^\calG_\calH\pi_1)(C_0([B_1,Z]))=0,\]
    which implies that $([\gamma_2,z_2],[\gamma_1,z_1])\in \supp(\infl^\calG_\calH(T))^c$.

    In the second case, without loss of generality assume that $\gamma_1=\gamma_2$ and $(z_2,z_1)\in \supp(T)^c$. Then there exists an open neighborhood $U_i$ of $z_i$ in $Z_i$ for $i=1,2$, such that
    \[\pi_2(C_0(U_2))T\pi_1(C_0(U_1))=0.\]
    Then
    \[[\Omega,U_i]:=\{[\omega,u]\in \Omega\times_\calH Z:(\omega,u)\in \Omega\times_{\calH\units}U_i\}\]
    is open neighborhood of $[\gamma_i,z_i]$ in $\Omega\times_\calH Z$. And for any $f\in C_0([\Omega,U_2])$, $g\in C_0([\Omega,U_1])$, $\xi\in \ind_{\Omega,c}E\subseteq \infl^\calG_\calH(E)$, $\gamma\in \Omega$,
    \[[(\infl^\calG_\calH\pi_2)(f) \infl^\calG_{\calH}(T) (\infl^\calG_\calH\pi_1)(g)\xi](\gamma)=\pi_{2,s(\gamma)}(f([\gamma,-]))T_{s(\gamma)}\pi_1(g([\gamma,-]))\xi(\gamma)=0,\]
    which is because $f([\gamma,-])\in C_0(U_{2,s(\gamma)})$ and $g([\gamma,-])\in C_0(U_{2,s(\gamma)})$. That is,
    \[(\infl^\calG_\calH\pi_2)(C_0([\Omega,U_2])) \infl^\calG_{\calH}(T) (\infl^\calG_\calH\pi_1)(C_0([\Omega,U_1]))=0,\]
    hence, $([\gamma_2,z_2],[\gamma_1,z_1])\in \supp(\infl^\calG_\calH(T))^c$.
\end{proof}

\section{Roe algebras and K-theory}\label{sec roe alg}

Recall that, in the theory of Roe algebras of proper metric space (see \cite{willett2020higher}, \cite{higson2000analytic}, etc.), ample modules play the role of standard modules. If $X$ is a proper metric space, an $X$-module is a non-degenerate representation $\pi:C_0(X)\ra B(H)$, where $H$ is a separable Hilbert space. Then the $X$-module $(H,\pi)$ is ample means that for any non-zero $f\in C_0(X)$, $\pi(f)$ is not a compact operator. The key fact is that an ample module $(H,\pi)$ will coarsely absorb all other $X$-modules: if $(H',\pi')$ is another $X$-module, then there exists an isometry $V:H'\ra H$, such that the propagation of $V$ is finite. Moreover, the propagation of $V$ can be chosen to be arbitrarily small.

However in out setting, let $\calG$ be a second countable \'etale groupoid, $Z$ be a second countable locally compact Hausdorff $\calG$-space and $A$ be a $\calG$-\cst-algebra, for a $Z,\calG,A$-module $(E,\pi)$, there cannot be a reasonable action of Borel functions over $Z$ on $E$, so a coarse counterpart to Voiculescu's theorem based on Borel decompositions cannot work here like in the proofs of \cite[Proposition 6.3.12]{higson2000analytic} or \cite[Corollary 4.2.7]{willett2020higher}. Our strategy to remedy this is surprisingly simple but effective: we will accept properly supported non-adjointable isometries to have *-homomorphisms between Roe algebras of geometric modules.

In this section, we fix an \'etale groupoid $\calG$ and a $\calG$-\cst-algebra $(A,\alpha)$.

\subsection{Roe algebras}

\begin{defn}
    Let $Z_1,Z_2$ be \lch  proper $\calG$-spaces, $(E_1,\pi_1)$ be a $Z_1,\calG,A$-module and $(E_2,\pi_2)$ be a $Z_2,\calG,A$-module. For an operator $T\in \Hom_A(E_1,E_2)$,
    \begin{enumerate}
        \item we say that $T$ is controlled, if $\supp_{\pi_2,\pi_1}(T)$ is a controlled subset of $Z_2\times_{\calG\units}Z_1$;
        \item we say that $T$ is weakly controlled, if $\supp_{\pi_2,\pi_1}(T)$ is a weakly controlled subset of $Z_2\times_{\calG\units}Z_1$;
        \item we say that $T$ is locally compact, if for any $f\in C_0(Z_1)$ and $g\in C_0(Z_2)$, we have $T\pi_1(f)\in \calK(E_1,E_2)$, $\pi_2(g)T\in \calK(E_1,E_2)$.
    \end{enumerate}
\end{defn}

By proposition \ref{summary of controlled and weakly controlled sets}, every controlled operator is weakly controlled; if $Z_1$ or $Z_2$ is $\calG$-compact, the converse is also true. The following lemma states that we should not distinguish $\calG$-equivariant properly supported operators and $\calG$-equivariant weakly controlled operators.

\begin{lem}\label{controlled operator is properly supported}
    Use the same notation as above, every weakly controlled operator is properly supported. And every $\calG$-equivariant properly supported operator is weakly controlled.
\end{lem}
\begin{proof}
    Suppose that $Z'$ is a $\calG$-invariant subset of $Z_2\times_{\calG\units} Z_1$ such that $\supp_{\pi_2,\pi_1}(T)\subseteq Z'$, and $r|_{Z'}, s|_{Z'}$ are proper. After composing the closed inclusion, $r|_{\supp(T)}$ and $s|_{\supp(T)}$ are also proper maps.

    Now, if $T$ is $\calG$-equivariant and properly supported, by lemma \ref{support of equivariant operator is equivariant}, $\supp_{\pi_2,\pi_1}(T)\in \ca E'(Z_2,Z_1)$.
\end{proof}

\begin{lem}
    Let $(Z_i,\rho_i)$ be a locally compact Hausdorff proper $\calG$-space for $i=1,2,3$, $(E_i,\pi_i)$ be a $Z_i,\calG,A$-module, $T_{12}\in \Hom_A(E_1,E_2)$ and $T_{23}\in \Hom_A(E_2,E_3)$. If $T_{12}$ is controlled and $T_{23}$ is $\calG$-equivariant properly supported, or $T_{12}$ is $\calG$-equivariant properly supported and $T_{23}$ is controlled, then $T_{23}\circ T_{12}$ is controlled.
\end{lem}

The following result is the key idea for our framework of Roe algebras.

\begin{prop}\label{properly supported+locally compact is adjointable}
    Let $Z,Z'$ be \lch proper $\calG$-spaces, $(E,\pi)$ be a $Z,\calG,A$-module and $(F,\pi')$ be a $Z',\calG,A$-module. If $T\in \Hom_A(E,F)$ is locally compact and properly supported, then $T$ is adjointable.
\end{prop}
\begin{proof}
    Let $S=\supp_{\pi,\pi'}(T)$. First, we define a map $\hat{T}:\pi(C_c(Z))E\ra F$ such that, for any $f\in \pi'(C_c(Z))F$, we define $\hat{T}f=(T\pi(h))^*f$, where $h\in C_c(Z)$ such that $h|_{ S\inv\circ \supp_{\pi'}(f)}=1$. For any two $h,h'$ satisfying these conditions, $\Delta h=h-h'$ vanishes on $S\inv\circ \supp_{\pi'}(f)$.

    \[\supp_{\pi'}((T\pi(\Delta h))^*f)\subseteq (S\circ \delta)\inv \circ \supp_{\pi'}(f)=\delta\inv\circ S\inv\circ \supp_{\pi'}(f),\]
    where
    \[\delta=\{(z,z):z\in \supp(\Delta h)\}\subseteq \{(z,z):z\in Z\setminus (S\inv\circ \supp_{\pi'}(f))\},\]
    therefore $\supp_{\pi'}((T\pi(\Delta h))^*f)=\emptyset$, which implies that $(T\pi(h))^*f=(T\pi(h'))^*f$. So $\hat T$ is a well-defined map. We can check easily that $\hat T$ is bounded $A$-linear map. Now, by abuse of language, we denote its extension on $E\ra F$ by $\hat T$.

    Claim: $T$ is adjoint to $\hat T$. Now for any $e\in \pi(C_c(Z))E$ and $f\in \pi'(C_c(Z))F$,
    \begin{align*}
        \langle Te, f\rangle & = \langle T \pi(h)e,f\rangle\\
        & = \langle e, (T\pi(h))^*f\rangle\\
        & = \langle e,\hat T f\rangle
    \end{align*}
    where $h\in C_c(Z)_+$ is chosen such that its restriction to $S\inv\circ \supp_{\pi'}(f)\cup \supp_{\pi}(e)$ is 1.
\end{proof}

\begin{defn}
    Let $Z,Z'$ be locally compact Hausdorff $\calG$-spaces, $(E,\pi)$ be a $Z,\calG,A$-module and $(F,\pi')$ be a $Z',\calG,A$-module. We define 
    \[\mathbb C_\calG[(E,\pi),(F,\pi')]:=\{T\in \Hom_A(E,F): T \text{ is } \calG\text{-equivariant, controlled and locally compact}\}.\]
    When there is no ambiguity, we write simply $\mathbb C_\calG[E,F]$ for $\mathbb C_\calG[(E,\pi),(F,\pi')]$.
    
    We denote the norm closure of $\mathbb C_\calG[(E,\pi),(F,\pi')]$ in $\Hom_A(E,F)$ by $C^*_\calG((E,\pi),(F,\pi'))$. We denote $C^*_\calG(E,E)$ by $C^*_\calG(E,\pi)$ for convenience of writing.
\end{defn}

\begin{lem}
    Use the same notation as above, $C_\calG^*(E,F)$ is a well-defined closed linear subspace of $\calL(E,F)$.
\end{lem}
\begin{proof}
    By lemma \ref{controlled operator is properly supported} and proposition \ref{properly supported+locally compact is adjointable}, all operators in $\mathbb C_\calG[E,F]$ are adjointable. And by proposition \ref{summary of controlled and weakly controlled sets} and proposition \ref{property of support of operator}, it is easy to see that $\mathbb C_\calG[E,F]$ is a linear subspace of $\calL(E,F)$.
\end{proof}

\begin{lem}
    Let $Z_1,Z_2,Z_3$ be \lch proper $\calG$-spaces, $(E_i,\pi_i)$ be a $Z_i,\calG,A$-module for $i=1,2,3$. Then
    \[C_\calG^*(E_2,E_3)\circ C_\calG^*(E_1,E_2)\subseteq C_\calG^*(E_1,E_3).\]
    Especially, if $Z$ is a \lch proper $\calG$-space and $(E,\pi)$ is a $Z,\calG,A$-module, then $C_\calG^*(E,\pi)$ is a well-defined closed *-subalgebra of $\calL(E)$.
\end{lem}
\begin{proof}
    It suffices to prove that, for $T\in \mathbb C_\calG[E_1,E_2]$ and $S\in \mathbb C_\calG[E_2,E_3]$, we have $ST\in \mathbb C_\calG[E_1,E_3]$.

    Clearly, $ST\in \calL(E_1,E_3)$ is $\calG$-equivariant and locally compact. Since $S$ and $T$ are controlled, by proposition \ref{summary of controlled and weakly controlled sets} and proposition \ref{property of support of operator}, $ST$ is also controlled.
\end{proof}

\begin{defn}
    Let $Z$ be a locally compact Hausdorff proper $\calG$-space,  $(E,\pi)$ and $(F,\pi')$ be $Z,\calG,A$-modules. For an operator $T\in \Hom_A(E,F)$, we say that $T$ is pseudolocal, if for any $f\in C_0(Z)$, we have 
    \[T\pi(f)-\pi'(f)T\in \calK(E,F).\]
    
    We define
    \[\mathbb D_\calG^*[(E,\pi),(F,\pi')]:=\{T\in \calL_A(E,F): T \text{ is } \calG\text{-equivariant, controlled and pseudolocal}\}.\]
    We denote the norm closure of $\mathbb D_\calG^*[(E,\pi),(F,\pi')]$ in $\calL_A(E,F)$ by $D^*_\calG((E,\pi),(F,\pi'))$.

    When there is no ambiguity, we write
    \[\mathbb D^*_\calG[E,F]:=\mathbb D^*_\calG[(E,\pi),(F,\pi')],\]
    \[D_\calG^*(E,F):=D_\calG^*((E,\pi),(F,\pi')),\]
    \[D_\calG^*(E,\pi):=D_\calG^*((E,\pi),(E,\pi)).\]
\end{defn}

\begin{lem}
    \begin{enumerate}
        \item Use the same notation as above, $D_\calG^*(E,F)$ is a closed subspace of $\calL(E,F)$.
        \item Let $(E_i,\pi_i)$ be a $Z,\calG,A$-module for $i=1,2,3$. Then
    \[D_\calG^*(E_2,E_3)\circ D_\calG^*(E_1,E_2)\subseteq D_\calG^*(E_1,E_3).\]
    Especially, if $Z$ is a \lch proper $\calG$-space and $(E,\pi)$ is a $Z,\calG,A$-module, $D_\calG^*(E,\pi)$ is a well-defined closed *-subalgebra of $\calL(E)$.
    \end{enumerate}
\end{lem}
\begin{proof}
    We prove that the composition of two pseudolocal operators is also pseudolocal. If $T\in \calL(E_1,E_2)$, $S\in \calL(E_2,E_3)$ are pseudolocal, then for any $f\in C_0(Z)$,
    \[ST\pi_1(f)-\pi_3(f)ST=S(T\pi_1(f)-\pi_2(f)T)+(S\pi_2(f)-\pi_3(f)S)T\in \calK(E_1,E_3).\]
    All other details are same as the lemmas above.
\end{proof}

\begin{prop}
    Let $Z$ be a \lch proper $\calG$-space and $(E,\pi)$ be a $Z,\calG,A$-module, then $C^*_\calG(E,\pi)$ is a two-sided ideal of $D^*_\calG(E,\pi)$.
\end{prop}
\begin{proof}
    Suppose that $T_1,T_2\in \calL(E)$ such that they are $\calG$-equivariant and controlled, $T_1$ is locally compact and $T_2$ is pseudolocal. Then clearly $T_1\circ T_2$ and $T_2\circ T_1$ are $\calG$-equivariant and controlled. Now for any $f\in C_0(Z)$, 
    \[\pi(f)(T_1\circ T_2) = (\pi(f)T_1)\circ T_2\in \calK(E),\]
    \[(T_1\circ T_2)\pi(f) = (T_1\pi(f)) T_2+T_1[T_2,\pi(f)]\in \calK(E).\]
    Hence, $T_1\circ T_2$ is also locally compact. Similarly, $T_2\circ T_1$ is also locally compact. So we prove that $C^*_\calG(E,\pi)$ is an ideal of $D^*_\calG(E,\pi)$.
\end{proof}

\subsection{Weakly controlled isometries and adjoint maps}

Recall that, let $E,F$ be Hilbert $A$-modules and $V:E\ra F$ be an isometry that may be not adjointable, we still have a well-defined adjoint map
\[Ad_V:\calK(E)\ra \calK(F),\quad T\mapsto (V\circ (V\circ T)^*)^*,\]
which maps $\theta_{e_1,e_2}$ to $\theta_{Ve_1, Ve_2}$. It is easy to check that this is an injective *-homomorphism that maps $\calK(E)$ onto a hereditary subalgebra of $\calK(F)$ (c.f. \cite[Lemma 4.1]{meyer2000equivariant}). Now replacing the compact operators by controlled locally compact operators, we will show that weakly controlled isometries can play the same role.

\begin{lem}\label{properly supported operators multiplies Roe algebra}
    Let $Z_1,Z_2,Z_3$ be locally compact Hausdorff proper $\calG$-spaces, $(E_i,\pi_i)$ be a $Z_i,\calG,A$-modules for $i=1,2,3$.
    \begin{enumerate}
        \item If $S\in \Hom_A(E_1,E_2)$ is properly supported, $T\in \Hom_A(E_2,E_3)$ is properly supported and locally compact, then $T\circ S$ is also properly supported and locally compact.
        \item If $S\in \Hom_A(E_1,E_2)$ is properly supported and locally compact, $T\in \Hom_A(E_2,E_3)$ is properly supported, then $T\circ S$ is also properly supported and locally compact.
    \end{enumerate}
\end{lem}
\begin{proof}
    We only prove (1), the proof of (2) is similar. Since $S$ and $T$ are properly supported, clearly $T\circ S$ is also properly supported by proposition \ref{property of support of operator}. Now for any $f\in C_c(Z_1)$, $S\pi_1(f)\in \calK(E_1,E_2)$, so $TS\pi_1(f)\in \calK(E_1,E_3)$. 
    For any $h\in C_c(Z_3)$, $\supp(h)\circ \supp(T)$ is a compact subset of $Z_2$, so there exists $g\in C_c(Z_2)$ such that $g|_{\supp(h)\circ \supp(T)}=1$. Then by lemma \ref{properly supported operator transit compact support}, $\pi_3(h)T=\pi_3(h)T\pi_2(g)$, hence $\pi_3(h)T\circ S=\pi_3(h)T\pi_2(g)\circ S\in \calK(E_1,E_3)$. We proved that $T\circ S$ is locally compact.
\end{proof}

Notice that, properly supported locally compact operators are adjointable. Therefore, even if the covering isometries are not adjointable, their composition with an element in the Roe algebra is still adjointable, and this will allow us to define the adjoint map by covering isometries between Roe algebras.

\begin{prop}
Let $Z_1,Z_1',Z_2,Z_2'$ be locally compact Hausdorff proper $\calG$-spaces, $(E_i,\pi_i)$ be $Z_i,\calG,A$-modules for $i=1,2$, $(E_i',\pi_i')$ be $Z_i',\calG,A$-modules for $i=1,2$, $V\in \Hom_A(E_2,E_2')$ and $W\in \Hom_A(E_1,E_1')$ be properly supported $\calG$-equivariant isometries. Then
\[Ad_{V,W}:\mathbb C_\calG[E_1,E_2]\ra \mathbb C_\calG[E_1',E_2'],\quad T\mapsto (W\circ (V\circ T)^*)^*\]
is a well-defined bounded linear map. Moreover, $Ad_{V,W}$ extends to a bounded linear map from $C^*_\calG(E_1,E_2)$ to $C^*_\calG(E_1',E_2')$.
\end{prop}
\begin{proof}
    Recall that for $\calG$-equivariant operators, being weakly controlled operators is equivalent to being properly supported (lemma \ref{controlled operator is properly supported}). Then by lemma \ref{properly supported operators multiplies Roe algebra}, for any $T\in \mathbb C_\calG[E_1,E_2]$, $V\circ T$ is properly supported and locally compact, and hence adjointable by proposition \ref{properly supported+locally compact is adjointable}. Similarly, $W\circ (V\circ T)^*$ is properly supported, locally compact and hence adjointable. Then it is easy to see that $Ad_{V,W}(T)$ is $\calG$-equivariant and locally compact. And
    \begin{align*}
        \supp_{\pi_2',\pi_1'}(Ad_{V,W}(T)) & \subseteq (\supp_{\pi_1',\pi_2'}(W(VT)^*))\inv\\
        & \subseteq (\supp_{\pi_1',\pi_1}(W)\circ \supp_{\pi_1,\pi_2'}(VT)\inv)\inv\\
        & \subseteq (\supp_{\pi_1',\pi_1}(W)\circ (\supp_{\pi_2',\pi_2}(V)\circ \supp_{\pi_2,\pi_1}(T))\inv)\inv\\
        & = \supp_{\pi_2',\pi_2}(V)\circ \supp_{\pi_2,\pi_1}(T)\circ \supp_{\pi_1',\pi_1}(W)\inv.
    \end{align*}
    Therefore, by proposition \ref{summary of controlled and weakly controlled sets}, $Ad_{V,W}(T)$ is also controlled.

    Clearly, for any $T\in \mathbb C_\calG[E_1,E_2]$, 
    \[\|Ad_{V,W}(T)\|\leqslant \|V\|\|W\|\|T\|.\] So $Ad_{V,W}$ is a bounded linear map.
\end{proof}

\begin{rem}
    Use the same notation as above, we can see that
    \[\supp_{\pi_2',\pi_1'}(Ad_{V,W}(T))\subseteq \supp_{\pi_2',\pi_2}(V)\circ \supp_{\pi_2,\pi_1}(T)\circ \supp_{\pi_1',\pi_1}(W)\inv,\] 
    and if $W$ is adjointable, then $Ad_{V,W}(T)=VTW^*$.
    
    Consider the map $\calK(E_1,E_2)\ra \calK(E_1',E_2'), T\mapsto (W\circ (V\circ T)^*)^*$. This map is a well-defined bounded linear map since compact operators are automatically adjointable. By abuse of notation, we also denote this map by $Ad_{V,W}$. Then for any $e_1\in E_1, e_2\in E_2$, we have $Ad_{V,W}(\theta_{e_2,e_1})=\theta_{Ve_2,We_1}$.
\end{rem}

The following result means that some properties about the map $Ad_{V,W}$ can be locally reduces to the case for compact operators.

\begin{prop}\label{reduce adjoint map to compact case}
    Use the same notation as above, if $e_1\in \pi_1'(C_c(Z_1'))E_1'$, then 
    \[Ad_{V,W}(T)e_1=Ad_{V,W}(T\pi_1(f))e_1,\] 
    where $f\in C_c(Z_1)$ is any function whose restriction on the compact subset $\supp_{\pi_1',\pi_1}(W)\inv\circ \supp_{\pi_1'}(e_1)$ is 1.  
\end{prop}
\begin{proof}
    It is easy to see that
    \[s(\supp_{\pi_1,\pi_1}(id_{E_1}-\pi_1(f)))\subseteq Z\setminus \supp(f)\subseteq Z\setminus (\supp_{\pi_1',\pi_1}(W)\inv\circ \supp_{\pi_1'}(e_1)).\]
    Hence,
    \begin{align*}
        & \supp_{\pi_2'}(Ad_{V,W}(T-T\pi_1(f))e_1) \\
        & \subseteq \supp_{\pi_2',\pi_2}(V)\circ \supp_{\pi_2,\pi_1}(T)\circ \supp_{\pi_1,\pi_1}(id_{E_1}-\pi_1(f)) \circ \supp_{\pi_1',\pi_1}(W)\inv \circ \supp_{\pi_1'}(e_1)\\
        & =\emptyset,
    \end{align*}
    which implies that $Ad_{V,W}(T)e_1=Ad_{V,W}(T\pi_1(f))e_1$ by lemma \ref{empty support must be 0}.
\end{proof}

\begin{prop}
    Let $Z_i,Z_i'$ be locally compact Hausdorff proper $\calG$-spaces for $i=1,2,3$, $(E_i,\pi_i)$ and $(E_i',\pi_i')$ be $Z_i,\calG,A$-modules for $i=1,2,3$, $V_i\in \Hom_A(E_i,E_i')$ be properly supported $\calG$-equivariant isometries for $i=1,2,3$. Then for any $T_1\in C_\calG^*(E_1,E_2)$ and $T_2\in C_\calG^*(E_2,E_3)$, we have
\[Ad_{V_3,V_1}(T_2\circ T_1)=Ad_{V_3,V_2}(T_2)\circ Ad_{V_2,V_1}(T_1).\]
\end{prop}
\begin{proof}
    Firstly, if $e_1\in E_1, e_2,e_2'\in E_2, e_3\in E_3$, then
    \begin{align*}
        Ad_{V_3,V_2}(\theta_{e_3,e_2})\circ Ad_{V_2,V_1}(\theta_{e_2',e_1}) & = \theta_{V_3e_3, V_2 e_2}\circ \theta_{V_2 e_2', V_1 e_1}\\
        & = \theta_{V_3e_3 \langle V_2 e_2, V_2 e_2'\rangle, V_1 e_1}\\
        & = \theta_{V_3e_3 \langle e_2,e_2'\rangle, V_1 e_1}\\
        & = Ad_{V_3,V_1}(\theta_{e_3,e_2}\circ \theta_{e_2',e_1}).
    \end{align*}
    This proves that for any $T_1'\in \calK(E_1,E_2)$ and $T_2'\in \calK(E_2,E_3)$, we have $Ad_{V_3,V_2}(T_2')\circ Ad_{V_2,V_1}(T_1')=Ad_{V_3,V_1}(T_2'\circ T_1')$.

    Now for $e_1\in \pi_1'(C_c(Z_1'))E_1'$, since $V_1$ is properly supported, $\supp(V_1)\inv \circ \supp(e_1)$ is a compact subset of $Z_1$, we can select $f\in C_c(Z_1)$ such that $f|_{\supp(V_1)\inv \circ \supp(e_1)}=1$. By proposition \ref{reduce adjoint map to compact case}, $Ad_{V_2,V_1}(T_1)e_1=Ad_{V_2,V_1}(T_1\pi_1(f))e_1$ and $Ad_{V_3,V_1}(T_2T_1)e_1=Ad_{V_3,V_1}(T_2T_1\pi_1(f))e_1$.

    Let $K=\supp(T_1)\circ \supp(f)\cup \supp(V_2)\inv\circ\supp(Ad_{V_2,V_1}(T_1\pi_1(f))e_1)$, which is a compact subset of $Z_2$. Let $g\in C_c(Z_2)$ such that $g|_K=1$. By lemma \ref{properly supported operator transit compact support}, we have $\pi_2(g)T_1\pi_1(f)=T_2\pi_1(f)$. And by proposition \ref{reduce adjoint map to compact case}, $Ad_{V_3,V_2}(T_2)Ad_{V_2,V_1}(T_1\pi_1(f))e_1=Ad_{V_3,V_2}(T_2\pi_2(g))Ad_{V_2,V_1}(T_1\pi_1(f))e_1$. Hence,
    \begin{align*}
        Ad_{V_3,V_2}(T_2)\circ Ad_{V_2,V_1}(T_1)e_1 & = Ad_{V_3,V_2}(T_2)Ad_{V_2,V_1}(T_1\pi_1(f))e_1\\
        & = Ad_{V_3,V_2}(T_2\pi_2(g))Ad_{V_2,V_1}(T_1\pi_1(f))e_1\\
        & = Ad_{V_3,V_1}(T_2\pi_2(g)T_1\pi_1(f))e_1\\
        & = Ad_{V_3,V_1}(T_2 T_1\pi_1(f))e_1\\
        & = Ad_{V_3,V_1}(T_2 T_1)e_1.
    \end{align*}
    This equality holds for any $e_1\in \pi_1'(C_c(Z_1'))E_1'$, which is dense in $E_1'$ because of non-degeneracy of $\pi_1'$, so we have
    \[Ad_{V_3,V_2}(T_2)\circ Ad_{V_2,V_1}(T_1) = Ad_{V_3,V_1}(T_2 T_1).\]
\end{proof}

\begin{prop}
    Let $Z_i,Z_i'$ be locally compact Hausdorff proper $\calG$-spaces for $i=1,2$, $(E_i,\pi_i)$ be $Z_i,\calG,A$-modules, $(E_i',\pi_i')$ be $Z_i',\calG,A$-modules, $V\in \Hom_A(E_2,E_2')$, $W\in \Hom_A(E_1,E_1')$ be $\calG$-equivariant properly supported isometries, then $Ad_{V,W}(T)^*=Ad_{W,V}(T^*)$ for any $T\in C^*_\calG(E_1,E_2)$.
\end{prop}
\begin{proof}
    For any $e_1\in E_1, e_2\in E_2$, we have
    \[Ad_{V,W}(\theta_{e_2,e_1})^*=(\theta_{Ve_2,We_1})^*=\theta_{We_1,Ve_2}=Ad_{W,V}(\theta_{e_2,e_1}^*).\]
    Hence, $Ad_{V,W}(T')^*=Ad_{W,V}(T'^*)$ for any $T'\in \calK(E_1,E_2)$.

    Now for $e_1\in \pi_1'(C_c(Z_1'))E_1'$ and $e_2\in \pi_2'(C_c(Z_2'))E_2'$, we select $f\in C_c(Z_1, \mathbb R)$ such that its restriction on $\supp(W)\inv\circ \supp(e_1)$ is 1. Then we select $g\in C_c(Z_2)$ such that its restriction to $\supp(V)\inv\circ \supp(e_2)$ and $\supp(f)\circ \supp(T^*)$ is 1. Now we have
    \begin{align*}
        \langle e_2, Ad_{V,W}(T)e_1\rangle & = \langle e_2, Ad_{V,W}(T\pi_1(f))e_1\rangle\\
        & = \langle Ad_{V,W}(T\pi_1(f))^* e_2, e_1\rangle\\
        & = \langle Ad_{W,V}(\pi_1(f)T^*) e_2, e_1\rangle\\
        & = \langle Ad_{W,V}(\pi_1(f)T^*\pi_2(g))e_2, e_1\rangle\\
        & = \langle Ad_{W,V}(T^*\pi_2(g)) e_2, e_1\rangle\\
        & = \langle Ad_{W,V}(T^*) e_2, e_1\rangle,
    \end{align*}
    where the first, fourth and sixth equalities are by proposition \ref{reduce adjoint map to compact case}, the second equality is because $T\pi_1(f)$ is a compact operator, and the fifth equality is because $\pi_1(f)T^*\pi_2(g)=T^*\pi_2(g)$ by lemma \ref{properly supported operator transit compact support}. Since $\pi_1'(C_c(Z_1'))E_1'$ and $\pi_2'(C_c(Z_2'))E_2'$ are dense in $E_1$ and $E_2$ respectively, we have $Ad_{V,W}(T)^*=Ad_{W,V}(T^*)$.
\end{proof}

Using the above two propositions, we can define the adjoint *-homomorphism by weakly controlled isometries between Roe algebras.

\begin{corr}
    Let $Z_1,Z_2$ be locally compact Hausdorff proper $\calG$-spaces, $(E_i,\pi_i)$ be a $Z_i,\calG,A$-module for $i=1,2$, and let $V\in \Hom_A(E_1,E_2)$ be a properly supported $\calG$-equivariant isometry. Then the map
    \[Ad_V:=Ad_{V,V}:C_\calG^*(E_1,\pi_1)\ra C_\calG^*(E_2,\pi_2), T\mapsto (V\circ (V\circ T)^*)^*.\]
    is a *-homomorphism.
\end{corr}

\begin{prop}\label{functoriality for covering isometries}
    Let $Z_1,Z_2,Z_3, Z_1',Z_2',Z_3'$ be locally compact Hausdorff proper $\calG$-spaces, $(E_i,\pi_i)$ be a $Z_i,\calG,A$-module and $(E_i',\pi_i')$ be a $Z_i',\calG,A$-module for $i=1,2,3$, 
    \[V_1\in \Hom_A(E_1, E_2), W_1\in \Hom_A(E_1',E_2'), V_2\in\Hom_A(E_2,E_3), W_2\in \Hom_A(E_2',E_3')\] 
    are properly supported $\calG$-equivariant isometries. Then for any $T\in C_\calG^*(E_1,E_1')$, we have
    \[Ad_{V_2\circ V_1, W_2\circ W_1}(T) = Ad_{V_2,W_2}(Ad_{V_1,W_1}(T)).\]
\end{prop}
\begin{proof}
    Clearly the equality holds for $T\in \calK(E_1,E_1')$. Now for any $e\in \pi_3'(C_c(Z_3'))E_3'$, choose $h_1\in C_c(Z_2')$ such that its restriction on $\supp(W_2)\inv\circ \supp(e)$ is 1, and choose $h_2\in C_c(Z_1')$ such that its restriction on $\supp(W_2\circ W_1)\inv\circ \supp(e)$ and $\supp(W_1)\inv \circ \supp(h_1)$ is 1. Then,
    \begin{align*}
        Ad_{V_2\circ V_1, W_2\circ W_1}(T)e & = Ad_{V_2\circ V_1, W_2\circ W_1}(T\pi_1'(h_2))e\\
        & = Ad_{V_2,W_2}(Ad_{V_1,W_1}(T\pi_1'(h_2)))e\\
        & = Ad_{V_2,W_2}(Ad_{V_1,W_1}(T\pi_1'(h_2))\pi_2'(h_1))e\\
        & = Ad_{V_2,W_2}(Ad_{V_1,W_1}(T)\pi_2'(h_1))e\\
        & = Ad_{V_2,W_2}(Ad_{V_1,W_1}(T))e,
    \end{align*}
    where the second equality is because $T\pi_1'(h_2)$ is compact, and all other equalities are due to proposition \ref{reduce adjoint map to compact case}.
\end{proof}

\begin{corr}\label{functoriality for covering isometries in not general}
    Let $Z_1,Z_2,Z_3$ be locally compact Hausdorff proper $\calG$-spaces, $(E_1,\pi_1)$ be a $Z_1,\calG,A$-module, $(E_2,\pi_2)$ be a $Z_2,\calG,A$-module, $(E_3,\pi_3)$ be a $Z_3,\calG,A$-module, $V_1\in \Hom_A(E_1,E_2)$ and $V_2\in \Hom_A(E_2,E_3)$ be properly supported $\calG$-equivariant isometries. Then
    \[Ad_{V_2}\circ Ad_{V_1}=Ad_{V_2\circ V_1}: C_\calG^*(E_1,\pi_1)\ra C_\calG^*(E_3,\pi_3).\]
\end{corr}

\begin{lem}\label{pushout of repr preserve Roe alg}
    Let $Z'$ be a locally compact Hausdorff proper $\calG$-space, $f:Z\ra Z'$ be a $\calG$-equivariant proper continuous map. Let $(E,\pi)$ be a $Z,\calG,A$-module. Then $C_\calG^*(E,\pi)=C_\calG^*(E,f_*\pi)$.
\end{lem}
\begin{proof}
    Since $f$ is proper, for any $T\in \End_A(E)$, by proposition \ref{support of intertwiner} and proposition \ref{property of support of operator}, $\supp_{\pi,\pi}(T)$ is controlled if and only if $\supp_{f_*\pi,f_*\pi}(T)$ is controlled. It suffices to prove that, $T\in \End_A(E)$ is locally compact with respect to $\pi$ if and only if it is locally compact with respect to $f_*\pi$. Obviously, local compactness with respect to $\pi$ implies local compactness with respect to $f_*\pi$ (since $f^*$ is $C_0(Z')\ra C_0(Z)$). Conversely, for any $g\in C_c(Z)$, let $h\in C_c(Z')$ such that $h|_{f(\supp(g))}=1$. Then
    \[T\pi(g)=T\pi(g\cdot h\circ f)=T \circ f_*\pi(h)\circ \pi(g)\in \calK(E),\]
    similarly $\pi(g)T\in \calK(E)$ for any $g\in C_c(Z)$. So the local compactness with respect to $f_*\pi$ implies the local compactness with respect to $\pi$.
\end{proof}

\subsection{K-theory of Roe algebras}

We will build up the framework of K-theory of Roe algebras in our setting.

\begin{prop}
    Let $A$ be a \cst-algebra, $\alpha:A\ra M_2(A), a\mapsto \begin{pmatrix}
a & 0\\
0 & 0  
    \end{pmatrix}$ be the top left corner inclusion, $\beta:A\ra M_2(A), a\mapsto \begin{pmatrix}
        0 & 0 \\ 0 & a
    \end{pmatrix}$ be the bottom right corner inclusion, then $\alpha$ and $\beta$ induce the same isomorphism in K-theory.
\end{prop}
\begin{proof}
    Consider the homotopy $\iota:A\ra M_2(A)[0,1]$,
    \[\iota(a)(t)=\begin{pmatrix}
        (1-t^2)a & t\sqrt{1-t^2}a\\ t\sqrt{1-t^2}a & t^2 a
    \end{pmatrix}.\]
\end{proof}

\begin{prop}\label{two covering isometries induce same map in K}
    Let $Z_1,Z_2$ be locally compact Hausdorff proper $\calG$-spaces, $(E_1,\pi_1)$ be a $Z_1,\calG,A$-module, $(E_2,\pi_2)$ be a $Z_2,\calG,A$-module, and $V_0,V_1\in \Hom_A(E_1,E_2)$ be properly supported $\calG$-equivariant isometries. Then $Ad_{V_0}$ and $Ad_{V_1}$ induce the same map $K_*(C_\calG^*(E_1,\pi_1))\ra K_*(C_\calG^*(E_2,\pi_2))$.
\end{prop}
\begin{proof}
    It is easy to see that we can canonically identify $M_2(C_\calG^*(E_2,\pi_2))$ with $C_\calG^*(E_2^{\oplus 2}, \pi_2^{\oplus 2})$. Let $\tilde V_0=V_0\oplus 0\in \Hom_A(E_1,E_2^{\oplus 2})$ and $\tilde V_1=0\oplus V_1\in \Hom_A(E_1,E_2^{\oplus 2})$, then $\tilde V_0$ and $\tilde V_1$ are two properly supported $\calG$-equivariant isometries.

    Now let $\alpha:C_\calG^*(E_2,\pi_2)\ra C_\calG^*(E_2^{\oplus 2}, \pi_2^{\oplus 2})$ be the top left corner inclusion, $\beta:C_\calG^*(E_2,\pi_2)\ra C_\calG^*(E_2^{\oplus 2}, \pi_2^{\oplus 2})$ be the bottom right corner inclusion, then $Ad_{\tilde V_0}=\alpha\circ Ad_{V_0}$ and $Ad_{\tilde V_1}=\beta\circ Ad_{V_1}$. By the previous corollary, $\alpha$ and $\beta$ induce the same isomorphism in K-theory, so now it suffices to prove that $Ad_{\tilde V_0}$ and $Ad_{\tilde V_1}$ induce the same map in K-theory.

    If we define $\tilde V_t=\cos(\pi t/2)V_0\oplus \sin(\pi t/2)V_1$ for $t\in [0,1]$, $(\tilde V_t)_{t\in [0,1]}$ is a continuous path of properly supported $\calG$-equivariant isometries connecting $\tilde V_0$ and $\tilde V_1$. After replacing $V_0$ and $V_1$ by $\tilde V_0$ and $\tilde V_1$, we can assume that $V_0$ and $V_1$ are homotopic.

    Let $(V_t)_{t\in[0,1]}$ be a family of properly supported $\calG$-equivariant isometries that connects $V_0$ and $V_1$, then for any $T\in C_\calG^*(E_1,\pi_1)$, $t\mapsto Ad_{V_t}(T)$ is continuous in norm. Hence, $(Ad_{V_t})_{t\in [0,1]}$ is a continuous path of *-homomorphisms. Then by homotopy invariance of K-theory,
    \[(Ad_{V_0})_*=(Ad_{V_1})_*:K_*(C_\calG^*(E_1,\pi_1))\ra K_*(C_\calG^*(E_2,\pi_2)).\]
\end{proof}

\begin{defn}
    Let $Z,Z'$ be two \lch proper $\calG$-spaces, $(E,\pi)$ be a $Z,\calG,A$-module, $(E',\pi')$ be a $Z',\calG,A$-module. We say that $(E,\pi)\preceq (E',\pi')$, if there exists a $\calG$-equivariant isometry $V:E\ra E'$ such that $V$ is properly supported with respect to $\pi$ and $\pi'$.

    For a $Z,\calG,A$-module $(E,\pi)$, we say that $(E,\pi)$ is a universal $Z,\calG,A$-module, if it is stable and for any other $Z,\calG,A$-module $(E_1,\pi_1)$, we have $(E_1,\pi_1)\preceq (E,\pi)$.
\end{defn}

\begin{rem}
    The stableness in this definition is not essential. It is just for simplifying notations in the following sections.
\end{rem}

\begin{ex}
    Let $G$ be a countable discrete group, $X$ be a second countable locally compact Hausdorff proper $G$-space, $H$ be a separable infinite-dimensional $G$-Hilbert space, $\pi:C_0(X)\ra B(H)$ be a $G$-equivariant non-degenerate representation. If $(H,\pi)$ is an ample $X,G$-module in the sense of \cite[Definition 4.5.2]{willett2020higher}, then $(H,\pi)$ is a universal $X,G,\mathbb C$-module in our definition by the same argument in \cite[Section 4.5]{willett2020higher}. Such an ample module always exists by \cite[Lemma 4.5.5]{willett2020higher}. Our terminology is closer to \cite{nishikawa2021crossed}, where the term ``universal'' is also used for a similar notion.
\end{ex}

\begin{prop}\label{K-theory of Roe algebras of universal modules are isomorphic}
    Let $Z$ be a locally compact Hausdorff proper $\calG$-space, $(E_1,\pi_1)$ and $(E_2,\pi_2)$ be two universal $Z,\calG,A$-modules. Then
    \[K_*(C_\calG^*(E_1,\pi_1))\cong K_*(C_\calG^*(E_2,\pi_2)).\]
\end{prop}
\begin{proof}
    By definition of universal modules, assume that $V\in \Hom_A(E_1,E_2)$ and $W\in \Hom_A(E_2,E_1)$ are properly supported $\calG$-equivariant isometry. Then apply proposition \ref{two covering isometries induce same map in K} to the two $\calG$-equivariant properly supported isometries $V\circ W$ and $id_{E_2}$, we have
    \[(Ad_{V\circ W})_*=id:K_*(C_\calG^*(E_1,\pi_1))\ra K_*(C_\calG^*(E_1,\pi_1)).\]
    Similarly, $(Ad_{W\circ V})_*$ is also the identity map. Then using corollary \ref{functoriality for covering isometries in not general}, we have \[(Ad_{V})_*\circ (Ad_{W})_*=(Ad_{V\circ W})_*=id,\]
    \[(Ad_{W})_*\circ (Ad_{V})_*=(Ad_{W\circ V})_*=id.\]
    Hence, $(Ad_{V})_*$ and $(Ad_{W})_*$ are isomorphisms.
\end{proof}

\begin{prop}
    Let $Z$ be a \lch proper $\calG$-space. We denote the set of equivariant unitary equivalent class of $Z,\calG,A$-modules by $M(Z,\calG,A)$. Then $(M(Z,\calG,A),\preceq)$ is a well-defined directed set.
\end{prop}
\begin{proof}
    By definition of $Z,\calG,A$-modules, if $(E,\pi)$ is a $Z,\calG,A$-module, then $E$ is a countably generated Hilbert $A$-module. Hence, up to unitary isomorphism, $E$ can be seen as a submodule of $A^\infty$ by Kasparov's stabilization theorem. Therefore, $M(Z,\calG,A)$ is a set.

    Now we check that $\preceq$ is a well-defined preorder. For $(E,\pi)\in M(Z,\calG,A)$, $id_E:E\ra E$ is a $\calG$-equivariant properly supported isometry. So $\preceq$ is reflexive. For $(E_1,\pi_1)\preceq (E_2,\pi_2)$ and $(E_2,\pi_2)\preceq (E_3,\pi_3)$, assume that $V:E_1\ra E_2$ and $W:E_2\ra E_3$ are $\calG$-equivariant properly supported isometry. Then $WV:E_1\ra E_3$ is a $\calG$-equivariant isometry. By proposition \ref{property of support of operator}, $WV$ is also properly supported. Therefore, $\preceq$ is also transitive.

    Clearly, for any $(E_1,\pi_1)$ and $(E_2,\pi_2)$, we have $(E_1,\pi_1)\preceq (E_1\oplus E_2,\pi_1\oplus \pi_2)$ and $(E_2,\pi_2)\preceq (E_1\oplus E_2,\pi_1\oplus \pi_2)$. So $(M(Z,\calG,A),\preceq)$ is a directed set.
\end{proof}

By proposition \ref{two covering isometries induce same map in K}, if $(E_1,\pi_1)\preceq (E_2,\pi_2)$ in $M(Z,\calG,A)$, then this relation induces uniquely a homomorphism $K_*(C_\calG^*(E_1,\pi_1))\ra K_*(C_\calG^*(E_2,\pi_2))$. The set $\{K_*(C_\calG^*(E,\pi)):(E,\pi)\in M(Z,\calG,A)\}$ and homomorphisms in this form constitute a small thin category. If there exists a universal $Z,\calG,A$-module $(E,\pi)$, then $K_*(C_\calG^*(E,\pi))$ is a terminal object in this small category.

\begin{defn}
    Let $Z$ be a \lch proper $\calG$-space. We define
    \[KC_*(Z;\calG,A):=\varinjlim_{(E,\pi)\in M(Z,\calG,A)}K_*(C_\calG^*(E,\pi)).\]

    If $f:Z\ra Z'$ is a proper continuous map between two \lch proper $\calG$-spaces, then we define
    \[f_*:KC_*(Z;\calG,A)\ra KC_*(Z';\calG,A)\]
    as the map that sends $K_*(C_\calG^*(E,\pi))$ onto $K_*(C_\calG^*(E,f_*\pi))$ in sense of lemma \ref{pushout of repr preserve Roe alg}.
\end{defn}

\begin{prop}\label{functoriality of KC}
    Use the same notation as above.
    \begin{enumerate}
        \item The map $f_*$ is a well-defined homomorphism.
        \item The assignments above form a functor $KC_*(-;\calG,A)$ from the category of \lch proper $\calG$-spaces and $\calG$-equivariant proper continuous maps to the category of graded abelian groups.
        \item If $(E,\pi)$ is a universal $Z,\calG,A$-module, then $KC_*(Z;\calG,A)\cong K_*(C_\calG^*(E,\pi))$.
    \end{enumerate}
\end{prop}
\begin{proof}
    (1) Firstly, if $(E_1,\pi_1)\preceq (E_2,\pi_2)$ in $M(Z,\calG,A)$, assume that $V:E_1\ra E_2$ is a $\calG$-equivariant properly supported isometry, then
    \begin{align*}
        \supp_{f_*\pi_2,f_*\pi_1}(V) & \subseteq \supp_{f_*\pi_2,\pi_2}(id_{E_2})\circ \supp_{\pi_2,\pi_1}(V)\circ \supp_{\pi_1,f_*\pi_1}(id_{E_1})\\
        & \subseteq Gr(f)\circ \supp_{\pi_2,\pi_1}(V)\circ Gr(f)\inv
    \end{align*}
    is weakly controlled, by proposition \ref{summary of controlled and weakly controlled sets} proposition \ref{support of intertwiner} and proposition \ref{property of support of operator}. Hence, $(E_1,f_*\pi_1)\preceq (E_2,f_*\pi_2)$. Clearly, the following diagram commutes, where the two equalities are due to lemma \ref{pushout of repr preserve Roe alg}.
    \[
\xymatrix{
K_*(C_{\mathcal G}^*(E_1,\pi_1))
  \ar@{=}[r]
  \ar[d]_{(Ad_V)_*}
&
K_*(C_{\mathcal G}^*(E_1,f_*\pi_1))
  \ar[d]^{(Ad_V)_*}
\\
K_*(C_{\mathcal G}^*(E_2,\pi_2))
  \ar@{=}[r]
&
K_*(C_{\mathcal G}^*(E_2,f_*\pi_2)).
}
\]
Therefore, $f_*$ is a well-defined homomorphism.

(2) For a \lch proper $\calG$-space $Z$, it is clear that $(id_Z)_*$ is the identity map. For two proper continuous maps $f:Z\ra Z'$ and $g:Z'\ra Z''$ between \lch proper $\calG$-spaces, by lemma \ref{pushout of repr preserve Roe alg} we have equalities $C_\calG^*(E,\pi)=C_\calG^*(E,f_*\pi)=C_\calG^*(E,g_*f_*\pi)=C_\calG^*(E,(gf)_*\pi)$ for any $(E,\pi)\in M(Z,\calG,A)$. Hence, $g_*\circ f_*=(gf)_*:KC_*(Z;\calG,A)\ra KC_*(Z'';\calG,A)$.

(3) If $(E,\pi)$ is a universal $Z,\calG,A$-module, for any $(E',\pi')\in M(Z,\calG,A)$ such that $(E,\pi)\preceq (E',\pi')$, we can assume that $V:E\ra E'$ and $W:E'\ra E$ are $\calG$-equivariant properly supported isometries. Use the same argument as in the proof of proposition \ref{K-theory of Roe algebras of universal modules are isomorphic}, $(Ad_V)_*$ and $(Ad_W)_*$ build up the isomorphism $K_*(C_\calG^*(E,\pi))\cong K_*(C_\calG^*(E',\pi'))$. Therefore, $KC_*(Z;\calG,A)\cong K_*(C_\calG^*(E,\pi))$.
\end{proof}

\begin{rem}
For the arguments above, by the proof of proposition \ref{two covering isometries induce same map in K}, we can actually replace the K-theory functor by any functor from the category of \cst-algebras and *-homomorphisms to the category of graded abelian groups that are homotopy invariant and stable under tensor product with $\calK$.
\end{rem}

From the definition, we can see that $KC_*(Z;\calG,A)$ is defined to be an ideal receptacle of some coarse index in form of
\[\tilde \mu: \kk_*^\calG(C_0(Z),A)\ra KC_*(Z;\calG,A).\]
See \cite{Roe96} or \cite{higson2000analytic}, etc. We will introduce a groupoid-equivariant coarse index map in the next paper of this series.

We end this section with an analog of coarse geometrical invariance of the functor $KC_*(-;\calG,A)$.

\begin{defn}
    Let $Z_1,Z_2$ be two locally compact Hausdorff proper $\calG$-spaces and $f,g:Z_1\ra Z_2$ be two $\calG$-equivariant proper continuous maps. We say that $f$ is close to $g$, if
    \[Gr(f)\circ Gr(g)\inv=\{(f(z_1),g(z_1))\in Z_2\times_{\calG\units} Z_2: z_1\in Z_1\}\]
    is weakly controlled.

    We say that $Z_1$ and $Z_2$ are strongly coarsely equivalent, if there exists $\calG$-equivariant proper continuous maps $\phi:Z_1\ra Z_2$ and $\psi:Z_2\ra Z_1$, such that $\phi\circ \psi$ is close to $id_{Z_2}$, $\psi\circ \phi$ is close to $id_{Z_1}$.
\end{defn}

\begin{prop}
    Let $Z_1,Z_2$ be two locally compact Hausdorff proper $\calG$-spaces and $f,g:Z_1\ra Z_2$ be two $\calG$-equivariant proper continuous maps such that $f$ is close to $g$. Then $f_*: KC_*(Z_1;\calG,A)\ra KC_*(Z_2;\calG,A)$ and $g_*: KC_*(Z_1;\calG,A)\ra KC_*(Z_2;\calG,A)$ are the same map.
\end{prop}
\begin{proof}
    Clearly, it suffices to show that, for any $(E,\pi)\in M(Z_1,\calG,A)$, we have $(E,f_*\pi)\preceq (E,g_*\pi)$ and $(E,g_*\pi)\preceq (E,f_*\pi)$. This is because that
    \[\supp_{f_*\pi,g_*\pi}(id_{E})\subseteq Gr(f)\circ Gr(g)\inv\]
    is weakly controlled.
\end{proof}
\begin{corr}
If $Z_1$ and $Z_2$ are strongly coarsely equivalent, then $KC_*(Z_1;\calG,A)\cong KC_*(Z_2;\calG,A)$.
\end{corr}

\section{Construction of universal modules}\label{sec construction univ mod}

In this section, we will prove that, there is a sufficiently large family of proper cocompact $\calG$-spaces who admit universal modules.

\subsection{Local structure of proper actions}

Recall that, if $\Gamma$ is a countable discrete group, $X$ be a \lch proper $\Gamma$-space, then $X$ has the following local property: for any $x\in X$ and an open neighborhood $U$ of $x$, let $\Gamma_x=\{g\in \Gamma:gx=x\}$ be the stabilizer of $\Gamma$ at $x$, then there exists an open neighborhood $V$ of $x$ in $U$, such that $\overline{V}\subseteq U$, $V$ is $\Gamma_x$-invariant, and for any $g\in \Gamma\setminus \Gamma_x$, $g \overline{V}\cap \overline{V}=\emptyset$. This implies that the map
\[\Gamma\times_{\Gamma_x}\overline{V}\ra \Gamma \overline{V},\quad [g,x]\mapsto gx\]
is a $\Gamma$-equivariant homeomorphism. See \cite[Lemma A.2.7]{willett2020higher}.

There is a similar result of proper action of \'etale groupoids.

\begin{prop}\label{local structure of proper cocompact action of étale groupoid}
    \textnormal{\cite[Proposition 3.2]{bonicke2024categorical}, \cite[Proposition 2.12]{Mao26}}
    Let $\calG$ be an étale groupoid. $Z$ be a locally compact Hausdorff space with left proper action of $\calG$ with anchor map $\rho: Z\ra \calG\units$. Then for any $z_0\in Z$ and a neighborhood $U$ of $z_0$, there exists an open neighborhood $V$ of $z_0$ in $U$, and a proper open subgroupoid $\calH\subseteq \calG$, such that
    \begin{enumerate}
        \item $V\subseteq \overline{V}\subseteq U$, $\overline{V}$ is compact;
        \item $\rho(V)\subseteq \calH\units$ and $\calH\units$ is a relatively compact open of $\calG\units$;
        \item $V$ is $\calH$-invariant;
        \item For any $\gamma \in\calG\setminus \calH$, $\gamma \overline{V}\cap \overline{V}=\emptyset$.
    \end{enumerate}
    Moreover, $F=\{\gamma\in \calG^{\rho(z_0)}_{\rho(z_0)}:\gamma z_0=z_0\}$ is a finite group, it admits an action on a relatively compact open neighborhood $A$ of $\rho(z_0)$ in $\calG\units$, such that $\calH$ is isomorphic to $F\ltimes A$. We have a $\calG$-equivariant homeomorphism
    \[\calG_{\calH\units}\times_{\calH}V\ra \calG V,\quad [\gamma,z]\mapsto \gamma z.\]
\end{prop}

\begin{defn}\label{defn G-cover}
    Let $\calG$ be an étale groupoid and $Z$ be a locally compact Hausdorff proper $\calG$-space with anchor map $\rho:Z\ra \calG\units$. Assume that $(V_i)_{i\in I}$ is a family of relatively compact open subsets of $Z$, $(\calH_i)_{i\in I}$ is a family of proper open subgroupoids of $\calG$. We say that $(V_i,\calH_i)_{i\in I}$ is a $\calG$-cover of $Z$, if
    \begin{enumerate}
        \item $Z=\cup_{i\in I} \calG V_i$;
        \item $\rho(V_i)\subseteq \calH_i\units$ and $V_i$ is $\calH_i$-invariant for every $i\in I$;
        \item for every $i\in I$, there exists a finite group $F_i$ and an action of $F_i$ on $\calH_i\units$, such that $\calH_i$ is isomorphic to $F_i\ltimes \calH_i\units$;
        \item for any $\gamma\in \calG\setminus \calH_i$, $\gamma \overline{V_i}\cap \overline{V_i}=\emptyset$. 
    \end{enumerate}
\end{defn}

Using this definition, proposition \ref{local structure of proper cocompact action of étale groupoid} can be rephrased as that a $\calG$-cover of $Z$ can be chosen to be arbitrarily fine.

\begin{corr}\label{exist G-cover}
    Let $\calG$ be an étale groupoid and $Z$ be a locally compact Hausdorff proper $\calG$-space with anchor map $\rho:Z\ra \calG\units$, $U$ be a $\calG$-invariant open neighborhood of the diagonal $\Delta_Z=\{(z,z)\in Z\times_{\calG\units}Z:z\in Z\}$ in $Z\times_{\calG\units}Z$. Then there exists a $\calG$-cover $(V_i,\calH_i)_{i\in I}$ of $Z$, such that $\cup_{i\in I}\calG (\overline{V_i}\times_{\calG\units}\overline{V_i})\subseteq U$.
\end{corr}
\begin{proof}
    For every $z\in Z$, there exists a relatively compact open neighborhood $W_z$ of $z$ in $Z$ such that $\overline{W_z}\times_{\calG\units}\overline{W_z}\subseteq U$. The $\calG$-invariance of $U$ implies that $\calG(\overline{W_z}\times_{\calG\units}\overline{W_z})\subseteq U$. By proposition \ref{local structure of proper cocompact action of étale groupoid}, there exists an open neighborhood $V_z$ of $z$ in $W_z$ and a proper open subgroupoid $\calH(z)$ of $\calG$, such that $\rho(V_z)\subseteq \calH(z)\units$, $V_z$ is $\calH(z)$-stable for every $z\in Z$, and for any $\gamma\in \calG\setminus \calH(z)$, $\gamma \overline{V_z}\cap \overline{V_z}=\emptyset$. Then $(V_z,\calH(z))_{z\in Z}$ is the $\calG$-cover we want.
\end{proof}

\begin{rem}
    Clearly, when $Z$ is $\calG$-compact, we can always choose the $\calG$-cover $(V_i,\calH_i)_{i\in I}$ such that the index set $I$ is finite.
\end{rem}

\begin{lem}\label{G-compact subsets are closed}
    Let $\calG$ be a locally compact Hausdorff groupoid such that the source and range maps of $\calG$ are open, $X$ be a locally compact Hausdorff proper $\calG$-space with anchor map $\rho$. If $Z\subseteq X$ be a $\calG$-compact subset, then $Z$ is closed in $X$.
\end{lem}

\begin{proof}
    Assume that $Z=\calG K$, $(\gamma_\lambda k_\lambda)_\lambda$ is a net that converges to $x\in X$, where for each $\lambda$, $s(\gamma_\lambda)=\rho(k_\lambda)$ and $k_\lambda\in K$. Since $K$ is compact, after replacing by a subnet, we can assume that $(k_\lambda)_\lambda$ is a convergent net with limit $k\in K$. By the properness and \cite[Proposition 1.84]{goehle2009groupoid}, $(\gamma_\lambda)_\lambda$ admits also a convergent subnet. Again after replacing by a subnet, we assume that $(\gamma_\lambda)_\lambda$ is a convergent net with limit $\gamma$. Therefore, $x=\gamma k\in Z$. So we proved that $Z$ is closed in $X$.
\end{proof}

\begin{corr}\label{closure of orbit of relative compact set}
    Let $\calG$ be a locally compact Hausdorff groupoid such that the source and range maps of $\calG$ are open, $X$ be a locally compact Hausdorff proper $\calG$-space, $A$ be a subset of $X$ such that $\overline{A}$ is compact, then $\overline{\calG A}\subseteq \calG \overline{A}$.
\end{corr}
\begin{proof}
    By lemma \ref{G-compact subsets are closed}, $\calG\overline{A}$ is closed, therefore $\overline{\calG A}\subseteq \overline{\calG \overline{A}}=\calG \overline{A}$.
\end{proof}

\begin{lem}\label{orbits of opens are open}
    Let $\calG$ be a \lch groupoid such that the source and range maps of $\calG$ are open and $X$ be a \lch $\calG$-space, $U$ be an open of $X$, then $\calG U$ is open in $X$.
\end{lem}
\begin{proof}
    Let $q:X\ra X/\calG$ be the canonical quotient map. By lemma \cite[Lemma 2.30]{Tu04}, $q$ is an open map, and therefore $\calG U=q\inv(q(U))$ is open.
\end{proof}

\begin{lem}\label{refine G-cover}
    Let $\calG$ be an étale groupoid and $Z$ be a locally compact Hausdorff proper $\calG$-space with anchor map $\rho:Z\ra \calG\units$. If $Z/\calG$ is second countable, then we can choose a $\calG$-cover $(V_i,\calH_i)_{i\in I}$ of $Z$, such that $I$ is countable, and $(\calG V_i)_{i\in I}$ is a locally finite open cover of $Z$.
\end{lem}
\begin{proof}
    Let $q:Z\ra Z/\calG$ be the quotient map. Since $Z/\calG$ is locally compact Hausdorff and second countable, $Z/\calG$ is Lindel\"of and paracompact. Firstly by corollary \ref{exist G-cover}, there exists a $\calG$-cover $(U_j,\calH_j)_{j\in J}$. By paracompactness, the open cover $(q(U_j))_{j\in J}$ has a locally finite open refinement $(W_k)_{k\in K}$. Assume that for $k\in K$, $j_k\in J$ such that $W_k\subseteq q(U_{j_k})$. So for every $k$, we define $V_{k}=q\inv(W_k)\cap U_{j_k}$. So $(V_k,\calH_{j_k})_{k\in K}$ is a $\calG$-cover of $Z$ such that $(\calG V_k)_{k\in K}$ is a locally finite cover of $Z$.

    Now $(q(V_k))_{k\in K}$ is an open cover of $Z/\calG$. Since $Z/\calG$ is Lindel\"of, there exists a countable subset $I$ of $K$ such that $(q(V_i))_{i\in I}$ is a subcover. Now $(V_i,\calH_i)_{i\in I}$ is a $\calG$-cover, such that $(\calG V_i)_{i\in I}$ is a countable locally finite open cover of $Z$.
\end{proof}

\begin{lem}\label{partition of unity of G-cover}
    Let $\calG$ be an étale groupoid and $Z$ be a locally compact Hausdorff proper  $\calG$-space with anchor map $\rho:Z\ra \calG\units$. Assume that $(V_i,\calH_i)_{i\in I}$ is a $\calG$-cover of $Z$, such that $(\calG V_i)_{i\in I}$ is a locally finite cover of $Z$. Then there exists $f_i\in C(\calG V_i,[0,1])$ for each $i\in I$, such that
    \begin{enumerate}
        \item for any $z\in Z$, $\sum_{i\in I}f_i^2(z)=1$;
        \item for each $i\in I$, $f_i$ is $\calG$-invariant, and $\supp(f_i)$ is $\calG$-compact;
        \item for each $i\in I$, $f_i|_{V_i}\in C_c(V_i)$, and $f_i|_{V_i}$ is $\calH_i$-invariant.
    \end{enumerate}
\end{lem}
\begin{proof}
    Let $q:Z\ra Z/\calG$ be the quotient map. Then $(q(V_i))_{i\in I}$ is a locally finite cover of $Z/\calG$. Then there exists a partition of unity $(\psi_i)_{i\in I}$ subordinate to $(q(V_i))_{i\in I}$, that is, for each $i\in I$, $\psi_i\in C_c(Z/\calG,[0,1])$ is supported in $q(V_i)$, and for any $z\calG\in Z/\calG$, $\sum_{i\in I}\psi^2_i(z\calG)=1$. Now let $f_i\in C(\calG V_i,[0,1])$ be defined as $f_i(z)=\psi_i(q(z))$. Clearly, $f_i$ is $\calG$-invariant and $\sum_{i\in I}f_i^2(z)=\sum_{i\in I}\psi^2_i(q(z))=1$ for any $z\in Z$. The composition of the open inclusion $V_i\hookrightarrow \calG V_i$ and the quotient map $q|_{\calG V_i}:\calG V_i\ra (\calG V_i)/\calG$ is equivalent to the quotient map $q|_{V_i}: V_i\ra V_i/\calH_i$, which is proper since there exists a finite group $F_i$ acting on $\calH_i\units$ such that $\calH_i$ is isomorphic to $F_i\ltimes \calH_i\units$. Hence, $\supp(f_i|_{V_i})$ is compact and $\supp(f_i)$ is $\calG$-compact.
\end{proof}

\subsection{A decomposition technique}

Let $\calG$ be a second countable étale groupoid, $Z$ be a locally compact Hausdorff proper $\calG$-space with anchor map $\rho:Z\ra \calG\units$, $(A,\alpha)$ be a separable $\calG$-\cst-algebra. We are interested in the existence of universal $Z,\calG,A$-modules.

To explain the motivation of the technique developed in this subsection, we recall the classical case, where $Z$ is a proper metric space with coarse structure given by the metric, $\calG$ is one point and $A=\mathbb C$. Assume that $(H,\pi)$ is an ample and stable $Z$-module, The proof that it is universal does not use the full power of the ampleness, but only the fact that $H_{[V]}=\overline{\pi(C_0(V))H}$ is an infinitely dimensional separable Hilbert space for any non-empty open $V$. More precisely, we can choose a locally finite open cover $(V_i)_{i\in \mathbb N}$ of $Z$ consisting of bounded open sets $V_i$. Then for any other $Z$-module $(H',\pi')$, there exists an embedding $T_i:H'_{[V_i]}\ra H_{[V_i]}$. By using a partition of unity, there exists an isometry $H'\ra \oplus_{i\in \mathbb N}H'_{[V_i]}$ that intertwines $\pi'$ and $\oplus_{i}\pi'_{[V_i]}$ (through the map $\sqcup_{i\in \mathbb N}V_i\ra Z$). On another side, there will be an isometry $\oplus_{i\in \mathbb N}H_{[V_i]}\hookrightarrow \oplus_{i\in \mathbb N}H\cong H$. Now put them together, the composite
\[H'\ra \oplus_{i\in \mathbb N}H'_{[V_i]}\xrightarrow{\oplus_i T_i}\oplus_{i\in \mathbb N}H_{[V_i]}\ra \oplus_{i\in \mathbb N}H\cong H\]
is an isometry with support contained in $\cup_{i\in I}\overline{V_i}\times \overline{V_i}$, which is of finite propagation. C.f. \cite[Lemma 12.4.6]{higson2000analytic}.

In this subsection, we will develop a similar technique by replacing the open covers by the $\calG$-covers (definition \ref{defn G-cover}). The argument above will be compared to proposition \ref{key proposition of universal module}. Indeed, these techniques are motivated by the proof of \cite[Theorem 5.2.4]{bessi2023}.

\textbf{Notation}: let $(E,\pi)$ be a $Z,\calG,A$-module, $U$ be a $\calG$-invariant open neighborhood $U$ of the diagonal $\Delta_Z$ in $Z\times_{\calG\units}Z$, and $(V_i,\calH_i)_{i\in I}$ be a $\calG$-cover of $Z$ such that $\cup_{i\in I} \calG(\overline{V_i}\times_{\calG\units}\overline{V_i})\subseteq U$ (whose existence is proved by corollary \ref{exist G-cover}), and $(\calG V_i)_{i\in I}$ is a \textbf{countable locally finite} cover of $Z$. Fix a $(Z,\calG,A)$-module $(E,\pi)$, 
\[E_{[V_i]}=\overline{\pi(C_0(V_i))E}=\overline{\pi(C_0(V_i))E|_{\calH_i\units}}\cong C_0(V_i)\otimes_{\pi}E|_{\calH_i\units}\]
is naturally a $\calH_i$-Hilbert $A|_{\calH_i}$-module. Let $W\in \calL(s^*E,r^*E)$ be the action of $\calG$ on $E$.

By lemma \ref{partition of unity of G-cover}, we can fix a ``partition of unity'' subordinate to the $\calG$-cover $(V_i,\calH_i)_{i\in I}$, that is, we choose $f_i\in C(\calG V_i,[0,1])$ for each $1\leqslant i\leqslant n$, such that $f_i$ is $\calG$-invariant, $\supp(f_i)$ is $\calG$-compact and $\sum_{i\in I} f^2_i(z)=1$ for any $z\in Z$. Let $g_i=f_i|_{V_i}$, which is an element of $C_c(V_i)$. By extension by zero, we also see $g_i$ as an element of $C_c(Z)$.

We denote the open inclusion $V_i\hookrightarrow Z$ by $\iota_i$ and the open inclusion $\calG V_i\hookrightarrow Z$ by $\tilde\iota_i$. Then 
\[(E_{[V_i]},\pi_{[V_i]})\]
is a well-defined $V_i,\calH_i,A|_{\calH_i}$-module (definition \ref{operation braquet}), and 
\[(\infl_{\calH_i}^\calG(E_{[V_i]}),\tilde\iota_{i,*}(\infl_{\calH_i}^\calG(\pi_{[V_i]})))\] 
is a well-defined $Z,\calG,A$-module (definition \ref{defn operation inflation} and definition \ref{operation pushout}). We denote $\calG_{\calH_i\units}$ by $\Omega_i$ for convenience of writing. More precisely, $\tilde\iota_{i,*}(\infl_{\calH_i}^\calG(\pi_{[V_i]}))$ is the composition
\[C_0(Z)\ra C_b(\calG V_i)\ra \calL(\infl_{\calH_i}^\calG(E_{[V_i]})),\]
where the first arrow is restriction map, the second arrow is the strict continuous extension of the representation $\infl_{\calH_i}^\calG(\pi_{[V_i]}): C_0(\calG V_i)\cong \ind_{\Omega_i}C_0(V_i)\ra \calL(\infl^\calG_{\calH_i}(E_{[V_i]}))$.

These are the notations we will use in this subsection. The key result is the following.

\begin{prop}\label{key proposition of universal module}
    Use the same notation as above, let $(E',\pi')$ be another stable $Z,\calG,A$-module. If for every $i\in I$, there exists an $\calH_i$-equivariant isometry $T_i\in \Hom_A(E_{[V_i]}, E'_{[V'_i]})$, then there exists a $\calG$-equivariant weakly controlled isometry $V\in \Hom_A(E,E')$ such that $\supp_{\pi',\pi}(V)\subseteq U$.
\end{prop}

We will divide the construction in several lemmas.

\begin{lem}\label{construction of Phi_i}
    Use the same notation as above, we define $\Phi_i:E\ra \infl_{\calH_i}^\calG (E_{[V_i]})$ as the linear map such that for any $e\in \pi(C_c(Z))E$,
    \[\Phi_i(e)(\gamma)=\pi_{s(\gamma)}(g_i|_{V_{i,s(\gamma)}}) W_\gamma^*(e(r(\gamma))).\]
    Then for each $1\leqslant i\leqslant n$, $\Phi_i:E\ra \infl_{\calH_i}^\calG (E_{[V_i]})$ is a well-defined $\calG$-equivariant operator.
\end{lem}
\begin{proof}
    Firstly we need to check that for every $e\in \pi(C_c(Z))E$, $\Phi_i(e)$ is a well-defined element of $\ind_{\Omega_i,c}(E_{[V_i]})$. For $x\in \calH_i\units$, we can canonically identify $E_{[V_i],x}$ with $\overline{\pi_x(C_0(V_{i,x}))E_x}$. Let $\ca E_i=\sqcup_{x\in \calH_i\units}E_{[V_i],x}$ be the associated Hilbert bundle over $\calH_i\units$. We know that by conditions in lemma \ref{partition of unity of G-cover}, $g_i$ is an element of $C_c(V_i)$, hence for every $\gamma\in \Omega_i$, $\Phi_i(e)(\gamma)\in E_{[V_i],s(\gamma)}$. The map $\gamma\mapsto \pi_{s(\gamma)}(g_i|_{V_{i,s(\gamma)}})$ is an element of $\Gamma_b(\Omega_i, s^*\calL(\ca E))$, and $\gamma\mapsto W_\gamma^*(e(r(\gamma)))$ is an element of $\Gamma_b(\Omega_i,s^*\ca E)$ by proposition \ref{determine continuity of groupoid action}, hence $\Phi_i(e)$ is an element of $\Gamma_b(\Omega_i, s^*\ca E)$ by proposition \ref{contiuous action of adjointable operator bundle}. Apply corollary \ref{Hilbert bundle of submodule} to $s|_{\Omega_i}^*E_{[V_i]}$ and $s|_{\Omega_i}^*E$, we can therefore see $\Phi_i(e)$ as an element of $\Gamma_b(\Omega_i,s^*\ca E_i)$.

    For any $(\gamma,h)\in \Omega_i\times_{s,\calH_i\units,r}\calH_i$, we have
    \begin{align*}
        \Phi_i(e)(\gamma h) & = \pi_{s(h)}(g_i|_{V_i,s(h)}) W_{\gamma h}^*(e(r(\gamma)))\\
        & = \pi_{s(h)}(g_i|_{V_i,s(h)}) W_h^* W_{\gamma}^*(e(r(\gamma)))\\
        & = W_h^* \pi_{s(\gamma)}(g_i|_{V_i,s(\gamma)})W_{\gamma}^*(e(r(\gamma)))\\
        & = W_h^*\Phi_i(e)(\gamma),
    \end{align*}
    where the second equality is because that $g_i$ is $\calH_i$-equivariant.

    Now we need to show that the map $\Omega_i/\calH_i\ra \mathbb R,\gamma\calH_i\mapsto \|\Phi_i(e)(\gamma)\|_{E_{s(\gamma)}}$ is compactly supported. Let $K_i=\{\gamma\in \calG:\gamma \supp(g_i)\cap \supp_{\pi}(e)\neq \emptyset\}$, which is a compact subset of $\calG$ since the action of $\calG$ on $Z$ is proper. If $\gamma\in K_i$, then $\gamma\supp(g_i)\neq \emptyset$, so $s(\gamma)\in \rho(\supp(g_i))\subseteq \calH_i\units$. That is, we have $K_i\subseteq \Omega_i=\calG_{\calH_i\units}$. Since $\|\Phi_i(e)(\gamma)\|=\|\pi_{r(\gamma)}(g_i|_{V_{i,s(\gamma)}}(\gamma\inv-))e(r(\gamma))\|$, $\|\Phi_i(e)(\gamma)\|\neq 0$ only if $\gamma\supp(g_i)\cap \supp_{\pi}(e)\neq \emptyset$, i.e. $\gamma\in K_i$.  Now let $q_i:\Omega_i\ra \Omega_i/\calH_i$ be the quotient map, the map $\Omega_i/\calH_i\ra \mathbb R,\gamma\calH_i\mapsto \|\Phi_i(e)(\gamma)\|_{E_{s(\gamma)}}$ is supported in $q(K_i)$, which is compact.

    So $\Phi_i(e)$ is a well-defined element in $\ind_{\Omega_i,c}(E_{[V_i]})$. Next, we will show that $\Phi_i$ is bounded $A$-linear map. For any $e\in \pi(C_c(Z))E$ and $x\in \calG\units$, 
    \begin{align*}
        \|\Phi_i(e)\|^2 & = \sup_{x\in \calG\units}\|\langle \Phi_i(e),\Phi_i(e)\rangle_{\infl_{\calH_i}^\calG(E_{[V_i]})}(x)\|_{A_x}\\ 
        & = \sup_{x\in \calG\units}\|\sum_{\gamma\calH_i\in \Omega_i^x/\calH}\langle \Phi_i(e)(\gamma), \Phi_i(e)(\gamma)\rangle_{E_{s(\gamma)}}\|_{A_x}\\
        & \leqslant \sup_{x\in \calG\units}\sum_{\gamma\calH_i\in \Omega_i^x/\calH}\| \pi_{x}(g_i|_{V_{i,s(\gamma)}}(\gamma\inv-))e(x)\|^2\\
        & \leqslant \sup_{x\in \calG\units}\#(q(K_i)\cap \tilde r_i\inv(x))\cdot \sup_{x\in \calG\units}\|e(x)\|^2,
    \end{align*}
    where map $\tilde r_i:\Omega_i/\calH_i\ra \calG\units, \gamma\calH_i\mapsto r(\gamma)$ is a local homeomorphism, by lemma \ref{uniformly finite fiber}, $C_i:=\sup_{x\in \calG\units}\#(q(K_i)\cap \tilde r_i\inv(x))<\infty$. The above inequality implies that $\|\Phi_i(e)\|\leqslant \sqrt{C_i}\|e\|$ for any $e\in \pi(C_c(Z))E$. So $\Phi_i$ is a bounded map.

    For any $e\in \pi(C_c(Z))E$ and $a\in A$, $\gamma\in \Omega_i$,
    \begin{align*}
        \Phi_i(ea)(\gamma) & = \pi_{s(\gamma)}(g_i|_{V_{i,s(\gamma)}})W_\gamma^*(e(r(\gamma))a(r(\gamma)))\\
        & = \Phi_i(e)(\gamma)\alpha_{\gamma}\inv(a(r(\gamma)))\\
        & = (\Phi_i(e)\cdot a)(\gamma).
    \end{align*}
    (See definition \ref{defn of inflation} for the last equality.) So $\Phi_i$ is a bounded $A$-linear map.

    Clearly, the fiber of $\Phi_i$ at $x\in \calG\units$ is
    \[\Phi_{i,x}:E_x\ra (\infl^\calG_{\calH_i}(E_{[V_i]}))_x,\]
    such that for any $e\in \pi_x(C_c(Z_x))E_x$, $\Phi_{i,x}(e)\in \ind_{\Omega_i^x,c}(E_{[V_i],x})$ that
    \[\Phi_{i,x}(e)(\gamma)=\pi_{s(\gamma)}(g_i|_{V_{i,s(\gamma)}})W_\gamma^*e.\]

    For any $\gamma'\in \calG$, $\gamma\in \Omega_i^{r(\gamma')}$, $e\in \pi_{s(\gamma')}(C_c(Z_{s(\gamma')}))E_{s(\gamma')}$,
    \begin{align*}
        V_{\gamma'}^{\infl}(\Phi_{i,s(\gamma')}(e)|_{\Omega_i^{r(\gamma')}})(\gamma) & = \Phi_{i,s(\gamma')}(e)({\gamma'}\inv \gamma)\\
        & = \pi_{s(\gamma)}(g_i|_{V_{i,s(\gamma)}})W_\gamma^* W_{\gamma'}e\\
        & = \Phi_{i,r(\gamma')}(W_{\gamma'}e)(\gamma).
    \end{align*}
    That is, $V_{\gamma'}^{\infl}\circ \Phi_{i,s(\gamma')}=\Phi_{i,r(\gamma')}\circ W_{\gamma'}$. The operator $\Phi_i$ is therefore $\calG$-equivariant.
\end{proof}

\begin{lem}\label{construction of isometry Psi_i}
    Use the same notation as above, for each $1\leqslant i\leqslant n$, we define the map $\Psi_i: \infl_{\calH_i}^\calG (E[V_i])\ra E$ as, for any $\xi\in \ind_{\Omega_i,c}(E[V_i])$, $x\in \calG\units$,
    \[\Psi_i(\xi)(x)=\sum_{\gamma\calH_i\in \Omega_i^x/\calH_i} W_\gamma\xi(\gamma).\]
    Then $\Psi_i$ is a well-defined $\calG$-equivariant isometry.
\end{lem}
\begin{proof}
    Since $\xi\in \ind_{\Omega_i,c}(E[V_i])$ is $\calH_i$-equivariant, the map $\gamma\calH_i\mapsto W_{\gamma}\xi(\gamma)$ is well-defined, and is an element of $\Gamma_c(\Omega_i/\calH_i,\tilde r^*\ca E)$. (Here $\tilde r$ is the map $\Omega_i/\calH_i\ra \calG\units, \gamma\calH_i\mapsto r(\gamma)$, which is a local homeomorphism.) Then by lemma \ref{sum of fiber}, $\Psi_i(\xi)$ is a well-defined element of $\Gamma_c(\calG\units, \ca E)\subseteq E$.

    For any $\xi\in \ind_{\Omega_i,c}(E[V_i])$, $a\in A$, $x\in \calG\units$,
    \begin{align*}
        \Psi_i(\xi\cdot a)(x) & = \sum_{\gamma\calH_i\in \Omega_i^x/\calH_i} W_\gamma(\xi(\gamma)\alpha_{\gamma\inv}(a(r(\gamma))))\\
        & = \sum_{\gamma\calH_i\in \Omega_i^x/\calH_i} W_\gamma(\xi(\gamma))a(x)\\
        & = \Psi_i(\xi)(x)a(x).
    \end{align*}
    So $\Psi_i$ is $A$-linear.

    Notice that $(V_i,\calH_i)_i$ is a $\calG$-cover, so if $\gamma\calH_i,\gamma'\calH_i$ are two different elements of $\Omega_i^x/\calH_i$, we have $\overline{V_{i,s(\gamma')}}\cap \overline{{\gamma'}\inv \gamma V_{i,s(\gamma)}}=\emptyset$. And for $\xi\in \ind_{\Omega_i,c}E_{[V_i]}$, we will have
    \[\supp_{\pi_{2,s(\gamma')}}(\xi(\gamma'))\subseteq \overline{V_{i,s(\gamma')}},\]\[\supp_{\pi_{2,s(\gamma')}}(W_{\gamma'}^*W_\gamma\xi(\gamma))\subseteq \overline{{\gamma'}\inv \gamma V_{i,s(\gamma)}}.\]
    Therefore, by corollary \ref{elements with empty support intersection are orthogono},
    \[\langle W_\gamma\xi(\gamma), W_{\gamma'}\xi(\gamma')\rangle = \langle W_{\gamma'}^*W_\gamma\xi(\gamma), \xi(\gamma')\rangle=0.\]
    This implies that,
    \begin{align*}
        \langle \Psi_i(\xi),\Psi_i(\xi)\rangle(x) & = \sum_{\gamma\calH_i\in \Omega_i^x/\calH_i}\sum_{\gamma'\in \Omega_i^x/\calH_i}\langle W_\gamma\xi(\gamma), W_{\gamma'}\xi(\gamma')\rangle\\
        & = \sum_{\gamma\calH_i\in \Omega_i^x/\calH_i}\langle W_\gamma\xi(\gamma),W_\gamma\xi(\gamma)\rangle\\
        & = \sum_{\gamma\calH_i\in \Omega_i^x/\calH_i}\langle \xi(\gamma),\xi(\gamma)\rangle\\
        & = \langle \xi,\xi\rangle_{\infl^\calG_{\calH_i}(E_{[V_i]})_x}.
    \end{align*}
    Hence, $\Psi_i$ is isometric. At $x\in \calG\units$, the fiber of $\Psi_i$ is
    \[\Psi_{i,x}: \infl_{\calH_i}^\calG(E_{[V_i]})_x\ra E_x,\]
    such that for any $\xi\in \ind_{\Omega^x_i,c}(E_{[V_i]})\subseteq \infl_{\calH_i}^\calG(E_{[V_i]})_x$,
    \[\Psi_{i,x}(\xi) = \sum_{\gamma\calH_i\in \Omega_i^x/\calH_i} W_\gamma \xi(\gamma).\]

    Now for any $\gamma'\in \calG$, $\xi\in \ind_{\Omega^{s(\gamma')},c}(E_{[V_i]})\subseteq \infl_{\calH_i}^\calG(E_{[V_i]})_{s(\gamma')}$,
    \begin{align*}
        \Psi_{i,r(\gamma')}(V_{\gamma'}^{\infl}(\xi)) & = \sum_{\gamma\calH_i\in \Omega_i^{r(\gamma')}/\calH_i} W_\gamma\xi({\gamma'}\inv \gamma)\\
        & = \sum_{\gamma\calH_i\in \Omega_i^{r(\gamma')}/\calH_i} W_{\gamma'}W_{{\gamma'}\inv\gamma}\xi({\gamma'}\inv \gamma)\\
        & = W_{\gamma'}(\sum_{\gamma\calH_i\in \Omega_i^{s(\gamma')}/\calH_i}W_\gamma\xi(\gamma))\\
        & = W_{\gamma'}\Psi_{i,s(\gamma')}(\xi).
    \end{align*}
    That is, $\Psi_{i,r(\gamma')}\circ V_{\gamma'}^{\infl}=W_{\gamma'}\Psi_{i,s(\gamma')}$, so $\Psi_i$ is $\calG$-equivariant.
\end{proof}

\begin{rem}
    The map $\Psi_i$ here is not necessarily adjointable.
\end{rem}

Since $\sum_{i\in I} f_i^2(z)=1$ for any $z\in Z$, we can see that, for any $x\in \calG\units$ and $z\in Z_x$, 
\[\sum_{i\in I} \sum_{\gamma\calH_i\in \Omega_i^x/\calH_i} g_i^2|_{V_{i,s(\gamma)}}(\gamma\inv z)\]
is a finite sum and equal to 1. Hence, we can write
\[\sum_{i\in I} \sum_{\gamma\calH_i\in \Omega_i^x/\calH_i} \pi_x(g_i^2|_{V_{i,s(\gamma)}}(\gamma\inv-))=id_{E_x}\]
by seeing the left-hand side as a net in $\calL(E_x)$ (directed by all finite subsets of the set $\{(i,\gamma\calH_i):i\in I,\gamma\calH_i\in \Omega^x_i/\calH_i\}$ that converges to $id_{E_x}$ in strict topology.

\begin{lem}\label{Phi is adjointable isometry}
    Use the same notation as above, we define the map
    \[\Phi:E\ra \oplus_{i\in I} \infl^\calG_{\calH_i}(E_{[V_i]}), e\mapsto (\Phi_i(e))_{i\in I}, \]
    Then $\Phi$ is a well-defined $\calG$-equivariant isometry.
\end{lem}
\begin{proof}
    For any $e\in \pi(C_c(Z))E$, there is only finitely many $i\in I$ such that $\calG V_i\cap \supp_\pi(e)\neq \emptyset$. While if $\calG V_i\cap \supp_{\pi}(e)=\emptyset$, then $\gamma\inv V_{i,s(\gamma)}\cap \supp_{\pi}(e(r(\gamma)))=\emptyset$ for any $\gamma\in \Omega_i$, which induces that $\Phi_i(e)=0$. Hence, any $e\in \pi(C_c(Z))E$, there is only finitely many $i\in I$ such that $\Phi_i(e)\neq 0$. This makes $\Phi$ a well-defined map on $\pi(C_c(Z))E$, which is clearly $A$-linear.

    For $e_1,e_2\in \pi(C_c(Z))E$, $x\in \calG\units$,
    \begin{align*}
        \langle \Phi(e_1),\Phi(e_2)\rangle_{\oplus_i \infl^\calG_{\calH_i}E_{[V_i]}}(x) & = \sum_{i\in I}\sum_{\gamma\calH_i\in \Omega_i^x/\calH_i}\langle \pi_{s(\gamma)}(g_i|_{V_{i,s(\gamma)}})W_\gamma^*e_1(x),\pi_{s(\gamma)}(g_i|_{V_{i,s(\gamma)}})W_\gamma^*e_2(x)\rangle_{E_x}\\
        & = \langle e_1(x), \sum_{i\in I} \sum_{\gamma\calH_i\in \Omega_i^x/\calH_i} \pi_x(g_i^2|_{V_{i,s(\gamma)}}(\gamma\inv-))e_2(x)\rangle_{E_x}\\
        & = \langle e_1,e_2\rangle_E(x).
    \end{align*}
    
    So $\Phi$ extends to an isometry, which is obviously $\calG$-equivariant by lemma \ref{construction of Phi_i}.
\end{proof}

\begin{rem}
If $I$ is finite (for example, $Z$ is $\calG$-compact), we can construct a well-defined map
\[S:\oplus_{i\in I} \infl^\calG_{\calH_i}(E_{[V_i]})\ra E,\]
\[S((\xi_i)_i)(x)=\sum_{i\in I}\sum_{\gamma\calH_i\in \Omega_i^x/\calH_i} W_\gamma(\pi_{s(\gamma)}(g_i|_{V_{i,s(\gamma)}})\xi_i(\gamma))\]
for $(\xi_i)_i\in \oplus_{i\in I} \ind_{\Omega_i,c}(E_{[V_i]})\subseteq \infl^\calG_{\calH_i}(E_{[V_i]})$, $x\in \calG\units$. In this case $\Phi$ and $S$ are adjointable and $\Phi^*=S$.
\end{rem}

\begin{lem}\label{Psi, Phi are intertwiners}
    Use the same notation as above, let $\iota: \sqcup_{i\in I} \calG V_i\ra Z$ be defined by disjoint union of the open inclusions $\tilde\iota_i:\calG V_i\ra Z$. Let $\pi_\oplus:=\iota_*(\oplus_{i\in I} \infl^\calG_{\calH_i}(\pi_{[V_i]}))$. That is, $\pi_\oplus$ is the following composition
    \[C_0(Z)\xrightarrow{\iota^*} \oplus_{i\in I} C_b(\calG V_i)\ra \oplus_{i\in I} \calL(\infl^\calG_{\calH_i}(E_{[V_i]})),\]
    where the second arrow is direct sum of the strict continuous extensions of the inflation representations $\infl^\calG_{\calH_i}(\pi_{[V_i]}): C_0(\calG V_i)\cong \ind_{\Omega_i}C_0(V_i)\ra \calL(\infl^\calG_{\calH_i}(E_{[V_i]}))$.
    \begin{enumerate}
        \item The operator $\Phi_i$ intertwines $\pi$ and $\tilde\iota_{i,*}(\infl^\calG_{\calH_i}(\pi_{[V_i]}))$.
        \item The isometry $\Phi$ intertwines $\pi$ and $\pi_\oplus$.
        \item The isometry $\Psi_i$ intertwines $\tilde\iota_{i,*}(\infl^\calG_{\calH_i}(\pi_{[V_i]}))$ and $\pi$.
    \end{enumerate}
\end{lem}
\begin{proof}
    (1) Recall that we have the isomorphism
    \[C_0(\calG V_i)\cong \ind_{\Omega_i}C_0(V_i), f\mapsto [\Omega_i\ni \gamma\mapsto f|_{\gamma V_{i,s(\gamma)}}(\gamma-)\in C_0(V_{i,s(\gamma)})]\in \ind_{\Omega_i}C_0(V_i).\]
    
    Then a function $\psi\in C_0(\calG V_i)$ is seen as the map $\Omega_i\ra s^*(\sqcup_{\calH_i\units}C_0(V_{i,x})), \gamma\mapsto \psi|_{\gamma V_{i,s(\gamma)}}(\gamma-)\in C_0(V_{i,s(\gamma)})$. So for any $\xi_i\in \infl^\calG_{\calH_i}(E_{[V_i]})$, by definition \ref{defn inflation of representation}
    \begin{align*}
        [(\infl^\calG_{\calH_i}\pi_{[V_i]})(\psi)\xi_i](\gamma) & =\pi_{[V_i],s(\gamma)}(\psi|_{\gamma V_{i,s(\gamma)}}(\gamma-))\xi_i(\gamma)
    \end{align*}
    for any $\gamma\in \Omega_i$.

    Let $\sqcup_{x\in \calH_i\units}C_b(V_{i,x})$ be equipped with topology of bundle of adjointable operators associated to $C_b(V_i)=M(C_0(V_i))$. For any $\phi\in C_0(Z)$, by proposition \ref{operators on induced modules},  the function $\phi|_{\calG V_i}\in C_b(\calG V_i)$ is then identified with the $\calH_i$-equivariant section $[\gamma\mapsto \phi|_{\gamma V_{i,s(\gamma)}}(\gamma-)]\in \Gamma_b(\Omega_i, s^*(\sqcup_{x\in \calH_i\units}C_b(V_{i,x})))$. Notice that here by $\phi|_{\gamma V_{i,s(\gamma)}}(\gamma-)$ we mean an element of $C_b(V_{i,s(\gamma)})$.

    Let $\widetilde{\infl^\calG_{\calH_i}\pi_{[V_i]}}: C_b(\calG V_i)\ra \calL(\infl^\calG_{\calH_i}E_{[V_i]})$ be the strictly continuous extension of $\infl^\calG_{\calH_i}\pi_{[V_i]}$. Then we have
    \begin{align*}
        [\big(\tilde\iota_{i,*}(\infl^\calG_{\calH_i}(\pi_{[V_i]}))(\phi)\big)\Phi_i (e)](\gamma) & = [\big(\widetilde{\infl^\calG_{\calH_i}(\pi_{[V_i]})}(\phi|_{\calG V_i})\big)\Phi_i (e)](\gamma) \\
        & = \widetilde {\pi_{[V_i]}}_{s(\gamma)}(\phi|_{\gamma V_{i,s(\gamma)}}(\gamma-))\Phi_i(e)(\gamma)\\
        & = \widetilde {\pi_{[V_i]}}_{s(\gamma)}(\phi|_{\gamma V_{i,s(\gamma)}}(\gamma-))\pi_{s(\gamma)}(g_i|_{V_{i,s(\gamma)}})W_\gamma^*(e(r(\gamma)))\\
        & = \widetilde {\pi_{[V_i]}}_{s(\gamma)}(\phi|_{\gamma V_{i,s(\gamma)}}(\gamma-))\pi_{[V_i],s(\gamma)}(g_i|_{V_{i,s(\gamma)}})W_\gamma^*(e(r(\gamma)))\\
        & = \pi_{[V_i],s(\gamma)}(g_i|_{V_{i,s(\gamma)}}\cdot \phi|_{\gamma V_{i,s(\gamma)}}(\gamma-))W_\gamma^*(e(r(\gamma)))\\
        & = \pi_{s(\gamma)}(g_i|_{V_{i,s(\gamma)}}\cdot \phi|_{\gamma V_{i,s(\gamma)}}(\gamma-))W_\gamma^*(e(r(\gamma)))\\
        & = \pi_{s(\gamma)}(g_i|_{V_{i,s(\gamma)}}\cdot \phi|_{Z_{r(\gamma)}}(\gamma-))W_\gamma^*(e(r(\gamma)))\\
        & = \pi_{s(\gamma)}(g_i|_{V_{i,s(\gamma)}}) W_\gamma^* \pi_{r(\gamma)}(\phi|_{Z_{r(\gamma)}})e(r(\gamma))\\
        & = \Phi_i(\pi(\phi)e)(\gamma),
    \end{align*}
    where $\gamma \in \Omega_i$, $e\in \pi(C_c(Z))E$, $\phi\in C_0(Z)$, and the seventh equality is because $\phi|_{Z_{r(\gamma)}}(\gamma-)\cdot g_i|_{V_{i,s(\gamma)}}= \phi|_{\gamma V_{i,s(\gamma)}}(\gamma-)\cdot g_i|_{V_{i,s(\gamma)}}$ as elements of $C_0(Z_{r(\gamma)})$. Hence, we can conclude that $\Phi_i$ intertwines $\pi$ and $\tilde\iota_{i,*}(\infl^\calG_{\calH_i}(\pi_{[V_i]}))$.

    (2) follows directly from (1).

    (3) For any $\phi\in C_0(Z)$, $\xi\in \infl^\calG_{\calH_i}(E_{[V_i]})$, $x\in \calG\units$, 
    \begin{align*}
        \Psi_i(\tilde\iota_{i,*}(\infl^\calG_{\calH_i}(\pi_{[V_i]}))(\phi)\xi)(x) & = \sum_{\gamma\calH_i\in \Omega_i^x/\calH_i} W_\gamma [\widetilde{\infl^\calG_{\calH_i}(\pi_{[V_i]})}(\phi|_{\calG V_i})\xi](\gamma)\\
        & = \sum_{\gamma\calH_i\in \Omega_i^x/\calH_i} W_\gamma \widetilde{\pi_{[V_i]}}_{s(\gamma)}(\phi|_{\gamma V_{i,s(\gamma)}}(\gamma-))\xi(\gamma)\\
        & = \sum_{\gamma\calH_i\in \Omega_i^x/\calH_i} W_\gamma \pi_{s(\gamma)}(\phi|_{Z_{r(\gamma)}}(\gamma-))\xi(\gamma)\\
        & = \sum_{\gamma\calH_i\in \Omega_i^x/\calH_i} \pi_{x}(\phi|_{Z_{x}})W_\gamma \xi(\gamma)\\
        & = \pi_{x}(\phi|_{Z_{x}}) \Psi_i(\xi)(x)\\
        & = [\pi(\phi)\Psi_i(\xi)](x),
    \end{align*}
    where the third line is because $\xi(\gamma)\in \overline{\pi_{s(\gamma)}(C_0(V_{i,s(\gamma)})E_{s(\gamma)}}$. Hence, $\Psi_i$ intertwines $\tilde\iota_{i,*}(\infl^\calG_{\calH_i}(\pi_{[V_i]}))$ and $\pi$.
\end{proof}

\begin{lem}\label{direct sum of inflation of isometry}
    Use the same notation as above, let $(E',\pi')$ be another $Z,\calG,A$-module. For each $1\leqslant i\leqslant n$, let $T_i\in \calL(E_{[V_i]},E'_{[V_i]})$ be a $\calH_i$-equivariant isometry. Let 
    \[T=\oplus_{i\in I} \infl_{\calH_i}^\calG(T_i): \oplus_{i\in I} \infl_{\calH_i}^\calG (E_{[V_i]})\ra \oplus_{i\in I} \infl_{\calH_i}^\calG (E'_{[V_i]}),\]
    then $T$ is a $\calG$-equivariant isometry, and the support of $T$ with respect to $\pi_\oplus=\iota_*(\oplus_{i\in I} \infl^\calG_{\calH_i}(\pi_{[V_i]}))$ and $\pi'_\oplus=\iota_*(\oplus_{i\in I} \infl^\calG_{\calH_i}(\pi'_{[V_i]}))$ is contained in $\cup_{i\in I} \calG(\overline{V_i}\times_{\calG\units}\overline{V_i})$. Moreover, $T$ is weakly controlled.
\end{lem}
\begin{proof}
    For each $1\leqslant i \leqslant n$, for any $\xi_1,\xi_2\in \ind_{\Omega_i,c}(E_{[V_i]})\subseteq \infl^\calG_{\calH_i}(E_{[V_i]})$, for any $x\in \calG\units$,
    \begin{align*}
        & \langle \infl_{\calH_i}^\calG(T_i)\xi, \infl_{\calH_i}^\calG(T_i)\xi'\rangle_{\infl^\calG_{\calH_i}(E'_{[V_i]})}(x)\\
        & = \sum_{\gamma\calH_i\in \Omega_i/\calH_i} \alpha_\gamma(\langle (\infl_{\calH_i}^\calG(T_i)\xi_1)(\gamma),(\infl_{\calH_i}^\calG(T_i)\xi_2)(\gamma)\rangle_{E_{s(\gamma)}})\\
        & = \sum_{\gamma\calH_i\in \Omega_i/\calH_i} \alpha_\gamma(\langle T_{s(\gamma)}\xi_1(\gamma),T_{s(\gamma)}\xi_2(\gamma)\rangle_{E_{s(\gamma)}})\\
        & = \sum_{\gamma\calH_i\in \Omega_i/\calH_i} \alpha_\gamma(\langle \xi_1(\gamma),\xi_2(\gamma)\rangle_{E_{s(\gamma)}})\\
        & = \langle \xi_1,\xi_2\rangle_{\infl^\calG_{\calH_i}(E'_{[V_i]})}(x).
    \end{align*}
    Therefore, $\infl_{\calH_i}^\calG(T_i)$ is also an isometry, and $T$ is also an isometry. It is easy to check that $T$ is $\calG$-equivariant.
    
    By proposition \ref{support of inflation}, for $1\leqslant i\leqslant n$,
    \[\supp_{\infl^\calG_{\calH_i}(\pi'_{[V_i]}),\infl^\calG_{\calH_i}(\pi_{[V_i]})}(\infl_{\calH_i}^\calG(T_i))\subseteq \calG\supp_{\pi',\pi}(T)\subseteq \calG(V_i\times_{\calG\units}V_i).\]

    For convenience, we write
    \[\pi_{\boxplus}:=\oplus_{i\in I} \infl^\calG_{\calH_i}(\pi_{[V_i]}): C_0(\sqcup_{i\in I} \calG V_i)\ra \calL(\oplus_{i\in I} \infl_{\calH_i}^\calG (E_{[V_i]})),\]
    \[\pi'_{\boxplus}:=\oplus_{i\in I} \infl^\calG_{\calH_i}(\pi'_{[V_i]}): C_0(\sqcup_{i\in I} \calG V_i)\ra \calL(\oplus_{i\in I} \infl_{\calH_i}^\calG (E'_{[V_i]})),\]
    then clearly
    \[\supp_{\pi'_\boxplus, \pi_\boxplus}(T)\subseteq \sqcup_{i\in I} \calG(V_i\times_{\calG\units}V_i)\subseteq \sqcup_{i\in I} (\calG V_i)\times_{\calG\units}(\calG V_i).\]

    By proposition \ref{support of intertwiner} and proposition \ref{property of support of operator},
    \begin{align*}
        \supp_{\pi'_\oplus, \pi_\oplus}(T) & \subseteq \overline{\supp_{\pi'_\oplus, \pi'_\boxplus}(id)\circ \supp_{\pi'_\boxplus, \pi_\boxplus}(T) \circ \supp_{\pi_\boxplus,\pi_\oplus}(id)}\\
        & \subseteq \overline{Gr(\iota)\inv \circ \supp_{\pi'_\boxplus, \pi_\boxplus}(T)\circ Gr(\iota)}\\
        & = \overline{\cup_{i\in I} \calG (V_i\times_{\calG\units}V_i)}\\
        & = \cup_{i\in I} \overline{\calG (V_i\times_{\calG\units}V_i)}\\
        & \subseteq \cup_{i\in I} \calG (\overline{ V_i\times_{\calG\units}V_i})\\
        & \subseteq \cup_{i\in I} \calG (\overline{ V_i}\times_{\calG\units}\overline{V_i}),
    \end{align*}
    where the fifth line is because of corollary \ref{closure of orbit of relative compact set}. Since $V_i$ is relatively compact for each $i$, it is easy to check that $T$ is properly supported with respect to $\pi_\oplus$ and $\pi'_\oplus$.
\end{proof}

\begin{proof}[Proof of proposition \ref{key proposition of universal module}]
    Let $T=\oplus_{i\in I} \infl_{\calH_i}^\calG(T_i)$ be defined as in lemma \ref{direct sum of inflation of isometry}, $U: E'\ra \oplus_{i\in I}{E'}$ be a $\calG$-equivariant unitary isomorphism that intertwines $\oplus_{i\in I}{\pi'}$ and $\pi'$, and let $V$ be the following composite
    \[E\xrightarrow{\Phi}\oplus_{i\in I} \infl^\calG_{\calH_i}(E_{[V_i]})\xrightarrow{T} \oplus_{i\in I} \infl^\calG_{\calH_i}(E'_{[V_i]})\xrightarrow{\oplus_{i\in I} \Psi_i} \oplus_{i\in I}{E'}\xrightarrow{U\inv} E',\]
    by lemma \ref{construction of isometry Psi_i}, lemma \ref{Phi is adjointable isometry} and lemma \ref{direct sum of inflation of isometry}, $V$ is a $\calG$-equivariant isometry. By lemma \ref{Psi, Phi are intertwiners}, $\Phi$ intertwines $\pi$ and $\pi_\oplus$, $\Psi$ intertwines $\pi'_\oplus$ and $\oplus_{i\in I}{\pi'}$, and $U\inv$ intertwines $\oplus_{i\in I}{\pi'}$ and $\pi$, so by proposition \ref{support of intertwiner}, proposition \ref{property of support of operator} and lemma \ref{direct sum of inflation of isometry},
    \begin{align*}
        \supp_{\pi',\pi}(V) & = \supp_{\pi',\pi}(U\inv \circ (\oplus_{i\in I} \Psi_i)\circ T\circ \Phi)\\
        & \subseteq \Delta_Z\circ \Delta_Z\circ \supp_{\pi'_{\oplus},\pi_\oplus}(T)\circ \Delta_Z\\
        & \subseteq \cup_{i\in I} \calG (\overline{V_i}\times_{\calG\units}\overline{V_i})\\
        & \subseteq U.
    \end{align*}
\end{proof}

With this proposition proved, we reduce the problem of existence of universal module to the local cases. Recall that, use the same notation as above, for each $i$, there exists a finite group $F_i$ acting on $\calH_i\units$ such that $\calH_i\cong F_i\ltimes \calH_i\units$. So an $F_i$-Hilbert $A|_{\calH_i}$-module is canonically an $\calH_i$-Hilbert $A|_{\calH_i}$-module and vice versa. By proposition \ref{key proposition of universal module}, for a stable $Z,\calG,A$-module $(E,\pi)$, if $E_{[V_i]}$ can equivariantly absorb every countable generated $F_i$-Hilbert $A|_{\calH_i}$-module for each $i$, then $(E,\pi)$ will be universal. Hence, it is time for the Kasparov stabilization theorem to work. We will refer to \cite{mingo1984equivariant}. For the case $A=C_0(\calG\units)$ we can refer to \cite{Paterson2009TheST}.

\begin{prop}\label{universal module of etale spaces}
    Let $\calG$ be a second countable étale groupoid, $(Y,\rho)$ be a second countable locally compact Hausdorff proper \'etale $\calG$-space, $A$ be a separable $\calG$-\cst-algebra. We define $\pi:C_0(Y)\ra \calL(L^2(Y,\rho,L^2(\calG,A^\infty)))$ as
    \[[\pi(f)\eta](y)=f(y)\eta(y),\quad \forall y\in Y, \eta\in \Gamma_c(Y,\rho^*\ca E'),\]
    where $\ca E'=\sqcup_{x\in \calG\units}\ell^2(\calG_x,A_x^\infty)$ is the associated Hilbert bundle over $\calG\units$ of $L^2(\calG,A^\infty)$. Then 
    \[(L^2(Y,\rho,L^2(\calG,A^\infty)),\pi)\] is a universal $Y,\calG,A$-module.
\end{prop}
\begin{proof}
    By corollary \ref{exist G-cover} and lemma \ref{refine G-cover}, we have a $\calG$-cover $(V_i,\calH_i)_{i\in I}$ of $Y$, such that $I$ is countable, $(\calG V_i)_{i\in I}$ is a locally finite open cover of $Y$, and for every $i\in I$, $\rho|_{V_i}$ is a homeomorphism onto some open of $\calG\units$. After replacing $\calH_i$ by $(\calH_i)^{\rho(V_i)}_{\rho(V_i)}$, we can assume that $\rho(V_i)=\calH_i\units$.

    Now for every $i\in I$, we can check easily that
    \[L^2(Y,\rho,L^2(\calG,A^\infty))_{[V_i]}\cong L^2(\calG_{\calH_i\units},s, A|_{\calH_i\units}^\infty),\]
    as $\calH_i$-Hilbert $A|_{\calH_i}$-module. And since $\calH_i$ is relatively clopen in $\calG$, $L^2(\calH_i,A|_{\calH_i}^\infty)$ is a direct summand of $L^2(\calG_{{\calH_i}\units},s, A|_{{\calH_i}\units}^\infty)$. By Kasparov's stabilization theorem and proposition \ref{key proposition of universal module}, $(L^2(Y,\rho,A)^\infty,\pi^\infty)$ is a universal $Y,\calG,A$-module.
\end{proof}

\begin{lem}\label{direct summand of inflation}
    Let $\calG$ be an étale groupoid, $\calH$ be a relatively clopen subgroupoid of $\calG$, $\Omega=\calG_{\calH\units}$, $Y$ be a locally compact Hausdorff $\calG$-space. We denote the open subset
    \[\{[h,y]\in \Omega\times_\calH Y: h\in \calH, y\in Y_{s(h)}\}\]
    of $\Omega\times_\calH Y$ by $[\calH,Y]$. Let $A$ be an $\calH$-\cst-algebra, $(E,\pi)$ be a $Y,\calH,A$-module, then 
    \[\infl^\calG_\calH(E)_{[[\calH,Y]]}=\overline{\infl^\calG_\calH(\pi)(C_0([\calH,Y]))(\infl^\calG_\calH(E))}\]
    is isomorphic to $E$ as Hilbert $A$-module.
\end{lem}
\begin{proof}
    Let $F_c$ be $\{\xi\in \ind_{\Omega,c}E: \xi|_{\Omega\setminus \calH}=0\}$ and let $F$ be the closure of $F_c$ in $\infl^\calG_\calH(E)$. We will prove that, firstly $\infl^\calG_\calH(E)_{[[\calH,Y]]}=F$, and secondly $F$ is isomorphic to $E$.

    For any $f\in C_0([\calH,Y])$ and $\xi\in \ind_{\Omega,c}E\subseteq \infl^\calG_\calH(E)$, if $\gamma\in \Omega$ but $\gamma\not\in \calH$, we have $f([\gamma,-])=0\in C_0(Y_{s(\gamma)})$, and therefore
    \[[\infl^\calG_\calH(\pi)(f)\xi](\gamma)=\pi_{s(\gamma)}(f([\gamma,-]))\xi(\gamma)=0.\]
    That is,
    \[span\{\infl^\calG_\calH(\pi)(C_0([\calH,Y]))(\ind_{\Omega,c}(E))\}\subseteq F_c,\]
    and hence $\infl^\calG_\calH(E)_{[[\calH,Y]]}\subseteq F$.

    If $\xi\in F_c$, that is $\xi$ is an element of $\ind_{\Omega,c}E$ that vanishes outside $\calH$ (notice that $\calH$ is a clopen subset of $\Omega$). Let $(h_\lambda)_\lambda$ be an approximate unit of $C_0([\calH,Y])$. Then for any $\gamma\in \Omega$,
    \[[\infl^\calG_\calH(\pi)(h_\lambda)\xi](\gamma)=\pi_{s(\gamma)}(h_\lambda([\gamma,-]))\xi(\gamma).\]
    If $\gamma\not\in \calH$, then $h_\lambda([\gamma,-])=0$ and therefore $[\infl^\calG_\calH(\pi)(h_\lambda)\xi](\gamma)=0$. If $\gamma\in \calH$, $(h_\lambda([\gamma,-])_\lambda$ will be an approximate unit of $C_0(Y_{s(\gamma)})$, hence $ [\infl^\calG_\calH(\pi)(h_\lambda)\xi](\gamma)\ra \xi(\gamma)$. In conclusion $\xi=\lim_\lambda \infl^\calG_\calH(\pi)(h_\lambda)\xi$. So we proved that $F\subseteq \infl^\calG_\calH(E)_{[[\calH,Y]]}$.

    Let $W\in \calL(s_\calH^*E, r_\calH^*E)$ be the action of $\calH$ on $E$, we define
    \[T:E\ra F,\quad (Te)(\gamma)=\begin{cases}
        W_\gamma^*e(r(\gamma)) & \gamma\in \calH;\\
        0_{s(\gamma)} & \gamma \in \Omega\setminus \calH
    \end{cases}\]
    for all $e\in \Gamma_c(\calH\units,\ca E)\subseteq E$, and define
    \[S:F\ra E, \quad \xi\mapsto [x\mapsto \xi(x)]\in \Gamma_0(\calH\units, \ca E)=E,\]
    it is easy to check that $T$ and $S$ are well-defined and inverse to each other. Hence, $F$ is isomorphic to $E$.
\end{proof}

We will show that, if for every $i\in I$, $C_0(V_i)$ can act correctly on some specific Hilbert modules, then we can achieve a construction of a universal module.

\begin{lem}\label{local universal module exists}
    Let $F$ be a finite group, $X$ be a locally compact $F$-space, $V$ be a locally compact Hausdorff $\calH$-space, $A$ be a separable $\calH$-\cst-algebra. Let $C:=\cl_X(\rho(V))$, which is an $F$-invariant closed subset of $X$. Following definition \ref{defn A_C}, let $A_{\langle C\rangle}=\{a\in A:C_0(X\setminus C)\cdot a=0\}$, which is an $F\ltimes X$-invariant closed two-sided ideal of $A$. We define the $F\ltimes X$-Hilbert $A_{\langle C\rangle}$-module $E$ (which is hence a Hilbert $A$-module) as
    \[E=\ell^2(F,A_{\langle C\rangle})^\infty\cong \ell^2(F)\otimes A_{\langle C\rangle}\otimes \ell^2(\mathbb N),\]
    then for any $V,F\ltimes X,A$-module $(E_0,\pi_0)$, there exists an $F\ltimes X$-invariant isometry $V\in \calL(E_0,E)$. 
\end{lem}
\begin{proof}
    By proposition \ref{C_0(X)-linearity force base to shrink}, for any $V,\calH,A$-module $(E_0,\pi_0)$, $\langle E_0,E_0\rangle\subseteq A_{\langle C\rangle}$, hence $E_0$ is a countably generated $\calH$-Hilbert $A_{\langle C\rangle}$-module. By the group-equivariant Kasparov stabilization theorem (see \cite[Theorem 2.5]{mingo1984equivariant}), there exists an $F$-equivariant isometry $V\in \calL(E_0,E)$. The $F$-equivariance and $C_0(X)$-linearity of $V$ implies that $V$ is $F\ltimes X$-equivariant.
\end{proof}

Recall that in above notation, $L^2(F\ltimes X,A_{\langle C\rangle})\cong \ell^2(F,A_{\langle C\rangle})$. Since $A_{\langle C\rangle}$ is a closed two-sided ideal of $A$, a Hilbert $A_{\langle C\rangle}$-module is naturally seen as a Hilbert $A$-module.

\begin{prop}\label{universal modules exist}
    Let $\calG$ be a second countable étale groupoid, $Z$ be a second countable locally compact Hausdorff proper $\calG$-compact $\calG$-space with anchor map $\rho:Z\ra \calG\units$, $(A,\alpha)$ be a separable $\calG$-\cst-algebra, let $U$ be a $\calG$-invariant open neighborhood of the diagonal $\Delta_Z$ in $Z\times_{\calG\units}Z$, $(V_i,\calH_i)_{i=1}^n$ be a $\calG$-cover of $Z$ such that $\cup_{i=1}^n \calG(\overline{V_i}\times_{\calG\units}\overline{V_i})\subseteq U$. If for every $1\leqslant i\leqslant n$, there exists a non-degenerate stable $\calH_i$-equivariant representation
    \[\pi_i: C_0(V_i)\ra \calL(L^2(\calH_i,A_i)^\infty),\]
    where $A_i=(A|_{\calH_i})_{\langle C_i\rangle}$, $C_i=\cl_{\calH\units}(\rho(V_i))$ is $\calH_i$-invariant closed subset of $\calH_i\units$, then there exists a universal $Z,\calG,A$-module $(E,\pi)$, such that for any $Z,\calG,A$-module $(E',\pi')$, there exists a $\calG$-equivariant controlled isometry $V\in \Hom_A(E',E)$ such that $\supp_{\pi',\pi}(V)\subseteq U$.
\end{prop}
\begin{proof}
    Use similar notations to the last section, let $\iota: \sqcup_{i=1}^n \calG V_i\ra Z$ be defined by disjoint union of the open inclusions $\tilde\iota_i:\calG V_i\ra Z$. We define $E=\oplus_{i=1}^n \infl^\calG_{\calH_i}(L^2(\calH_i,A_i)^\infty)$, and $\pi:C_0(Z)\ra \calL(E)$ as $\pi=\iota_*(\oplus_{i=1}^n \infl^\calG_{\calH_i}(\pi_i))$. Then $(E,\pi)$ is a stable $Z,\calG,A$-module. Let $(E',\pi')$ be any other $Z,\calG,A$-module.

    Let $\Omega_i$ be $\calG_{\calH_i\units}$. Notice that, we used implicitly the $\calG$-equivariant homeomorphism
    \[\Omega_i\times_{\calH_i}V_i\cong \calG V_i,\quad [\gamma,v]\mapsto \gamma v.\]
    Under this map, we identify the open subset $[\calH_i,V_i]$ of the left-hand side with the open subset $V_i$ of the right-hand side.

    By lemma \ref{direct summand of inflation}, for each $i$, $E_{[V_i]}$ has an orthogonal direct summand which is isomorphic to $L^2(\calH_i,A_i)^\infty$. Apply lemma \ref{local universal module exists} to the $V_i,\calH_i,A|_{\calH_i}$-module $(E'_{[V_i]},\pi'_{[V_i]})$, there exists an $\calH_i$-equivariant isometry $T_i\in \calL(E'_{[V_i]},L^2(\calH_i,A_i)^\infty)$. After composing with the inclusion \[L^2(\calH_i,A_i)^\infty\ra E_{[V_i]},\] there exists an isometry $T_i'\in \calL(E'_{[V_i]},E_{[V_i]})$ which is $\calH_i$-equivariant.

    Then applying proposition \ref{key proposition of universal module}, there exists a $\calG$-equivariant isometry $V\in \Hom_A(E',E)$ such that $\supp_{\pi',\pi}(V)\subseteq U$.
\end{proof}

\subsection{Local structure of groupoid simplicial complexes}

Our aim is to prove the existence of universal modules for a family of proper $\calG$-compact $\calG$-spaces called Rips complexes. Therefore, by the result of the last subsection, we should be clear about the local structure of groupoid simplicial complexes.

We recall the definition of groupoid simplicial complexes from \cite{Mao26}.

\begin{defn}\label{defn groupoid simplicial complex}
    Let $X$ be a locally compact Hausdorff space. An $X$-simplicial complex of dimension less than $n$ is a pair $(Y,\Delta)$, where $Y$ is a locally compact space with a structure map $\rho: Y\ra X$, $\Delta$ is a family of non-empty finite subsets of $Y$ of cardinality at most $n+1$ such that
    \begin{enumerate}
        \item for any $\delta\in \Delta$, there exists $x\in X$ such that $\delta\subseteq \rho\inv(x)$;
        \item $\rho: Y\ra X$ is local homeomorphism;
        \item $\delta\in \Delta$, $\delta'$ is a non-empty subset of $\delta$, then $\delta'\in \Delta$.
    \end{enumerate}

    For any $0\leqslant m\leqslant n$, we define the $m$-skeleton $\Delta^m=\{\delta\in \Delta: \#\delta\leqslant m+1\}$.

    Let $\calG$ be an étale groupoid. A (left) $\calG$-simplicial complex of dimension less than $n$ is an $\calG\units$-simplicial complex $(Y,\Delta)$ of dimension less than $n$, such that $Y$ is a left $\calG$-space, and if $\delta=\{y_1,\cdots, y_m\}\in \Delta$ is contained in some fiber $\rho\inv(x)$, then for any $\gamma\in \calG_x$, $\gamma\delta:=\{\gamma y_1, \cdots, \gamma y_m\}\in\Delta$.
\end{defn}

\begin{defn}
    Let $\rho:Y\ra X$ be a local homeomorphism between \lch spaces. We denote the linear space of Radon measures on $Y$ by $R(Y)$, and let
\[P(Y):=\{\mu\in R(Y): \exists x\in X, \supp(\mu)\subseteq \rho\inv(x), \mu(Y)=1, \mu \text{ is positive}\},\]
be equipped with weak-*-topology defined by $C_c(Y,\mathbb R)$, that is a net $(\mu_\lambda)_\lambda$ in $P(Y)$ converges to $\mu$ if and only if for any $f\in C_c(Y,\mathbb R)$, $(\mu_\lambda(f))_\lambda$ converges to $\mu(f)$.

For any $\mu\in P(Y)$, we denote the point $x\in X$ such that $\supp(\mu)\subseteq \rho\inv(x)$ by $\tilde\rho(\mu)$.

For an $X$-simplicial complex $(Y,\Delta)$, the geometric realization $|\Delta|$ of $(Y,\Delta)$ is the subspace $\{\mu\in P(Y): \supp(\mu)\in \Delta\}$ of $P(Y)$.

For any $\mu_0\in |\Delta|$, $f\in C_c(Y,\mathbb R)$ and $\epsilon>0$, we define the set
\[W(\mu,f,\epsilon):=\{\mu\in |\Delta|:|\mu(f)-\mu_0(f)|<\epsilon\}.\]
\end{defn}

We can prove that $\tilde\rho$ is continuous (see \cite[Proposition 3.3]{Mao26}). Clearly, if $(Y,\Delta)$ is a $\calG$-simmplicial complex, then $|\Delta|$ is a $\calG$-space (see \cite[Proposition 3.4]{Mao26}): for any $\gamma\in \calG$ and $\mu=c_0\delta_{y_0}+\cdots+c_m\delta_{y_m} \in |\Delta|$ such that $\rho(y_0)=\cdots=\rho(y_m)$, we have
\[\gamma.\mu=c_0\delta_{\gamma y_0}+\cdots + c_m\delta_{\gamma y_m}.\]
The sets like $W(\mu_0,f,\epsilon)$ are open subsets of $|\Delta|$, and indeed form a topological basis of $|\Delta|$. When $Y$ is second countable, the unit ball of $R(Y)$ is compact and metrizable by Banach-Alaoglu's theorem and \cite[13.4.2]{Dieudonne1970treatisevol2}, hence second countable. As a subspace, $|\Delta|$ is also second countable.

\begin{defn}
    Let $\calG$ be an \'etale groupoid and $(Y,\Delta)$ be a $\calG$-simplicial complex. We say that $(Y,\Delta)$ is proper, if the action of $\calG$ on $Y$ is proper. We say that $(Y,\Delta)$ is $\calG$-compact, if $|\Delta^0|\cong \{y\in Y:\{y\}\in \Delta\}$ is $\calG$-compact.
\end{defn}

\begin{defn}\label{H_1 H_2}
Let $X$ be a locally compact Hausdorff space, $(Y,\Delta)$ be an $X$-simplicial complex.
    \begin{enumerate}
        \item We say that $(Y,\Delta)$ has hypothesis $(H_1)$, if for any compact subset $K$ of $Y$,
        \[C_K:=\{y\in Y: \exists y'\in K,\{y,y'\}\in \Delta\}\]
        is compact subset of $Y$.
        \item We say that $(Y,\Delta)$ has hypothesis $(H_2)$, if $(y^{(0)}_\lambda)_\lambda, \cdots, (y^{(m)}_\lambda)_\lambda$ are convergent nets, such that $\lim_\lambda y^{(i)}_\lambda=y^{(i)}\in Y$, and for any $\lambda$, $\{y^{(0)}_\lambda, \cdots, y^{(m)}_\lambda\}\in \Delta$, then $\{y^{(0)},\cdots, y^{(m)}\}\in \Delta$.
    \end{enumerate}
\end{defn}

It turns out that the two hypotheses are important to ensure that $|\Delta|$ has good topological properties.

\begin{prop}
    \textnormal{\cite[Proposition 3.19, Proposition 3.22, Proposition 3.24]{Mao26}}
    \begin{enumerate}
        \item Let $X$ be a \lch space and $(Y,\Delta)$ be a $X$-simplicial complex with hypotheses $(H_1)$ and $(H_2)$, then $|\Delta|$ is a locally compact Hausdorff space.
        \item Let $\calG$ be an \'etale groupoid and $(Y,\Delta)$ be a $\calG$-simplicial complex with hypotheses $(H_1)$ and $(H_2)$. If $Y$ is a proper $\calG$-sapce, then so is $|\Delta|$. If $(Y,\Delta)$ is $\calG$-compact, then $|\Delta|$ is a $\calG$-compact $\calG$-space.
    \end{enumerate}
\end{prop}

As we expected, the Rips complexes have all these properties above.

\begin{prop}\label{Rips complex}
    \textnormal{\cite[Proposition 3.26]{Mao26}}
    Let $\calG$ be an étale groupoid and $K$ be a compact subset of $\calG$. See $\calG$ as a left $\calG$-space with action defined by multiplication. Define
\[\Delta_K(\calG):=\{\delta\subseteq \calG: \forall \gamma_1,\gamma_2\in \delta, r(\gamma_1)=r(\gamma_2), \gamma_1\inv \gamma_2\in K\}.\]
\begin{enumerate}
    \item $(\calG,\Delta_K(\calG))$ is a finite dimensional $\calG$-simplicial complex (called Rips complex).
    \item The $\calG$-simplicial complex $(\calG,\Delta_K(\calG))$ has hypotheses $(H_1)$ and $(H_2)$.
    \item The $\calG$-simplicial complex $(\calG,\Delta_K(\calG))$ is proper and $\calG$-compact.
\end{enumerate}
\end{prop}

\begin{defn}\label{defn typed}
   A $\calG$-simplicial complex $(Y,\Delta)$ is typed, if there exists a finite discrete space $T$ and a continuous map $\tau: Y\ra T$, such that
    \begin{enumerate}
        \item the map $\tau$ is $\calG$-invariant, that is for any $(\gamma,y)\in \calG\times_{s,\calG\units,\rho}Y$, $\tau(\gamma y)=\tau(y)$;
        \item for any $\delta\in \Delta$, $\tau|_{\delta}$ is injective.
    \end{enumerate}
\end{defn}

A typed $\calG$-simplicial complexes has a stratification where each stratum is equivariantly homeomorphic to the product of some closed subset of an \'etale $\calG$-space and a Euclidean space.

\begin{prop}\label{typed sim com}
    \textnormal{\cite[Proposition 3.32]{Mao26}}
    Let $(Y,\Delta)$ be a typed $\calG$-simplicial complex of dimension less than $n$ with the type function $\tau:Y\ra \{1,\cdots, N\}$. Then for any $1\leqslant m\leqslant n$, the map
    \[|\Delta^m|\setminus |\Delta^{m-1}|\cong center(m,\Delta)\times \sigma_m,\]
    \[c_0\delta_{y_0}+\cdots+c_{m}\delta_{y_m}\mapsto (\frac{\delta_{y_0}+\cdots+\delta_{y_m}}{m+1}, (c_0,\cdots,c_m))\]
    is a $\calG$-equivariant homeomorphism, where $\tau(y_0)<\tau(y_1)<\cdots<\tau(y_m)$, $\sigma_m:=\{(t_0,\cdots, t_m)\in \mathbb R^{m+1}:\forall i, t_i>0;\sum_{i=0}^m t_i=1\}$ is the interior of a standard $m$-simplex, $\calG$ acts on $center(m,\Delta)\times \sigma_m$ by \[\gamma.(\mu,(t_0,\cdots, t_m))=(\gamma\mu,(t_0,\cdots, t_m))\] when $s(\gamma)=\tilde\rho(\mu)$.
\end{prop}

\begin{defn}[Barycentric subdivision]
    Let $\calG$ be an étale groupoid and $(Y,\Delta)$ be a $\calG$-simplicial complex of dimension less than $n$. We define 
    \[A^i:=\{\frac{1}{i+1}(\delta_{y_0}+\cdots+\delta_{y_i})\in P(Y):(y_0,\cdots, y_i)\in Y^{\times_{\calG\units} (i+1)}\}, \]
    where $Y^{\times_{\calG\units}(i+1)}$ is the fiber product of $i+1$ copies of $Y$ over $\calG\units$, $Y':=\sqcup_{i=1}^n A^i\subseteq P(Y)$, and define $\rho':Y'\ra X$ to be the restriction of $\tilde\rho$ on $Y'$. We define $\Delta'$ to be the family of finite subsets of $Y'\cap|\Delta|$ consisting of elements in form of $\{\mu_1,\mu_2,\cdots,\mu_k\}$ such that $\supp(\mu_1)\subset \supp(\mu_2)\subset\cdots\subset \supp(\mu_k)$. We call $(Y',\Delta')$ the barycentric subdivision of $(Y,\Delta)$.
\end{defn}

\begin{prop}\label{barycentric subdivision keeps realization}
    \textnormal{\cite[Proposition 3.33, Proposition 3.35]{Mao26}}
    Use the same notation as above, $Y'$ is a well-defined étale $\calG$-space, $(Y',\Delta')$ is a well-defined $\calG$-simplicial complex and $(Y',\Delta')$ is typed. Moreover, the geometric realization $|\Delta'|$ is $\calG$-equivariantly homeomorphic to $|\Delta|$.
\end{prop}

We need a family of opens that are in ``good shapes". Therefore, we will construct a topological basis as described below.

\begin{defn}
Let $\rho:Y\ra X$ be a local homeomorphism between locally compact Hausdorff spaces, $(Y,\Delta)$ be an $X$-simplicial complex satisfying hypotheses $(H_1)$ and $(H_2)$. We call a tuple $((V_y)_{y\in \supp(\mu)},V,U,\epsilon)$ a compatible datum of $\mu$, if
    \begin{enumerate}
        \item for every $y\in \supp(\mu)$, $V_y$ is an open neighborhood of $y$ in $Y$ such $\rho|_{V_y}$ is a homeomorphism onto an open of $X$;
        \item the set $V$ is an open subset of $X$, such that for every $y\in \supp(\mu)$, $\rho(V_y)=V$;
        \item the set $U$ is a relatively compact open neighborhood of $\tilde\rho(\mu)$ in $V$;
        \item the real number $\epsilon$ is positive and for every $y\in \supp(\mu)$, $\epsilon<\min \{\mu(\{y\}), 1-\mu(\{y\}), \frac{1}{3\#\supp(\mu)}\}$.
    \end{enumerate}
    Let $CD(\mu)$ be the set of all compatible data of $\mu$. For any $((V_y)_{y},V,U,\epsilon)\in CD(\mu)$ and for any $y\in \supp(\mu)$, we write $U_y=\rho|_{V_y}\inv(U)$.

    For a compatible datum $((V_y)_{y},V,U,\epsilon)$, we define the associated subset of $|\Delta|$ as
    \[W((V_y)_y,V,U,\epsilon)=\{\mu'\in |\Delta|:\tilde\rho(\mu')\in U, |\mu'(V_y)-\mu(\{y\})|<\epsilon,\forall y\in \supp(\mu)\}.\]
\end{defn}

\begin{lem}\label{first glimpse of W}
Use the same notation as above, the subset $W((V_y)_y,V,U,\epsilon)$ is a relatively compact open of $|\Delta|$. And $\mu'(\cup_{y\in \supp(\mu)}V_y)>\frac{2}{3}$ for any $\mu'\in W((V_y)_y,V,U,\epsilon)$.
\end{lem}
\begin{proof}
    By Urysohn's lemma, we can choose $f\in C_c(X)$ such that $f|_U=1$ and $\supp(f)\subseteq V$. For each $y\in \supp(\mu)$, we define $f_y\in C_c(Y)$ as the function
    \[f_y(y')=\begin{cases}
        f(\rho(y')) & y'\in V_y\\
        0 & \text{else} 
    \end{cases}.\]
    Then $W((V_y)_y,V,U,\epsilon)=|\Delta|\cap \tilde\rho\inv(U)\cap(\cap_{y\in \supp(\mu)}W(\mu,f_y,\epsilon))$, which is open. \cite[Proposition 3.19]{Mao26} implies that it is relatively compact. And by the inequality
    \[\mu'(\cup_{y\in \supp(\mu)}V_y)=\sum_{y\in \supp(\mu)}\mu'(V_y)>\sum_{y\in \supp(\mu)}(\mu(\{y\})-\epsilon)=1-\#\supp(\mu)\cdot \epsilon>\frac{2}{3},\]
    every element of $W((V_y)_y,V,U,\epsilon)$ has over 2/3 of mass on $\cup_{y\in \supp(\mu)}V_y$.
\end{proof}

To have an instinct view, if $\mu=c_0\delta_{y_0}+\cdots+c_m\delta_{y_m}$, then $W((V_{y})_{y},V,U,\epsilon)$ should be constituted by elements of $|\Delta|$ in form of
\[t_0\delta_{\rho|_{V_{y_0}}\inv(x)}+\cdots+ t_m\delta_{\rho|_{V_{y_m}}\inv(x)}+\mu'',\]
where $x\in U$, $t_i\in (c_i-\epsilon,c_i+\epsilon)$ for $i=0,\cdots, m$, and $\mu''$ is an element of $R(Y_x)$ which is not supported on $\{\rho|_{V_{y_i}}\inv(x):i=0,\cdots, m\}$.

\begin{defn}\label{defn local basis S'}
We define the following families of open neighborhoods of $\mu$ in $|\Delta|$
    \[\ca S'(\mu)=\{W((V_y)_{y\in \supp(\mu)},V,U,f,\epsilon): ((V_y)_{y\in \supp(\mu)},V,U,f,\epsilon)\in CD(\mu)\}.\]
We define $\ca S'=\cup_{\mu\in |\Delta|}\ca S'(\mu)$.
\end{defn}

\begin{prop}\label{S,S'' are subbases}
    Use the same notation as above, for $\mu_0\in |\Delta|$, $\ca S'(\mu_0)$ is a fundamental system of neighborhoods of $\mu_0$ in $|\Delta|$. Moreover, $\ca S'$ is a topological basis. 
\end{prop}
\begin{proof}
    Firstly, we will show that, for any $\mu_0\in |\Delta|$, $S'(\mu_0)$ is stable under finite intersection. If 
    $$((V_y)_{y\in \supp(\mu_0)},V,U,\epsilon),\quad ((V_y')_{y\in \supp(\mu_0)},V',U',\epsilon')$$
    are two compatible data of $\mu$, then
    \[((V_y\cap V_y')_{y\in \supp(\mu_0)},V\cap V',U\cap U',\min(\epsilon,\epsilon'))\]
    is also a compatible datum, and
    \begin{align*}
        & W((V_y)_{y\in \supp(\mu_0)},V,U,\epsilon)\cap W((V_y')_{y\in \supp(\mu_0)},V',U',\epsilon')\\
        & = \{\mu\in |\Delta|:\tilde\rho(\mu)\in U\cap U', |\mu(V_y)-\mu_0(\{y\})|<\epsilon, |\mu(V_y')-\mu_0(\{y\})|<\epsilon', \forall y\in \supp(\mu_0)\}\\
        & = \{\mu\in |\Delta|:\tilde\rho(\mu)\in U\cap U', |\mu(V_y\cap V_y')-\mu_0(\{y\})|<\min(\epsilon,\epsilon'), \forall y\in \supp(\mu_0)\}\\
        & = W((V_y\cap V_y')_{y\in \supp(\mu_0)},V\cap V',U\cap U',\min(\epsilon,\epsilon')).
    \end{align*}

    By \cite[Proposition 3.18]{Mao26}, the opens of type $W(\mu_0,\phi,\epsilon)$ form a subbase, where $\mu_0\in |\Delta|$, $\phi\in C_c(Y,[0,1])$ and $\epsilon>0$, such that
    \begin{enumerate}
        \item there exists $y_0\in \supp(\mu_0)\cap \{y\in Y:\phi(y)\neq 0\}$ and a relatively compact open neighborhood $V_0$ of $y_0$, such that $\rho|_{V_0}:V_0\ra \rho(V_0)$ is a homeomorphism onto an open of $\calG\units$, and $\phi$ is supported in $V_0$;
        \item $0<\epsilon<\mu_0(\phi)$.
    \end{enumerate}
    
    So in order to show that $\ca S'$ is a topological basis, it suffices to prove that, for any $\mu_0\in |\Delta|$ and an open neighborhood of $\mu_0$ in form of $W_0:=W(\mu_0,\phi,\epsilon)$ that satisfies the condition above, there exists $W_1\in S'(\mu_0)$, such that $W_1\subseteq W_0$.

    For $y\in \supp(\mu_0)\setminus\{y_0\}$, we choose $V_y'$ as an open neighborhood of $y$ such that $\rho|_{V'_y}$ is a homeomorphism onto an open of $X$. Then we define $V'=(\cap_{y\in \supp(\mu_0)\setminus\{y_0\}}\rho(V'_y))\cap \rho(V_0)$. Let $$V_{y_0}=\rho|_{V_0}\inv(V')\cap \{y\in V_0:\phi(y_0)-\frac{\epsilon}{2}<\phi(y)<\phi(y_0)+\frac{\epsilon}{2}\},$$ $V''=\rho(V_{y_0})$, and for any $y\in \supp(\mu)\setminus \{y_0\}$, define $V_y=\rho|_{V'_y}\inv(V'')$.
    Let $U$ be a relatively compact open neighborhood of $\tilde\rho(\mu_0)$ in $V''$. By Urysohn's lemma, we choose $\psi\in C_c(X,[0,1])$ such that $\supp(\psi)\subseteq V''$ and $\psi|_{\overline{U}}=1$. Finally, we choose $\epsilon'\in \mathbb R_+$ such that
    \[\epsilon'<\min\{\frac{\epsilon}{2},\min_{y\in \supp(\mu_0)}\mu_0(\{y\}),\min_{y\in \supp(\mu_0)}(1-\mu_0(\{y\})\}.\]

    We define $f\in C_c(X)$ as 
    \[f(x)=\begin{cases}
    (1-\psi(x))\frac{\phi(\rho|_{V_0}\inv(x))}{\phi(y_0)}+\psi(x) & x\in \rho(V_0)\\
    0 & else
    \end{cases}.\]
    So $f$ is supported in $\rho(V_0)$ and equals to 1 on $U$. And for any $y\in V_0$,
    \[|\phi(y)-\phi(y_0)f(\rho(y))|=|\psi(\rho(y))(\phi(y)-\phi(y_0))|<\frac{\epsilon}{2}.\]

    Let $f_{y_0}\in C_c(V_0)$ be $f\circ \rho|_{V_0}$ and $g\in C_c(Y)$ be $\phi(y_0)f_{y_0}$. So we have $\|g-\phi\|_\infty<\frac{\epsilon}{2}$. Claim: $\tilde\rho\inv(U)\cap W(\mu_0,g,\epsilon')\subseteq W_0$. If $\mu\in \tilde\rho\inv(U)\cap W(\mu_0,g,\epsilon')$, then $\tilde\rho(\mu)\in U$, and $|\mu(g)-\mu_0(g)|=|\mu(g)-\mu_0(\phi)|<\epsilon'$, and hence
    \[|\mu(\phi)-\mu_0(\phi)|\leqslant |\mu(\phi-g)|+|\mu(g)-\mu_0(\phi)|<\frac{\epsilon}{2}+\epsilon'<\epsilon,\]
    which implies that $\mu\in W(\mu,\phi,\epsilon)$.

    Claim: $W((V_y)_y,V'',U,\epsilon')\subseteq \tilde\rho\inv(U)\cap W(\mu_0,g,\epsilon')$. If $\mu\in W((V_y)_y,V'',U,\epsilon')$, then $|\mu(V_{y_0})-\mu_0(\{y_0\})|<\epsilon'$, which implies that 
    \[|\mu(g)-\mu_0(g)|=\phi(y_0)|\mu(f_{y_0})-\mu_0(f_{y_0})|=\phi(y_0)|\mu(V_{y_0})-\mu_0(\{y_0\})|<\phi(y_0)\epsilon'\leqslant \epsilon',\]
    and hence $\mu\in \tilde\rho\inv(U)\cap W(\mu_0,g,\epsilon')$.
    
    In conclusion, $W((V_y)_y,V'',U,\epsilon')$ is an element of $\ca S'(\mu)$ that is contained in $W_0$. We finished the proof.
\end{proof}

\begin{lem}
    Let $W\in \ca S'(\mu)$. Then for any $\mu'\in W$, $\#\supp(\mu)\leqslant \#\supp(\mu')$. Moreover, let $m=\#\supp(\mu)-1$, $W\cap (|\Delta^m|\setminus |\Delta^{m-1}|)$ is a closed subset of $W$.
\end{lem}
\begin{proof}
    Let $((V_y)_{y\in \supp(\mu)},V,U,\epsilon)$ be a compatible datum associated to $W$. By definition, all the opens $V_y$ are disjoint from each other. Since $\mu'\in W$, for every $y\in \supp(\mu)$, $|\mu'(V_y)-\mu(\{y\})|<\epsilon<\mu(\{y\})$, which implies that $\mu'(V_y)\neq 0$. Hence, $\supp(\mu')\cap V_y=\{\rho|_{V_y}\inv(\tilde\rho(\mu'))\}$ is a singleton. We have therefore $\#\supp(\mu')\geqslant \#\supp(\mu)$. Moreover, $W\cap (|\Delta^m|\setminus |\Delta^{m-1}|)=W\cap |\Delta^m|$ is closed in $W$.
\end{proof}

For geometrical realizations of typed groupoid simplicial complexes, we need to put more requirements on the choice of opens to form a topological basis.

\begin{defn}
    Let $(Y,\rho)$ be a proper $\calG$-space, $(Y,\Delta)$ be a typed proper $\calG$-simplicial complex satisfying hypotheses $(H_1)$ and $(H_2)$. Let $\tau:Y\ra \{1,\cdots, N\}$ be the type function. For $\mu_0\in |\Delta|$, we say that $((V_y)_{y\in \supp(\mu_0)},V,U,\epsilon)\in CD(\mu_0)$ is a locally symmetric compatible datum of $\mu_0$, if
    \begin{enumerate}
        \item for every $y\in \supp(\mu_0)$, $\tau|_{V_y}$ is constant;
        \item there exist actions of $F_{\mu_0}:=\{\gamma\in\calG_{\tilde\rho(\mu_0)}^{\tilde\rho(\mu_0)}:\gamma\mu_0=\mu_0\}$ on $V$ and a proper open subgroupoid $\calH$ of $\calG$, such that $\calH\units=V$ and  $\calH$ is isomorphic to $F_{\mu_0}\ltimes V$;
        \item $U$ is an $F_{\mu_0}$-invariant open of $V$;
        \item for every $y\in \supp(\mu_0)$, $V_y$ is an $\calH$-invariant relatively compact open of $Y$, and for every $\gamma\in \calG\setminus\calH$, we have $\gamma \overline{V_y}\cap \overline{V_y}=\emptyset$.
    \end{enumerate}
    We denote the set of all locally symmetric compatible data of $\mu_0$ by $CD_{ls}(\mu_0)$. We define
    \[\ca S'_{ls}(\mu_0)=\{W((V_y)_{y\in C_{\supp(\mu_0)}},V,U,\epsilon): ((V_y)_{y\in C_{\supp(\mu_0)}},V,U,\epsilon)\in CD_{ls}(\mu_0))\},\]
    \[\ca S'_{ls}=\cup_{\mu_0\in |\Delta|} S''_{ls}(\mu_0).\]
\end{defn}

\begin{prop}\label{locally symmetric subbase}
    Use the same notation as above. 
    \begin{enumerate}
        \item For $((V_y)_{y\in \supp(\mu_0)},V,U,\epsilon)\in CD_{ls}(\mu_0)$ and $W=W((V_y)_{y\in \supp(\mu_0)},V,U,\epsilon)\in \ca S'_{ls}(\mu_0)$, let $\calH\cong F_{\mu_0}\ltimes V$ be the associated proper open subgroupoid of $\calG$, then $W$ is an $\calH$-invariant open neighborhood of $\mu_0$ in $|\Delta|$, and for any $\gamma\not\in \calH$, $\gamma \overline{W}\cap \overline{W}=\emptyset$.
        \item The family $\ca S'_{ls}(\mu_0)$ is a fundamental system of neighborhoods of $\mu_0$ in $|\Delta|$. Moreover, $\ca S'_{ls}$ is a topological basis.
    \end{enumerate}
\end{prop}
\begin{proof}
    (1) Let $\mu\in W$, then for $h\in \calH_{\tilde\rho(\mu)}$, $\tilde\rho(h\mu)=r(h)\in U$ since $U$ is $\calH$-invariant, and for any $y\in \supp(\mu_0)$,
    \[|(h\mu)(V_y)-\mu_0(\{y\})|=|\mu(V_y)-\mu_0(\{y\})|<\epsilon.\]
    Hence, $h\mu\in W$, and therefore $W$ is $\calH$-invariant.

    Now for $\gamma\in \calG\setminus \calH$, for every $y\in \supp(\mu_0)$, $\gamma V_y\cap V_y=\emptyset$ by assumption of locally symmetric compactible datum, and $\gamma V_y\cap V_{y'}=\emptyset$ for any other $y'\in \supp(\mu_0)$ because they have different type, therefore $(\cup_{y\in \supp(\mu_0)}V_y)\cap (\cup_{y\in \supp(\mu_0)}\gamma V_y)=\emptyset$. Now by Urysohn's lemma, there exists $\phi\in C_c(Y,\mathbb R)$ such that $\phi|_{\cup_{y\in \supp(\mu_0)}\overline{U_y}}=1$ and $\supp(\phi)\subseteq \cup_{y\in \supp(\mu_0)}V_y$. Lemma \ref{first glimpse of W} shows that every element of $W$ has more than $\frac{2}{3}$ of mass on $\cup_{y\in \supp(\mu_0)}V_y$, hence $\mu(\phi)>\frac{2}{3}$ for all $\mu\in W$, and therefore $\mu(\phi)\geqslant \frac{2}{3}$ for all $\mu\in \overline{W}$. Now if $\mu\in\overline{W}$ such that $\mu\in \gamma\overline{W}\cap \overline{W}$, we will have $\|\phi|_{Y_{r(\gamma)}}+\phi|_{Y_{s(\gamma)}}(\gamma\inv-)\|_{\infty}\leqslant 1$ but $\mu(\phi|_{Y_{r(\gamma)}}+\phi|_{Y_{s(\gamma)}}(\gamma\inv-))\geqslant
    \frac{2}{3}+\frac{2}{3}>1$, which is a contradiction. So $\gamma\overline{W}\cap \overline{W}=\emptyset$.
    
    (2) Let $W_0$ be any open neighborhood of $\mu_0$. By proposition \ref{S,S'' are subbases}, there exists $((V_y)_{y},V,U,\epsilon)\in CD(\mu_0)$ and $W_1=W((V_y)_{y},V,U,\epsilon)\in \ca S'(\mu_0)$, such that $W_1\subseteq W_0$.

    Now, assume that $\mu_0=c_0\delta_{y_0}+\cdots+c_m\delta_{y_m}$, where $c_i\in (0,1]$ for $0\leqslant i\leqslant m$. Since $|\Delta|$ is typed, the finite group $F_{\mu_0}$ is exactly the intersection of the stabilizers of points in $\supp(\mu_0)$. Without loss of generality, assume $Y_i=\{y\in Y:\tau(y)=\tau(y_i)\}$ for $0\leqslant i\leqslant m$, which are $m$ different clopen of $Y$. Apply proposition \ref{local structure of proper cocompact action of étale groupoid} to $(y_0,\cdots, y_m)$ of the proper $\calG$-space $Y_0\times_{\calG\units}\cdots\times_{\calG\units}Y_m$, then there exists
    \begin{enumerate}
        \item a relatively compact open neighborhood $V_{y_i}'$ of $y_i$ in $V_{y_i}\cap Y_i$ for each $0\leqslant i\leqslant m$;
        \item a proper open subgroupoid $\calH$ of $\calG$;
        \item an action of $F_{\mu_0}$ on $\calH\units$,
    \end{enumerate}
    such that
    \begin{enumerate}
        \item there is an isomorphism $F_{\mu_0}\ltimes \calH\units\cong \calH$;
        \item for every $0\leqslant i\leqslant m$, $\rho(V_{y_i}')\subseteq \calH\units$, and $V_{y_i}'$ is an $\calH$-invariant open, such that $\gamma V_{y_i}'\cap V_{y_i}'=\emptyset$ for any $\gamma\in \calG\setminus \calH$.
    \end{enumerate}
    Let $V'=\cap_{i=0}^m \rho(V_{y_i}')$. After replacing $V_{y_i}'$ by $V_{y_i}'\cap \rho\inv(V')$, we can assume that $\rho(V_{y_i}')=V'$ for every $0\leqslant i\leqslant m$. Let $U'$ be an $F_{\mu_0}$-invariant open neighborhood of $\tilde\rho(\mu_0)$ in $U\cap V'$.

    Let $\calH'=\calH^{V'}_{V'}$. Then $((V_y')_{y},V',U',\epsilon)$ is a locally symmetric compatible datum. Claim: the associated open $W_2=W((V_y')_{y},V',U',\epsilon)$ is contained in $W_1$. If $\mu\in W_2$, then $\tilde\rho(\mu)\in U'$ and for every $0\leqslant i\leqslant m$, $|\mu(V_{y_i}')-c_i|<\epsilon$. Since $\tilde\rho(\mu)\in U'$, we have $\mu(V_{y_i}')=\mu(V_{y_i})$, and hence $|\mu(V_{y_i})-c_i|<\epsilon$, which implies that $\mu\in W_1=W((V_y)_{y},V,U,\epsilon)$. We proved that $W_2\subseteq W_1$ and finished the proof.
\end{proof}

\subsection{Construction of universal modules}

We have now reduced the question about existence of universal $Z,\calG,A$-modules to the question about existence of the representation in the statement of proposition \ref{universal modules exist}. However, their existences are not very clear, and the principal difficulty is the $C_0(\calG\units)$-linearity. 

Let $\rho:Y\ra X$ be a continuous map, $\pi:C_0(Y)\ra \calL_{C_0(X)}(E)$ be a $C_0(X)$-linear non-degenerate representation over a Hilbert $C_0(X)$-module. By proposition \ref{C_0(X)-linearity force base to shrink}, the $C_0(X)$-linearity forces $\langle E,E\rangle\subseteq \{f\in C_0(X):C_0(X\setminus \overline{\rho(Y)})\cdot f=0\}=C_0(\inter_X(\overline{\rho(Y)}))$, and therefore $E$ is a direct summand of $\ell^2(\mathbb N)\otimes C_0(\inter_X(\overline{\rho(Y)}))$. What is difficult is the following question.

\begin{quest}
Use the same notation as above, does there always exist a $C_0(X)$-linear non-degenerate representation
    \[C_0(Y)\ra \calL(\ell^2(\mathbb N)\otimes C_0(\inter_X(\overline{\rho(Y)}))) ?\]
\end{quest}

When $Y$ is a closed subset of $X$ and $\rho$ is the closed inclusion, the answer is yes, since in this case $\inter_X(Y)$ is an open of $Y$ and therefore $C_0(\inter_X(Y))$ is a $C_0(Y)$-algebra.

\begin{prop}\label{property of locally symmetric subbase}
    Let $\calG$ be an \'etale groupoid, $(Y,\Delta)$ be a typed proper $\calG$-simplicial complex. Let $\mu_0\in |\Delta|$, $m=\#\supp(\mu_0)-1$, and
    \[\phi: |\Delta^m|\setminus|\Delta^{m-1}|\ra center(m,\Delta)\times\sigma_m\]
    be the homeomorphism as described in proposition \ref{typed sim com}. Then for $((V_y)_{y\in \supp(\mu_0)},V,U,\epsilon)\in CD_{ls}(\mu_0)$ and $W=W((V_y)_{y\in \supp(\mu_0)},V,U,\epsilon)\in \ca S'_{ls}(\mu_0)$, let $\calH'\cong F_{\mu_0}\ltimes V$ be the associated proper open subgroupoid and $\calH=(\calH')_U^U$, then there exists a clopen subset $Q$ of $A^m\cap\tilde\rho\inv(U)$ and an open subset $R$ of $\sigma_m$, such that $\tilde\rho|_Q$ is homeomorphism onto $U$, $Q$ is $\calH$-invariant, and
    \[\phi(W\cap (|\Delta^m|\setminus|\Delta^{m-1}|))=(Q\cap center(m,\Delta))\times R\subseteq center(m,\Delta)\times \sigma_m.\]
\end{prop}
\begin{proof}
    Assume that $\mu_0=\sum_{i=0}^m c_i\delta_{y_i}$ such that $c_i\in (0,1]$ for every $0\leqslant i\leqslant m$. Without loss of generality, we assume that the type function is $\tau:Y\ra \{0,\cdots, N\}$ and $\tau(y_i)=i$ for $0\leqslant i\leqslant m$. Let $Y_i=\tau\inv(i)$, which is clopen in $Y$.
    So $A^m\cap \tilde\rho\inv(U)$ is
    \[\{\frac{\sum_{i=0}^m \delta_{y_i'}}{m+1}:(y'_0,\cdots, y'_m)\in \bigsqcup_{0\leqslant k_0\leqslant\cdots\leqslant k_m\leqslant N} Y_{k_0}|_U\times_U\cdots \times_U Y_{k_m}|_U\}.\]
    Let $U_i=\rho|_{V_{y_i}}\inv(U)\subseteq V_{y_i}$, which is $\calH$-invariant clopen in $Y_i|_{U}$. So we construct $Q$ as
    \[\{\frac{\sum_{i=0}^m \delta_{y_i'}}{m+1}:(y'_0,\cdots, y'_m)\in U_0\times_U U_1\times_U\cdots \times_U U_m\},\]
    and $Q$ is clopen in $A^m\cap \tilde\rho\inv(U)$. For $0\leqslant i\leqslant m$, $\rho|_{U_i}$ is a homeomorphism onto $U$, so $\tilde\rho|_Q$ is a homeomorphism onto $U$, and clearly $Q$ is $\calH$-invariant.

    Recall that $\sigma_m=\{(t_0,\cdots, t_m)\in \mathbb R_+^{m+1}:\sum_{i=0}^m t_i=1\}$. We construct an open neighborhood $R$ of $(c_0,\cdots, c_m)$ in $\sigma_m$ as
    \[\{(t_0,\cdots, t_m)\in \sigma_m: \forall 0\leqslant i\leqslant m, t_i\in (c_i-\epsilon, c_i+\epsilon)\}.\]
    By our condition on $\epsilon$, we have $0<c_i-\epsilon<c_i+\epsilon<1$ for all $0\leqslant i\leqslant m$.

    Claim: $\phi(W\cap (|\Delta^m|\setminus|\Delta^{m-1}|))=(Q\cap |\Delta|)\times R$. If $(y_0',\cdots, y_m')\in Q$, $(t_i)_{i=0}^m\in R$, $\phi\inv((y'_i),(t_i))=\sum_{i=0}^mt_i\delta_{y_i'}$. Then for every $0\leqslant i\leqslant m$,
    \[|\mu(V_{y_i})-\mu_0(f_{y_i})|=|t_i-c_i|<\epsilon.\]
    Hence, $\mu\in W$. Conversely, if $\mu=\sum_{i=0}^m c_i'\delta_{y_i'}\in W$ with $\tau(y_0')<\cdots<\tau(y_m')$, $\tilde\rho(\mu)\in U$. Since for $0\leqslant i\leqslant m$, $|\mu(V_{y_i})-c_i|<\epsilon<c_i$, so $\supp(\mu)\cap U_i$ has exactly one point, and for $0\leqslant i\leqslant l$, $\tau(y_i')=i$, hence $y_i'$ is exactly the point of $\supp(\mu)\cap U_i$. Hence, $\frac{\sum_{i=0}^m y_i'}{m+1}\in Q$. And for every $0\leqslant i\leqslant l$, $|c_i'-c_i|<\epsilon$, hence $(t_i')_i\in R$. Therefore,
    \[\phi(\mu)=(\frac{\sum_{i=0}^m y_i'}{m+1},(c'_i)_i)\in (Q\cap center(m,\Delta))\times R.\]
\end{proof}

\begin{prop}\label{locally exists universal module for simp complex}
    Let $\calG$ be a second countable étale groupoid, $A$ be a separable $\calG$-\cst-algebra. Let $Y$ be a second countable locally compact Hausdorff proper étale $\calG$-space, $(Y,\Delta)$ be a proper $\calG$-compact typed $\calG$-simplicial complex, $Z=|\Delta|$, $U_{\Delta_Z}$ be a $\calG$-invariant open neighborhood of the diagonal $\Delta_Z$ in $Z\times_{\calG\units}Z$. Then there exists a $\calG$-cover $(W_i,\calH_i)_{i=1}^n$ of $|\Delta|$, such that $\cup_{i=1}^n \calG(\overline{W_i}\times_{\calG\units}\overline{W_i})\subseteq U_{\Delta_Z}$, and for every $1\leqslant i\leqslant n$, there exists a stable non-degenerate $\calH_i$-equivariant representation
    \[\pi_i:C_0(W_i)\ra \calL(L^2(\calH_i,A_i)^\infty)\]
    where $A_i=(A|_{\calH_i})_{\langle C_i\rangle}$, $C_i=\tilde\rho(W_i)$ is $\calH_i$-invariant closed subset of $\calH_i\units$.
\end{prop}
\begin{proof}
    It suffices to prove that, for any $\mu_0\in |\Delta|$ and an open neighborhood $W_0$ of $\mu_0$, let $F=\{\gamma\in \calG:\gamma \mu_0=\mu_0\}$ be the stabilizer of $\mu_0$, there exists an open neighborhood $W_1$ of $\mu_0$ in $W_0$, an open neighborhood $U$ of $\tilde\rho(\mu_0)$, an action of $F$ on $U$ and a proper open subgroupoid $\calH\cong F\ltimes U$, such that
    \begin{enumerate}
        \item $W_1$ is $\calH$-invariant, and $\gamma W_1\cap W_1=\emptyset$ for any $\gamma\in \calG\setminus\calH$;
        \item there exists a stable $W_1,\calH,A|_{\calH}$-module in form of $(L^2(\calH,(A|_\calH)_{\langle C\rangle})^\infty,\pi)$, where $C=\tilde\rho(V)$ is $\calH$-invariant closed subset of $U=\calH\units$.
    \end{enumerate}

    By proposition \ref{locally symmetric subbase}, $\ca S_{ls}'(\mu_0)$ is a fundamental system of neighborhoods of $\mu_0$ in $|\Delta|$. So there exists a locally symmetric compatible datum $((V_y)_{y\in \supp(\mu_0)},V,U,\epsilon)$ of $\mu_0$, and for which $\calH'$ is an associated proper open subgroupoids with $\calH'\cong F_{\mu_0}\ltimes V$ such that $W_1=W((V_y)_{y\in \supp(\mu_0)},V,U,\epsilon)\subseteq W_0$, and by proposition \ref{property of locally symmetric subbase}, $W_1$ is $\calH=(\calH')^U_U\cong F\ltimes U$-invariant, and for any $\gamma\in \calG\setminus \calH$, $\gamma W_1\cap W_1=\emptyset$.

    Assume that $m=\#\supp(\mu_0)-1$, let $T=W_1\cap(|\Delta^m|\setminus|\Delta^{m-1}|)$, then $T$ is $\calH$-invariant closed subset of $W_1$ and $\tilde\rho(W_1)=\tilde\rho(T)$. By proposition \ref{property of locally symmetric subbase}, there exists an $\calH$-invariant clopen subset $Q$ of $A^m\cap \tilde\rho\inv(U_i)$ and an open subset $R$ of $\sigma_m$, such that $\tilde\rho|_{Q}$ is homeomorphism onto $U$, and
    \[\phi(T)=(Q\cap center(m,\Delta))\times R,\]
    Since $\tilde\rho|_Q$ is a homeomorphism onto $U$, $\tilde\rho(Q\cap center(m,\Delta))$ is an $F$-invariant closed subset of $U$, and hence $C=\tilde\rho(W_1)$ is an $F$-invariant closed subset of $U$.

    As seen in example \ref{canonical module for closed inclusion}, the structure map of $(A|_{\calH})_{\lgl C\rgl}$ is a $C_0(C)$-\cst-algebra. Then since $L^2(\calH,A_{\lgl C\rgl})$ is a $\calH$-Hilbert $(A|_\calH)_{\lgl C\rgl}$-module, the structure map
    \[\phi: C_0(C)\ra Z\calL(L^2(\calH,(A|_\calH)_{\lgl C\rgl}))\]
    is a non-degenerate $\calH$-equivariant representation. After identifying $Q\cap center(m,\Delta)$ with $C$ through $\tilde\rho$, the canonical composition
    \[\pi_1: C_0(Q\cap center(m,\Delta))\cong C_0(C)\ra \calL(L^2(\calH,(A|_\calH)_{\lgl C\rgl}))\]
    is non-degenerate and $\calH$-equivariant.
    
    And $R$ is an open of $\sigma_m$, there exists a stable non-degenerate representation
    \[\pi_2:C_0(R)\ra B(H),\]
    where $H$ is an infinitely dimensional separable Hilbert space. After identifying $L^2(\calH,A_{\lgl C\rgl})\otimes H$ with $L^2(\calH,A_{\langle C\rgl})^\infty$, we can define an $\calH$-equivariant representation
    \[\pi_3:=\pi_1\otimes\pi_2: C_0((Q\cap center(m,\Delta))\times R)\ra \calL(L^2(\calH,(A|_\calH)_{\lgl C\rgl})^\infty).\]
    Now define $g:\phi(T)\ra W_1, t\mapsto \phi\inv(t)$, then $g$ is an $\calH$-equivariant closed inclusion, we define
    \[\pi: C_0(W_1)\xrightarrow{g^*} C_0(\phi(T))\xrightarrow{\pi_3} \calL(L^2(\calH,(A|_\calH)_{\lgl C\rgl})^\infty),\]
    this is the stable $W_1,\calH,A|_{\calH}$-module we want.
\end{proof}

\begin{thm}\label{universal module exists for simp complex}
    Let $\calG$ be a second countable étale groupoid, $A$ be a separable $\calG$-\cst-algebra, $Y$ be a second countable locally compact Hausdorff proper étale $\calG$-space, $(Y,\Delta)$ be a proper $\calG$-compact $\calG$-simplicial complex, $Z=|\Delta|$. Then for any $\calG$-invariant open neighborhood $U$ of the diagonal $\Delta_Z$ in $Z\times_{\calG\units}Z$, there exists a universal $|\Delta|,\calG,A$-module $(E,\pi)$, such that for any $Z,\calG,A$-module $(E',\pi')$, there exists a $\calG$-equivariant controlled isometry $V\in \Hom_A(E',E)$ such that $\supp_{\pi',\pi}(V)\subseteq U$.
\end{thm}
\begin{proof}
    By proposition \ref{barycentric subdivision keeps realization}, without loss of generality we can assume that $(Y,\Delta)$ is typed. Then this is the consequence of proposition \ref{universal modules exist} and proposition \ref{locally exists universal module for simp complex}.
\end{proof}

\begin{rem}
It is easy to see that we can remove the cocompactness and require for weakly controlled equivariant isometries, but we do not pursue it.
\end{rem}

\addcontentsline{toc}{section}{References}
\bibliography{bibref}{}
\bibliographystyle{acm}

\end{document}